\documentclass[11pt,british,reqno]{amsart}

\usepackage[T1]{fontenc}

\usepackage{babel,eucal,url,amssymb,stmaryrd,booktabs,enumerate,amscd,paralist,cite,color}

\markleft{\sc Francisco Martín Cabrera}
\theoremstyle{plain}
\newtheorem{lemma}{Lemma}[section]

\newtheorem{proposition}[lemma]{Proposition}
\newtheorem{corollary}[lemma]{Corollary}
\newtheorem{theorem}[lemma]{Theorem}
\newtheorem{remark}[lemma]{Remark}
\newtheorem{example}[lemma]{Example}

\DeclareMathOperator{\Ric}{Ric}

\newcommand{\Lie}[1]{\operatorname{\textsl{#1}}}
\newcommand{\lie}[1]{\operatorname{\mathfrak{#1}}}

\newcommand{\SO}{\Lie{SO}}

\newcommand{\Un}{\Lie{U}}
\newcommand{\un}{\lie{u}}
\newcommand{\Gtwo}{\ifmmode{{\rm G}_2}\else{${\rm G}_2$}\fi}

 \newcommand{\cyclic}{\mathop{\kern0.9ex{{+}\kern-2.2ex\raise-.28ex\hbox{\Large\hbox
 {$\circlearrowright$}}}}}

\newcommand{\real}[1]{\left\llbracket #1 \right\rrbracket}

\newcommand{\lcf}{\lbrack\!\lbrack}
\newcommand{\rcf}{\rbrack\!\rbrack}
\newcommand{\talt}{\tilde{\mathbf a}}

\newcommand{\Ricac}{\Ric^{*}} 

\def\sideremark#1{\ifvmode\leavevmode\fi\vadjust{\vbox to0pt{\vss
 \hbox to 0pt{\hskip\hsize\hskip1em
 \vbox{\hsize2.5cm\tiny\raggedright\pretolerance10000
 \noindent #1\hfill}\hss}\vbox to8pt{\vfil}\vss}}}%

\newfont{\eusm}{eusm10 scaled \magstep1}
\newfont{\eusmiii}{eusm10 scaled \magstep3}

\newcommand{\comp}{\makebox[7pt]{\raisebox{1.5pt}{\tiny $\circ$}}}

\title{Harmonic almost contact metric structures revisited}

\author{Francisco~Mart\'\i n~Cabrera}
\address[Francisco~Mart\'\i n~Cabrera]{Departamento de Matemáticas, Estadística e Investigación   Operativa \\
  Universidad de La Laguna\\ 38200 La Laguna, Tenerife, Spain}
\email{fmartin@ull.edu.es}

\begin{document}

\maketitle

\markboth{\protect\small \sc  francisco martín cabrera}
         {\protect\small \sc harmonic almost contact metric structures revisited}

\begin{abstract}{\indent}
The study of harmonicity for almost contact metric
structures was  initiated by Vergara-Díaz and Wood in \cite{VergaraWood} and continued by González-Dávila and the present author in \cite{GDMC4}. By using the
intrinsic torsion and some restriction on the type of almost contact metric structure, in \cite{GDMC4}  harmonic structures are characterised by  showing conditions relating
harmonicity and classes of almost contact metric structures. Here we do this in a more general context. 
Moreover, the harmonicity of almost contact metric
structures  as a map is also  done in such a general context. Finally, some remarks on the classification of almost contact metric structures are exposed.

\vspace{2mm}

  \vspace{2mm}

\end{abstract}


\section{Introduction}{\indent} \setcounter{equation}{0}
For an oriented Riemannian manifold $M$ of dimension $n$, given
 a Lie subgroup $G$ of $\SO(n)$, $M$ is said to  be equipped with a
\emph{$G$-structure}, if there exists a subbundle $\mathcal{G}(M)$
of the oriented orthonormal frame bundle $\mathcal{SO}(M)$ with
structural group $G$. For a fixed $G$, a natural question arises,
'which are the best $G$-structures on M?'. An approach to answer
this question is based on the notion of the energy  of a
$G$-structure which is a particular case of the energy of a map
between Riemannian manifolds. Such a functional has been widely
studied by diverse authors \cite{EeLe1,EeLe2,Ur}. The corresponding
critical points are called \emph{harmonic maps} and have been
characterised by Eells and Sampson \cite{EeSa}.

For principal $G$-bundles $Q \to M$ over a Riemannian manifold,
Wood in \cite{Wood2} considers global sections $\sigma :  M
\to Q/H$ of the quotient bundle $\pi : Q/H \to M$, where $H$ is a
Lie subgroup of $G$ such that $G/H$ is reductive. Such 
sections are in one-to-one correspondence with the $H$-reductions
of the $G$-bundle $Q \to M$. Likewise, a connection on $Q \to M$
and a $G$-invariant metric on $G/H$ are fixed. Thus, $Q/H$ can be
equipped in a natural way with a metric defined by using the
metrics on $M$ and $G/H$.  In such conditions, Wood regards
{\it harmonic sections} as generalisations of  harmonic maps from $M$
into $Q/H$, deriving the corresponding  {\rm harmonic sections equations}.

The situation described in the previous paragraph arises when the
Riemannian manifold $M$ is equipped with some $G$-structure. Thus, in \cite{GDMC} it is  considered $G$-structures defined on an oriented Riemannian
manifold $M$ of dimension $n$, where $\Lie{G}$ is  a closed and
connected subgroup of $\SO(n)$. Since the existence of a
$G$-structure on $M$ is equivalent to the existence of a global
section $\sigma : M \to \mathcal{SO}(M)/G$ of the quotient bundle,
the energy  of a $G$-structure is defined as the energy of the
corresponding map $\sigma$.  Such an energy functional is essentially
determined by $\tfrac12 \int_{M}\| \xi^G\|^{2} \, dv
$, where  $\xi^G$ denotes the intrinsic torsion
of the $G$-structure. As a consequence, the notion of \emph{harmonic $G$-structure},
introduced by Wood in \cite{Wood2}, is given in terms of the
intrinsic torsion in \cite{GDMC}. This analysis has made  possible to go further in
the study of relations between harmonicity and classes of
$G$-structures. Thus, the  study of harmonic almost Hermitian
structures initiated in \cite{Wood1,Wood2} is enriched in
\cite{GDMC} with additional results.


Our purpose in the present work is going on the study of
harmonicity for almost contact metric structures initiated by
Vergara-Díaz and Wood in  \cite{VergaraWood} and continued by González-Dávila and the present author in \cite{GDMC4}.  Almost contact
metric structures can be seen as $\Lie{U}(n)$-structures defined
on  manifolds of dimension $2n+1$. In \cite{GDMC4},  by using the
intrinsic torsion and some restriction on the type of almost contact metric structure, type $\mathcal{C}_1 \oplus \dots \oplus \mathcal{C}_{10}$,   harmonic structures are characterised by  showing conditions relating
harmonicity and classes of almost contact metric structures. Here we do this for   the general  type  $\mathcal{C}_1 \oplus \dots \oplus \mathcal{C}_{12}$. 
Moreover, the harmonicity of the
structure  as a map is analyzed   in such a general context.

In Section \ref{harmonicalmostcontact}, some characterization of  harmonic almost contact metric structures  is firstly recalled in Theorem \ref{characharmalmcontact}. Then  we
show conditions relating harmonicity and   Chinea and
González-Dávila's  classes \cite{ChineaGonzalezDavila} of almost
contact metric structures.  The harmonicity  of such structures as a map of $M$ into $\mathcal{SO}(M)/(\Lie{U}(n)
\times 1)$ is also studied in Subsection \ref{subsfiveone}. Section \ref{harmonicalmostcontact} ends by describing situations where the Reeb vector field is harmonic as unit vector field.

As a relevant remark, we point out the  rôle played by the
identities given in Section \ref{foursection}.
They are consequences of the trivial identities $d^2 F=0$ and $d^2 \eta=0$, where $F$
is the fundamental two-form of the almost contact metric structure and $\eta$ is the one-form metrically equivalent to the Reeb vector field $\zeta$
(see Section \ref{sect:almcont}). Such identities are deduced by
firstly expressing $d^2 F$ and $d^2 \eta$ in terms of the intrinsic torsion and
the minimal connection, and then extracting certain
$\Lie{U}(n)$-components. Analog identities for almost Hermitian
structures were deduced  in \cite{FMCAS,FMC7}. In the proofs of some theorems in Section
\ref{harmonicalmostcontact}, the use of these
identities beside the harmonicity criteria is fundamental.

Finally, as another application of the identities in Section \ref{foursection}, we will derived  some results relative to  the   classification of almost contact metric structures. In \cite{FMC6}, results in such a direction have been already displayed. Here we   derive another results with  the same regards.  In fact, the non-existence in a proper way  of certain  classes is proved.

\section{Preliminaries}{\indent} \setcounter{equation}{0}
On an $n$-dimensional oriented Riemannian manifold $(M ,\langle
\cdot , \cdot \rangle),$ we consider the bundle
$\pi_{\Lie{SO}(n)} : \mathcal{SO}(M) \to M$ of the oriented
orthonormal frames with respect to the metric $\langle \cdot,
\cdot \rangle.$ Given a closed and connected subgroup $G$ of
$\SO(n),$ a {\it $G$-structure} on $(M ,\langle \cdot , \cdot \rangle)$
is a reduction $\mathcal{G}(M)\subset \mathcal{SO}(M)$ to $G.$ In
the present Section we briefly recall some notions relative to
$G$-structures (see \cite{GDMC, GDMC4, Wood2} for  more details).

Let ${\mathcal S\mathcal O}(M)/G$ be the orbit space under the
action of $G$ on ${\mathcal S\mathcal O}(M)$ on the right. Then  $\pi_{G}:{\mathcal
S\mathcal O}(M)\to {\mathcal S\mathcal O}(M)/G$ is a principal
$G$-bundle and  $\pi_{SO(n)} = \pi\comp \pi_{G},$ where
$\pi: {\mathcal S\mathcal O}(M)/G\to M$ is a bundle with
fibre $SO(n)/G$.
 The map
$\sigma:M\to {\mathcal S\mathcal O}(M)/G$ given by $\sigma(m) =
\pi_{G}(p),$ for all $p\in {\mathcal G}(M)$ with $\pi_{SO(n)}(p) =
m,$ is a smooth section. Thus    one has  a one-to-one
correspondence between the totally of $G$-structures and the
manifold $\Gamma^{\infty}({\mathcal S\mathcal O}(M)/G)$ of all
global sections of ${\mathcal S\mathcal O}(M)/G.$ Hence 
we will also denote by $\sigma$ the $G$-structure determined by
a section $\sigma$.

The reduced subbundle $\mathcal{G}(M)$ gives rise  to
express the bundle of endomorphisms $\mbox{End}(\mbox{T} M)$ 
on the tangent bundle  as the associated vector
bundle $\mathcal{G}(M) \times_{G} \mbox{End}(\mathbb{R}^n)$. We
restrict our attention on the subbundle $\lie{so}(M)$ of $\mbox{End}
(\mbox{T} M )$ of skew-symmetric endomorphisms $\varphi_m$, for all
$m \in M$, i.e. $\langle \varphi_m X , Y \rangle= -\langle \varphi_m
Y , X \rangle$. Note that this subbundle $\lie{so}(M)$ is expressed
as $ \lie{so}(M) = \mathcal{SO}(M) \times_{\Lie{SO}(n)} \lie{so}(n)
= \mathcal{G}(M) \times_{G} \lie{so}(n) $. Furthermore, because
$\lie{so}(n)$ is decomposed into the $G$-modules $\lie{g},$ the Lie
algebra of $G,$ and its  orthogonal complement $\lie{m}$  with respect to the natural extension  to  endomorphisms of the  Euclidean metric on $\mathbb{R}^n$, 
the bundle $\lie{so}(M)$ is also decomposed into $\lie{so}(M) =
\lie{g}_{\sigma} \oplus \lie{m}_{\sigma}$, where $\lie{g}_{\sigma} =
\mathcal{G}(M) \times_G \lie{g}$ and $\lie{m}_{\sigma} =
\mathcal{G}(M) \times_G \lie{m}$.

Under the conditions above fixed, if $M$ is equipped with a
$G$-structure, then there always  exists a $G$-connection
$\widetilde{\nabla}$ defined on $M$. Doing the difference
$\widetilde{\xi}_X = \widetilde{\nabla}_X - \nabla_X$, where
$\nabla_X$ is the Levi-Civita connection of $\langle \cdot , \cdot
\rangle$, a tensor $\widetilde{\xi}_X \in \lie{so}(M)$ is
obtained.  Decomposing $\widetilde{\xi}_X = ( \widetilde{\xi}_X
)_{\lie{g}_{\sigma}} + ( \widetilde{\xi}_X )_{\lie{m}_{\sigma}}$,
$( \widetilde{\xi}_X )_{\lie{g}_{\sigma}} \in \lie{g}_{\sigma}$
and $( \widetilde{\xi}_X )_{\lie{m}_{\sigma}} \in
\lie{m}_{\sigma}$, a new $G$-connection $\nabla^G$, defined by
$\nabla^G_X = \widetilde{\nabla}_X - (\tilde{\xi}_X
)_{\lie{g}_{\sigma}}$, can be considered. Because the difference
between two $G$-connections must be in $\lie{g}_{\sigma}$,
$\nabla^G$ is the unique $G$-connection on $M$ such that its
torsion $\xi^G_X = ( \widetilde{\xi}_X
)_{\lie{m}_{\sigma}} = \nabla^{G}_X - \nabla_X$ is in
$\lie{m}_{\sigma}$. $\nabla^G$ is called the {\it minimal
connection} and $\xi^G$ is referred to as the {\it intrinsic
torsion} of the $G$-structure $\sigma$
\cite{CleytonSwann:torsion}. 


{\rm  A natural way of classifying $G$-structures arises by
decomposing the space $\mathcal W = \mbox{T}^* M \otimes
\lie{m}_{\sigma}$ of possible intrinsic torsion into irreducible
$G$-modules. This was initiated by Gray and Hervella
\cite{Gray-H:16} for almost Hermitian structures. In this
particular case, $G = \Lie{U}(n)$, the dimension of the manifold is $2n$ and the space $\mathcal W$ is
decomposed into four irreducible $\Lie{U}(n)$-modules. Therefore,
$2^4=16$ classes of almost Hermitian structures were obtained. }
\vspace{1mm}

Along  the present paper, we will consider the natural extension
of the metric $\langle \cdot , \cdot \rangle$ to $(r,s)$-tensors
on $M$ defined by
\begin{equation} \label{extendedmetric}
\langle \Psi,\Phi \rangle = \Psi^{i_{1}\dots i_{r}}_{j_{1}\dots
j_{s}} \Phi^{i_{1}\dots
 i_{r}}_{j_{1}\dots j_{s}},
  \end{equation}
where the summation convention is used and  $\Psi^{i_{1}\dots
i_{r}}_{j_{1}\dots j_{s}}$ and $\Phi^{i_{1}\dots i_{r}}_{j_{1}\dots
j_{s}}$ are the components of the $(r,s)$-tensors  $\Psi,\Phi\in
\mbox{T}_{s\, m}^{r} M$, with respect to an orthonormal frame over
$m \in M$. Likewise, we will make reiterated use of the {\it musical
isomorphisms} $\flat : \mbox{\rm T}M \to \mbox{\rm T}^* M$ and
$\sharp : \mbox{\rm T}^* M \to \mbox{\rm T} M$, induced by the
metric $\langle \cdot , \cdot \rangle$ on $M$, defined respectively
by $X^{\flat} = \langle X , \cdot \rangle$ and $\langle
\theta^{\sharp} , \cdot \rangle = \theta $.

 \vspace{1mm}

If $\omega$ is  the connection one-form associated to $\nabla$,
then $\mbox{T} \mathcal{SO}(M) = \ker \pi_{\Lie{SO}(n)\ast} \oplus
\ker \omega$. Now considering  the projection $\pi_{G} \colon
\mathcal{SO}(M) \to \mathcal{SO}(M)/G$,   the tangent bundle of
$\mathcal{SO}(M)/G$ is decomposed into $\mbox{T} \mathcal{SO}(M)/G
= \mathcal{V} \oplus \mathcal{H}$, where $\mathcal{V} =
\pi_{G\ast} (\ker \pi_{\Lie{SO}(n)\ast} )$ and $\mathcal{H} =
\pi_{G\ast} (\ker \omega )$.
 For the projection
$\pi  :  \mathcal{SO}(M)/G \to M$, $\pi(pG) =
\pi_{\Lie{SO}(n)} (p)$,  the {\it vertical} and {\it horizontal}
distributions $\mathcal{V}$ and $\mathcal{H}$ are such that
$\pi_{\ast} \mathcal{V} =0$ and $\pi_{\ast} \mathcal{H} =
\mbox{\rm T}M$. Moreover, it is  considered
the bundle $\pi^* \lie{so}(M)$ on $\mathcal{SO}(M)/G$  consisting
of those pairs $(pG, \breve{\varphi}_m)$, where $\pi(pG)=m$ and
 $\breve{\varphi}_m \in \lie{so}(M)_m$.
  Alternatively, $\pi^* \lie{so}(M)$ is described as the bundle
$\pi^* \lie{so}(M)  = \mathcal{SO}(M) \times_G \lie{so}(n)=
\lie{g}_{\mathcal{SO}(M)} \oplus \lie{m}_{\mathcal{SO}(M)}$, where
$\lie{g}_{\mathcal{SO}(M)}= \mathcal{SO}(M) \times_G \lie{g} $ and
 $\lie{m}_{\mathcal{SO}(M)}= \mathcal{SO}(M) \times_G \lie{m}$.
A metric on each fibre in $\pi^* \lie{so}(M)$ is defined by
$
\langle (pG , \breve{\varphi}_m) , (pG , \breve{\psi}_m) \rangle =
\langle \breve{\varphi}_m , \breve{\psi}_m \rangle,
$
where $\langle \cdot , \cdot \rangle$ in the right side is the
metric on $(1,1)$-tensors  on $M$ given by
\eqref{extendedmetric}. With respect to this metric,  the
decomposition $\pi^* \lie{so}(M) = \lie{g}_{\mathcal{SO}(M)}
\oplus \lie{m}_{\mathcal{SO}(M)}$ is orthogonal.

There is a canonical isomorphism between $\mathcal{V}$ and  $\lie{m}_{\mathcal{SO}(M) }$. In fact,  elements in
$\lie{m}_{\mathcal{SO}(M)}$ can be seen as pairs $(pG ,
\breve{\varphi}_m)$ such that if $\breve{\varphi}_m$ is expressed
with respect to  $p$, then  it is obtained a matrix $( a_{ji}) \in
\lie{m}$.  For all $a \in \lie{m}$, we have the fundamental vector
field $a^*$ on $\mathcal{SO}(M)$ given by
$
a^*_p = \tfrac{d}{dt}_{|t=0} p . \exp t a \in \ker
\pi_{\Lie{SO}(n)*p} \subseteq \mbox{T}_p \mathcal{SO}(M).
$
Any vector in $\mathcal{V}_{pG}$ is given by $\pi_{G*p} (a^*_p)$,
for some $a =(a_{ji}) \in \lie{m}$. The isomorphism $\phi_{|
\mathcal{V}_{pG}} \colon \mathcal{V}_{pG} \to \left(
\lie{m}_{\mathcal{SO}(M)} \right)_{pG}$  is defined by
$
\phi_{| \mathcal{V}_{pG}} ( \pi_{G*p} (a^*_p)) = (pG, a_{ji} \,
p(u_i)^{\flat} \otimes p(u_j)).
$
Next it is extended the map $\phi_{|\mathcal{V}} : \mathcal{V} \to
\lie{m}_{\mathcal{SO}(M)}$ to $\phi : \mbox{T} \,
\mathcal{SO}(M)/G \to \lie{m}_{\mathcal{SO}(M)}$ by saying that
$\phi (A) =0$, for all $A \in \mathcal{H}$,  and $\phi (V) =
\phi_{|\mathcal{V}}(V)$, for all $V \in \mathcal{V}$. This is used
to define a metric $\langle \cdot , \cdot
\rangle_{\mathcal{SO}(M)/G}$ on $\mathcal{SO}(M)/G$ by
\begin{equation} \label{metricquo}
\langle A , B \rangle_{\mathcal{SO}(M)/G} = \langle \pi_{\ast} A ,
\pi_{\ast} B \rangle + \langle \phi (A) , \phi (B) \rangle.
\end{equation}
For this metric,  $\pi \, : \, \mathcal{SO}(M)/G \to
M$ is a Riemannian submersion with totally geodesic fibres (see
\cite{Vilms} and \cite[page 249]{Besse:Einstein}).

\vspace{1mm}

Now, we consider the set of all possible $G$-structures on a
closed and oriented Riemannian manifold $M$ which are compatible
with the metric $\langle \cdot , \cdot \rangle$. Such a set is
identified with the manifold $\Gamma^{\infty}(\mathcal{SO}(M)/G)$ of
all possible global sections $\sigma  :  M \to
\mathcal{SO}(M)/G$. With   respect to the metrics $\langle \cdot ,
\cdot \rangle$ and  $ \langle \cdot , \cdot
\rangle_{\mathcal{SO}(M)/G}$, the {\it energy} of $\sigma$ is the
integral
\begin{equation}\label{primera1}
 {\mathcal E}(\sigma)=\frac{\textstyle 1}{\textstyle2}\int_{M}\|\sigma_{\ast}\|^{2}dv,
\end{equation}
where $\|\sigma_{\ast}\|^{2}$ is the norm of the differential
$\sigma_{\ast}$ of $\sigma$ and $dv$ denotes the volume form on
$(M,\langle \cdot , \cdot \rangle)$. On the domain of a local
orthonormal frame field $\{e_1, \dots , e_n \}$ on $M$,
$\|\sigma_{*}\|^{2}$ can be locally expressed as
$\|\sigma_{*}\|^{2} = \langle
\sigma_{*}e_{i},\sigma_{*}e_{i}\rangle_{\mathcal{SO}(M)/G}$.
Furthermore,  using \eqref{metricquo}, from \eqref{primera1}  it is
obtained that the energy $\mathcal{E}(\sigma)$ of $\sigma$ is
given by
$
{\mathcal E}(\sigma) = \frac{n}{2} {\rm
Vol}(M) + \frac{1}{2}\int_{M}\|  \phi \,
\sigma_*\|^{2}dv.
$
The relevant part of this formula $B(\sigma)=\frac{
1}{ 2}\int_{M}\| \phi \, \sigma_*\|^{2}dv$  is called
the {\it total bending} of the $G$-structure $\sigma$. In
\cite{GDMC}, it was shown that $\phi \, \sigma_{\ast} = - \xi^G$.
Therefore,
$
B(\sigma)=\tfrac{1}{2}\int_{M}\| \xi^G\|^{2}
\, dv.
$
\vspace{1mm}

 To study critical points
  of the functional $\mathcal{E}$ on
$\Gamma^{\infty}(\mathcal{SO}(M)/G)$,   smooth variations
$\sigma_t \in \Gamma^{\infty} (\mathcal{SO}(M)/G)$ of
$\sigma=\sigma_0$ are considered. The corresponding {\it variation
fields} $m
 \to \varphi(m) = \frac{d}{dt}_{|t=0} \sigma_t(m)$  are sections of the
induced bundle $\sigma^* \mathcal{V}$ on $M$.
  Furthermore, by using $\phi$, we will have $\phi
  \mbox{pr}_2^{\sigma} \sigma^* \mathcal{V}   \cong \sigma^* \lie{m}_{\mathcal{SO}(M)}
\cong \lie{m}_{\sigma}$. Thus, the tangent space
$\mbox{T}_{\sigma} \Gamma^{\infty} (\mathcal{SO}(M)/G)$ is firstly
identified with the space $\Gamma^{\infty} (\sigma^* \mathcal V)$
of global sections of $\sigma^* \mathcal V$ \cite{Ur}. A second
identification is $\Gamma^{\infty} (\sigma^* \mathcal V ) \cong
\Gamma^{\infty} (\lie{m}_{\sigma})$ as global sections of
$\lie{m}_{\sigma}$.

In next theorem it is   considered the coderivative $d^*
\xi^G$ of the intrinsic torsion $\xi^G$, which is given by
$
d^* \xi^G_m = - (\nabla_{e_i} \xi^G)_{e_i}  = - (\nabla^{G}_{e_i}
\xi^G)_{e_i} - \xi^G_{\xi^G_{e_i} e_i} \in \lie{m}_{\sigma \, m} ,
$
where $\{ e_1, \dots , e_n \}$ is any orthonormal frame on $m\in
M$. 

\begin{theorem}[\cite{GDMC}]
If  $\,G$ is a closed and connected subgroup of $\Lie{SO}(n)$,
$(M,\langle \cdot , \cdot\rangle)$ a closed and oriented
Riemannian manifold and $\sigma$ a global section of
$\mathcal{SO}(M)/G$, then:
\begin{enumerate}
\item[{\rm (i)}] $($The first variation formula$)$. For the energy
functional $\mathcal{E}: \Gamma^{\infty}(\mathcal{SO}(M)/G) \to
\mathbb{R}$ and for all $\varphi \in \Gamma^{\infty}
(\lie{m}_{\sigma}) \cong \mbox{\rm T}_{\sigma} \Gamma^{\infty}
(\mathcal{SO}(M)/G)$, we have
$$
d \mathcal{E}_{\sigma} (\varphi) = - \int_M \langle \xi^G , \nabla
\varphi \rangle dv =  - \int_M \langle d^* \xi^G , \varphi \rangle
dv.
$$
\item[{\rm (ii)}] $($The second variation formula$)$.  The Hessian
form $({\rm Hess}\;{\mathcal E})_{\sigma}$ on $
\Gamma^{\infty}(\lie{m}_{\sigma} )$ is given by
\begin{eqnarray*}
({\rm Hess}\; {\mathcal E})_{\sigma}\varphi & = & \int_{M} \left(
\|\nabla  \varphi \|^{2}   - \tfrac12 \| [ \xi^G ,
\varphi]_{\lie{m}_{\sigma}}
  \|^2   +   \langle \nabla
\varphi , 2[ \xi^G , \varphi] -  [ \xi^G ,
\varphi]_{\lie{m}_{\sigma}} \rangle  \right) dv.
\end{eqnarray*}
\end{enumerate}
\end{theorem}
As a consequence of this Theorem the following notion is
introduced: for general Riemannian manifolds $(M,\langle \cdot
,\cdot \rangle)$ not necessarily closed and oriented,
 a $G$-structure $\sigma$ is said to be  {\em harmonic}, if it satisfies $d^* \xi^G=0$ or, equivalently,
 $(\nabla^G_{e_i} \xi)_{e_i} = -  \xi^G_{\xi^G_{e_i} e_i}$.

Given a $G$-structure $\sigma$ on a closed Riemannian manifold
$(M,\langle\cdot ,\cdot \rangle),$ the map $\sigma:(M, \langle
\cdot, \cdot \rangle) \mapsto (\mathcal{SO}(M)/G,
\langle\cdot,\cdot\rangle_{\mathcal{SO}(M)/G})$ is harmonic, i.e.
$\sigma$ is a critical point for the energy functional on
$\mathcal{C}^{\infty}(M,\mathcal{SO}(M)/G)$ if and only if its {\it
tension field} $\tau(\sigma)= \left( \nabla_{e_i}
\sigma_*\right)(e_i)$ vanishes \cite{Ur}. Here, $\nabla \sigma_{*}$
is defined by $\left( \nabla_X \sigma_* \right)(Y) =
\nabla^q_{\sigma_* X} \sigma_* Y - \sigma_* (\nabla_X Y),$ where
$\nabla^q$ denotes the induced connection by the Levi-Civita
connection $\nabla^q$ of the metric in $\mathcal{SO}(M)/G$.
According with \cite{GDMC,Wood2}, harmonic sections $\sigma$ are
characterised by the vanishing of the vertical component of
$\tau(\sigma)$ and the horizontal component of $\tau(\sigma)$ is
determined by the horizontal lift of the vector field metrically
equivalent to the one-form $\nu_{\sigma},$ defined by
$ \nu_{\sigma}(X) = \langle \xi^G_{e_i} ,
R_{e_i,X}\rangle$.
Hence one has the following
\begin{proposition}  \label{harmmap2} 
The map $\sigma:(M, \langle \cdot, \cdot \rangle) \mapsto
(\mathcal{SO}(M)/G, \langle\cdot,\cdot\rangle_{\mathcal{SO}(M)/G})$
is  a harmonic map if and only if $\sigma$ is a harmonic
$G$-structure such that $\nu_{\sigma}=0$.
\end{proposition}
\vspace{1mm}


Relevant types of  $G$-structures are those ones such
that $\xi^G$ is metrically equivalent to a
skew-symmetric three-form, i.e. $\xi^G_X Y = - \xi^G_Y X$. Next
we recall some facts satisfied by  them.

\begin{proposition}[\cite{GDMC}] \label{pro:skew}  For   a $G$-structure $\sigma$  such
that $\xi^G_X Y = - \xi^G_{Y} X$, we have:
\begin{enumerate}
  \item[{\rm (i)}]
 If $[\xi^G_X  , \xi^G_Y] \subseteq \lie{g}_\sigma$, for all $X,Y
 \in \mathfrak X (M)$,
  then
    $\langle R_{X,Y  \lie{m}_{\sigma}} X , Y \rangle  =  2\langle
\xi^G_{X} Y , \xi^G_X Y \rangle$. 
 \item[{\rm (ii)}] If $\sigma$ is a harmonic $G$-structure, then
$\sigma$ is also a harmonic map.
\end{enumerate}
\end{proposition}

In Section \ref{harmonicalmostcontact}, we will  study harmonicity
of almost contact metric structures. Such structures are examples of
$G$-structures defined by means of one or several $(r,s)$-tensor
fields $\Psi$ which are stabilised under the action of $G$, i.e. $g
\cdot \Psi= \Psi$, for all $g \in G$. Moreover, it will be possible to 
characterise the harmonicity of such $G$-structures by conditions
given in terms of those tensors $\Psi$. The \emph{connection
Laplacian} (or {\it rough Laplacian}) $\nabla^* \nabla \Psi$ will
play a relevant rôle  in such conditions. We recall that
$
\nabla^* \nabla \Psi = -  \left( \nabla^2 \Psi \right)_{e_i,e_i},
$
where $\{ e_1, \dots , e_n \}$ is an orthonormal  frame field and
$(\nabla^2\Psi)_{X,Y} = \nabla_X (\nabla_Y \Psi) -
\nabla_{\nabla_XY}\Psi$.  If a Riemannian manifold $(M , \langle
\cdot , \cdot \rangle )$ of dimension $n$ is equipped with a
$G$-structure, where $G \subseteq \Lie{SO}(n)$, and
 $\Psi$ is a $(r,s)$-tensor field on $M$ which is stabilised under the action of $G$, then
\begin{equation}\label{lapstaten}
\nabla^* \nabla \Psi =  \left( \nabla^{G}_{e_i}
\xi^{G}\right)_{e_i} \Psi + \xi^{G}_{\xi^{G}_{e_i}e_i} \Psi -
\xi^{G}_{e_i} (\xi^{G}_{e_i} \Psi).
\end{equation}
As a consequence, if the $G$-structure is harmonic, then  $ \nabla^* \nabla \Psi
= -\xi^{G}_{e_i} (\xi^{G}_{e_i} \Psi). $
\vspace{1mm}

\section{Almost contact metric structures}{\indent}
\label{sect:almcont} \setcounter{equation}{0}
 An {\it  almost contact
metric manifold} is a $2n+1$-dimensional Riemannian manifold
$(M,\langle \cdot, \cdot \rangle)$  equipped with a $(1,1)$-tensor
field $\varphi$ and a unit vector field $\zeta$,  called the {\it
Reeb vector field} of the structure,  such that
\[
\varphi^{2} = -I + \eta\otimes \zeta,\;\;\;\;\; \langle \varphi
X,\varphi Y\rangle = \langle X,Y\rangle - \eta(X)\eta(Y),
\]
where $\eta = \zeta^{\flat}$. Associated with such a structure the two-form $F= \langle \cdot, \varphi
\cdot\rangle$, called the {\em fundamental two-form}, is usually
considered. Using $F$ and $\eta$, $M$ can be oriented by  fixing a
constant multiple of $F^n \wedge \eta = F \wedge
\stackrel{(n)}{\dots}\wedge F \wedge \eta$ as volume form. Likewise,
the presence of an almost contact metric structure is equivalent to
say that $M$ is equipped with a $\Lie{U}(n) \times 1$-structure. It
is well known that $\Lie{U}(n)\times 1$ is a closed and connected
subgroup of $\SO(2n+1)$ and $\SO(2n+1)/ (\Lie{U}(n) \times 1)$ is
reductive. In this case, the cotangent space on each point
$\mbox{T}^{*}_m M$ is not irreducible under the action of the group
$\Lie{U}(n) \times 1$. In fact, $\mbox{T}^* M = \eta^{\perp} \oplus
\mathbb{R} \eta$ and
$
\lie{so}(2n+1) \cong \Lambda^{2} \mbox{T}^* M = \Lambda^2
\eta^{\perp} \oplus \eta^{\perp} \wedge \mathbb R \eta.
$

From now on, we will denote $X_{\zeta^{\perp}} = X - \eta(X) \zeta$,
for all $X  \in \mathfrak{X}(M)$. Since $\Lambda^2 \eta^{\perp} =
\lie{u}(n) \oplus \lie{u}(n)^{\perp}_{|\zeta^{\perp}} $, where
$\lie{u}(n)$ (resp., $\lie{u}(n)^{\perp}_{|\zeta^{\perp}} $)
consists of those two-forms $b$ such that $b(\varphi X , \varphi Y)
= b( X_{\zeta^{\perp}} ,  Y_{\zeta^{\perp}})$ (resp., $b(\varphi X ,
\varphi Y) = - b( X_{\zeta^{\perp}} , Y_{\zeta^{\perp}})$), we have
$
\lie{so}(2n+1) = \lie{u}(n) \oplus \lie{u}(n)^{\perp},\;$ with  $\; \lie{u}(n)^{\perp} =\lie{u}(n)^{\perp}_{|\zeta^{\perp}}
 \oplus \eta^{\perp} \wedge
\mathbb R \eta.
$
Therefore, for the space $\mbox{T}^* M \otimes \lie{u}(n)^{\perp}$
of possible intrinsic $\Lie{U}(n) \times 1$-torsions,  we obtain
$$
\mbox{T}^* M \otimes  \lie{u}(n)^{\perp} = (\eta^{\perp} \otimes
\lie{u}(n)^{\perp}_{|\zeta^{\perp}}) \oplus (\eta \otimes
\lie{u}(n)^{\perp}_{|\zeta^{\perp}}) \oplus ( \eta^{\perp} \otimes
\eta^{\perp} \wedge   \eta  ) \oplus (\eta \otimes
  \eta^{\perp} \wedge   \eta).
$$

In  \cite{ChineaGonzalezDavila}
it is showed that $\mbox{T}^* M \otimes \lie{u}(n)^{\perp}$ is
decomposed into twelve irreducible $\Lie{U}(n)$-modules $\mathcal
C_1, \dots , \mathcal C_{12}$, where
\begin{align*}
&\eta^{\perp} \otimes \lie{u}(n)^{\perp}_{|\zeta^{\perp}}
  = 
\mathcal
 C_1 \oplus \mathcal C_2 \oplus \mathcal C_3 \oplus \mathcal C_4, &
 & \eta^{\perp} \otimes \eta^{\perp} \wedge  \eta  =  \mathcal C_5
\oplus \mathcal C_8 \oplus \mathcal C_9 \oplus \mathcal C_6 \oplus
\mathcal C_7  \oplus \mathcal C_{10}, \\
& \eta \otimes \lie{u}(n)^{\perp}_{|\zeta^{\perp}}
 =  \mathcal C_{11} , & & 
 \eta \otimes \eta^{\perp} \wedge   \eta  =  \mathcal
C_{12}.
\end{align*}
The modules  $\mathcal C_1, \dots , \mathcal C_4$ are isomorphic
to  the Gray and Hervella's $\Lie{U}(n)$-modules above mentioned. Furthermore, note that $\varphi$
restricted to $\zeta^{\perp}$ works as an almost complex structure
and, if one considers the $\Lie{U}(n)$-action on the bilinear
forms $\otimes^2 \eta^{\perp}$,  we have the decomposition
$$
\textstyle \otimes^2 \eta^{\perp} = \mathbb R \langle \cdot ,
\cdot \rangle_{|\zeta^{\perp}} \oplus \lie{su}(n)_s \oplus
\real{\sigma^{2,0}} \oplus \mathbb R F  \oplus \lie{su}(n)_a
\oplus \lie{u}(n)^{\perp}_{|\zeta^{\perp}}.
$$
The modules  $\lie{su}(n)_s$ (resp., $\lie{su}(n)_a$) consists of
Hermitian symmetric (resp., skew-symmetric) bilinear forms
orthogonal to $\langle \cdot , \cdot \rangle_{|\zeta^{\perp}}$
(resp., $F$),
 and  $\real{\sigma^{2,0}}$ (resp., $\lie{u}(n)^{\perp}_{|\zeta^{\perp}}
 $) is the space of  anti-Hermitian
symmetric (resp., skew-symmetric) bilinear forms. In relation with
the modules $\mathcal C_i$, from $\eta^{\perp} \otimes
\eta^{\perp} \wedge  \mathbb R \eta \cong   \otimes^2
\eta^{\perp}$, using the $\Lie{U}(n)$-map $\xi^{\Lie{U}(n)}
\to - \xi^{\Lie{U}(n)} \eta = \nabla \eta$, it is obtained
$$
\mathcal C_5 \cong  \mathbb R \langle \cdot , \cdot
\rangle_{|\zeta^{\perp}} , \quad \mathcal C_8 \cong \lie{su}(n)_s,
\quad \mathcal C_9 \cong \real{\sigma^{2,0}}, \quad \mathcal C_6
\cong \mathbb R F, \quad \mathcal C_7 \cong  \lie{su}(n)_a, \quad
\mathcal C_{10} \cong \lie{u}(n)^{\perp}_{|\zeta^{\perp}}.
$$
In summary,  the space of possible intrinsic torsions
  $\mbox{T}^* M\otimes  \lie{u}(n)^{\perp}$ consists of those tensors
$\xi^{\Lie{U}(n)}$ such that
\begin{equation} \label{inttorcar}
  \varphi \xi^{\Lie{U}(n)}_X Y + \xi^{\Lie{U}(n)}_X \varphi Y =
  \eta (Y) \varphi \xi^{\Lie{U}(n)}_X \zeta +
\eta ( \xi^{\Lie{U}(n)}_X \varphi Y) \zeta
\end{equation}
and, under the action of $U(n)\times 1$, is decomposed into:
\begin{enumerate}
\item if $n=1$, $ \xi^{\Lie{U}(1)} \in \mbox{T}^* M \otimes
\un(1)^\perp = \mathcal C_{5} \oplus \mathcal C_{6} \oplus
\mathcal C_{9} \oplus \mathcal C_{12}$; \item if $n=2$, $
\xi^{\Lie{U}(2)} \in \mbox{T}^* M \otimes \un(2)^\perp = \mathcal
C_{2} \oplus \mathcal C_{4} \oplus \dots \oplus \mathcal C_{12}$;
\item if $n \geqslant 3$, $ \xi^{\Lie{U}(n)} \in \mbox{T}^* M
\otimes \un(n)^\perp =
  \mathcal C_{1}  \oplus \dots  \oplus \mathcal C_{12}$.
\end{enumerate}

Now, we recall how  some of these classes are referred to  by
diverse authors \cite{Bl,ChineaGonzalezDavila}:

 $\{ \xi^{\Lie{U}(n)}=0 \}=$ cosymplectic manifolds,
 $\mathcal C_1=$ nearly-K-cosymplectic manifolds,
 $\mathcal C_5=$ $\alpha$-Kenmotsu manifolds,
 $\mathcal C_6=$ $\alpha$-Sasakian manifolds,
  $\mathcal C_5\oplus \mathcal C_6=$ trans-Sasakian manifolds,
 $\mathcal C_2 \oplus \mathcal C_9=$ almost cosymplectic manifolds,
  $\mathcal C_6 \oplus \mathcal C_7=$ quasi-Sasakian manifolds,
 $\mathcal C_1 \oplus \mathcal C_2 \oplus \mathcal C_9 \oplus \mathcal C_{10}=$ quasi-K-cosymplectic manifolds,
 $\mathcal C_3 \oplus \mathcal C_4 \oplus \mathcal C_5 \oplus \mathcal C_{6}
  \oplus \mathcal C_7 \oplus \mathcal C_{8}=$ normal manifolds, $\mathcal C_{2}
  \oplus \mathcal C_6 \oplus \mathcal C_{9}=$ almost $a$-Sasakian manifolds,  etc.
  \vspace{1mm}

The minimal $\Un(n)$-connection is given by $\nabla^{\Lie{U}(n)} =
\nabla + \xi^{\Lie{U}(n)}$ with
\begin{equation*} 
  \xi^{\Lie{U}(n)}_X   =   -  \tfrac12  \varphi \circ \nabla_X \varphi +
\nabla_X \eta \otimes \zeta - \tfrac12 \eta \otimes  \nabla_X \zeta \\
 =  \tfrac12 (\nabla_X \varphi) \circ \varphi   + \tfrac12
\nabla_X \eta \otimes \zeta - \eta \otimes \nabla_X \zeta.
\end{equation*}
For sake of simplicity, we will write $\xi = \xi^{\Lie{U}(n)}$ in
the sequel. Likewise, $\xi_{(i)}$ will denote the component of
$\xi$  obtained by the $\Lie{U}(n)$-isomorphism  $(\nabla F)_{(i)}
=(-\xi F)_{(i)}
 \in \mathcal C_i \to \xi_{(i)}$. In this way,  classes  or types are referred to as   in
\cite{ChineaGonzalezDavila}.  

Certain $\Lie{U}(n)$-components of the Riemannian
curvature tensor $R$ of an almost contact metric manifold are determined
by  a Ricci type tensor $\Ricac$ associated to the structure, called the {\it $\ast$-Ricci tensor}. Such a tensor is defined by $\Ricac (X,Y) =
\langle R_{ e_i,X} \varphi e_i , \varphi Y \rangle$.
 In general, $\Ricac$ is not symmetric. However, since $\Ricac$  satisfies
 the identities
$
\Ricac(\varphi X, \varphi Y) =
 \Ricac( Y_{\zeta^{\perp}}, X_{\zeta^{\perp}})$, $\; \;
 \Ricac(X, \zeta ) = 0,
$
it can be claimed that
$$
\Ricac \in \mathbb{R} \langle \cdot , \cdot \rangle \oplus
\lie{su}(n)_s \oplus \lie{u}(n)^{\perp}_{|\zeta^{\perp}}
 \oplus
\eta^{\perp}_d \subseteq \mathbb{R} \langle \cdot , \cdot \rangle
\oplus  \lie{su}(n)_s \oplus
 \eta \odot
\eta^{\perp} \oplus \lie{u}(n)^{\perp}_{|\zeta^{\perp}}
 \oplus \eta \wedge \eta^{\perp},
$$
where $ \eta^{\perp}_d= \{ 2 \eta \odot \alpha +  \eta \wedge
\alpha \, | \, \alpha \in \eta^{\perp} \} \cong \eta^{\perp}$ and
we follow the convention $a \odot b = \tfrac12 (a\otimes b + b
\otimes a)$.

The skew-symmetric part $\Ricac_{\rm alt}$ of
 $\Ricac$ will play a special r{\^o}le. Relative to  $\Ricac_{\rm alt}$,  the following result was already given in \cite{GDMC4}. However,  there are some summands missing there.
\begin{lemma}\label{astricciacm1}
  Let $(M,\langle \cdot ,\cdot \rangle, \varphi , \zeta)$ be a $2n+1$-dimensional
  almost contact metric manifold.  Then  the  $*$-Ricci tensor
satisfies  
    \begin{align*}
  \Ricac_{\mbox{\rm \footnotesize alt}}( X_{\zeta^{\perp}} , Y_{\zeta^{\perp}})  =
 &
  \langle  (\nabla^{\Lie{U}(n)}_{e_i} \xi)_{\varphi e_i} \varphi X_{\zeta^{\perp}} , Y_{\zeta^{\perp}} \rangle
  + \langle \xi_{\xi_{e_i} \varphi e_i} \varphi X_{\zeta^{\perp}} , Y_{\zeta^{\perp}}
  \rangle 
  \\
  &
  - (\xi_{e_i} \eta) \wedge (\xi_{\varphi e_i} \eta) \circ \varphi  ( X_{\zeta^{\perp}} , Y_{\zeta^{\perp}}),
   \\
 \Ricac( \zeta) =
 &
 - \varphi  (\nabla^{\Lie{U}(n)}_{e_i} \xi)_{\varphi e_i} \zeta 
  - \varphi    \xi_{\xi_{e_i} \varphi e_i} \zeta + \varphi  \xi_{e_i} \xi_{\varphi e_i} \zeta,
    \end{align*}
for all $X,Y \in \mathfrak{X}(M)$,  and $\langle  \Ricac( \zeta) , X \rangle  = \Ricac( \zeta , X)$. Furthermore,  if $n>1$,  we
have:
\begin{enumerate}
\item[{\rm (i)}] The restriction $\Ricac_{{\rm alt} |
\zeta^\perp}$ of $\Ricac_{\mbox{\rm \footnotesize alt}}$ to the
space $\zeta^\perp$ is in $\lie{u}(n)^{\perp}_{|\zeta^{\perp}}$
and determines a $\Lie{U}(n)$-component of the Weyl curvature
tensor $W$.
  \item[{\rm (ii)}] The vector field  $\Ricac( \zeta )$
   is in $\zeta^\perp$ and determines another
   $\Lie{U}(n)$-component of $W$.
\end{enumerate}
As a consequence, if $M$ is
conformally flat 
and  $n>1$, then $\Ricac_{{\rm alt}
| \zeta^\perp}=0$
 and $\Ricac( \zeta)=0$.
\end{lemma}
\begin{proof} The proof follows in the same way as in \cite{GDMC4}. However, we rewrite it because of  missing summands in identities there. 
The so-called Ricci formula \cite[p.~26]{Besse:Einstein} implies
  $
  - (R_{e_i,\varphi e_i} F)(X,Y) =
    \talt(\nabla^2F)_{e_i,\varphi e_i}(X,Y),
  $
  where $\talt \colon T^* M \otimes T^* M \otimes \Lambda^2 T^* M \to
  \Lambda^2 T^* M \otimes \Lambda^2 T^* M$ is the skewing
  mapping.
  On one hand, by making use of first Bianchi's identity, it is
  straightforward to see 
  $
  - (R_{e_i,\varphi e_i} F)(X,Y) = 4   \Ricac_{\mbox{\footnotesize alt}}( X , Y).
  $
On the other hand, it is relatively direct to check 
\begin{align*}
 \talt(\nabla^2F)_{e_i,\varphi
 e_i}(X,Y)  = &-  2 \langle \varphi (\nabla^{\Lie{U}(n)}_{e_i} \xi)_{\varphi e_i}
 X ,   Y \rangle + 2 \langle (\nabla^{\Lie{U}(n)}_{e_i} \xi)_{\varphi e_i}
\varphi  X ,  Y \rangle \\
& - 2 \langle \varphi \xi_{\xi_{e_i} \varphi e_i}
 X ,
  Y \rangle + 2 \langle \xi_{\xi_{e_i} \varphi e_i}
\varphi  X ,  Y \rangle\\
&
- 2 \langle \xi_{e_i} \xi_{\varphi e_i} X , \varphi Y \rangle 
+ 2 \langle \xi_{e_i} \xi_{\varphi e_i} Y , \varphi X \rangle. 
\end{align*}
In \cite{GDMC4} the last two summands are missing.
Now, using equation \eqref{inttorcar}, we will obtain the following right 
expression for $\Ricac_{\mbox{\rm \footnotesize alt}}( X , Y)$:
\begin{align} 
     \Ricac_{\mbox{\footnotesize alt}}( X , Y)  = &   \langle (\nabla^{\Lie{U}(n)}_{e_i} \xi)_{\varphi e_i}
\varphi  X ,  Y \rangle
  +  \langle \xi_{\xi_{e_i} \varphi e_i} \varphi  X ,  Y \rangle   +  \eta \odot ((\nabla^{\Lie{U}(n)}_{e_i} \xi)_{\varphi e_i} \eta) \circ\varphi (X,Y) \nonumber
  \\
&
   +  \eta \odot (\xi_{\xi_{e_i} \varphi e_i} \eta) \circ\varphi
   (X,Y)  - (\xi_{e_i} \eta) \wedge (\xi_{\varphi e_i} \eta) \circ \varphi (X,Y)   \label{otraricsac1}
    \\
& + \eta \wedge (\eta \circ \xi_{e_i} \circ \xi_{\varphi e_i} \circ \varphi)(X,Y). \nonumber 
\end{align}
 The tensors 
$\nabla^{\Lie{U}(n)}_{e_i} \xi$ and $\xi$ are  of the same
type because $\nabla^{\Lie{U}(n)}$ is a $\Lie{U}(n)$-connection.
 Now, by replacing
$X=X_{\zeta^{\perp}}$ and $Y=Y_{\zeta^{\perp}}$ 
\eqref{otraricsac1}, we will obtain the first required identity.
Likewise, by replacing $X=\zeta $ and $Y=X$ in equation
\eqref{otraricsac1}, the second required identity follows.

For  the proof for the  final assertions in the Lemma, see the one given in \cite{GDMC4}.
\end{proof}

The vector field $\xi_{e_i} \varphi
 e_i$ involved in  $\Ricac$ is given by
$
 \xi_{e_i} \varphi e_i = -\tfrac12 (d^* F)^\sharp   -\tfrac12 d^* F(\zeta) \zeta
 - \varphi \nabla_\zeta \zeta.
$
Thus, this vector field is contributed by the components of $\xi$ in
$\mathcal C_4$ and   $\mathcal C_6$. In fact, 
$$ 
\xi_{{(4)}e_i}
\varphi e_i = - \tfrac12 (d^* F)^{\sharp} -  \varphi \nabla_\zeta \zeta + \tfrac12 d^*
F(\zeta) \zeta, \qquad \xi_{{(6)}e_i} \varphi e_i = - d^* F(\zeta)
\zeta. 
$$
Likewise, the vector field $\xi_{e_i} e_i$ which
is involved in the harmonicity criteria  is given by
$
  \xi_{e_i} e_i  =  - \tfrac12 \varphi (d^* F)^{\sharp} - d^* \eta \;
\zeta  - \tfrac12  \nabla_{\zeta}\zeta.
$
Because  
\begin{equation} \label{leeac}
 \xi_{{(4)}e_i} e_i = -   \tfrac12 \varphi (d^* F)^{\sharp} +   \tfrac12
\nabla_{\zeta}\zeta, \qquad \xi_{{(5)}e_i} e_i = - d^* \eta \; \zeta, \qquad  \xi_{{(12)}e_i} e_i = -
 \nabla_{\zeta}\zeta,
 \end{equation}
  one has that $\xi_{e_i} e_i$ is  contributed by
$\mathcal C_4$, $\mathcal C_5$ and $\mathcal C_{12}$.
\vspace{1mm}

For a $2n$-dimensional  almost Hermitian manifold $(M,J,\langle \cdot, \cdot \rangle)$, where $J$ is the almost complex structure and  $\langle \cdot, \cdot \rangle$ is the metric, the \textit{Lee one-form} $\theta$  is defined by $\theta = - \frac{1}{n-1} J d^* \omega$, where $\omega = \langle \cdot , J \cdot \rangle$ is the Kähler two-form (see \cite{Gray-H:16}). The one-form $\theta$ determines the component usually denoted by  $\xi_{(4)}$ of the intrinsic torsion of the almost Hermitian structure. Such a component is given by
$
4 \xi_{(4)X} = X^\flat \otimes \theta^\sharp - \theta \otimes X-  JX^\flat \otimes J\theta^\sharp +  J \theta  \otimes J X.
$
 Note that $\sum_{i=1}^{2n} \xi_{e_i} e_i = \frac{n-1}{2} \theta^\sharp$.
 \vspace{2mm}
 
 In  the context of almost contact metric geometric, taking \eqref{leeac} into account,   the Lee form is defined by 
 $
(n-1) \theta  =  -   \varphi (d^* F)^{\sharp} +
\nabla_{\zeta}\eta, 
$
where $2n+1$ is the dimension of the almost contact metric  manifold. 
The  component $\xi_{(4)}$ is given by
$$
4 \xi_{(4)X} = X_{\zeta^\perp}^\flat \otimes \theta^\sharp - \theta \otimes X_{\zeta^\perp}-  \varphi X^\flat \otimes \varphi \theta^\sharp +  \varphi \theta  \otimes \varphi X.
$$

Likewise, for the vector field $ \varphi  \xi_{e_i} \xi_{\varphi e_i} \zeta$ involved in the expression for   $ \Ricac( \zeta)$  obtained above, we have the results given in next lemma which will be useful later. 
\begin{lemma} \label{xixizeta1}
 Denoting $\mathcal{A} = \mathcal{C}_1 \oplus \mathcal{C}_2$, $\mathcal{B} = \mathcal{C}_3 \oplus \mathcal{C}_4$, $\mathcal{C} = \mathcal{C}_5 \oplus \mathcal{C}_6 \oplus  \mathcal{C}_7 \oplus \mathcal{C}_8$,  $\mathcal{D} = \mathcal{C}_9 \oplus \mathcal{C}_{10}$, $\mathcal{E} = \mathcal{C}_{11} \oplus \mathcal{C}_{12}$ and being $\{e_1, \dots , e_{2n} , e_{2n+1} =\zeta\}$ an orthonormal basis for tangent vectors,  we have:
\begin{enumerate}
\item[$(i)$]
If the almost contact metric structure is of type $\mathcal{A} \oplus \mathcal{B}\oplus \mathcal{C} \oplus \mathcal{E}$, 
then 
$$
 \varphi \xi_{e_i} \xi_{\varphi e_i} \zeta =    \xi_{e_i} \xi_{ e_i} \zeta - \xi_{(11)\zeta} \xi_\zeta \zeta + \| \nabla \zeta \|^2 \zeta .
$$

\item[$(ii)$]
If the almost contact metric structure is of type $\mathcal{A} \oplus \mathcal{B}\oplus \mathcal{D} \oplus \mathcal{E}$, 
 then 
$$
 \varphi \xi_{e_i} \xi_{\varphi e_i} \zeta = -   \xi_{e_i} \xi_{ e_i} \zeta + \xi_{(11)\zeta} \xi_\zeta \zeta - \| \nabla \zeta \|^2 \zeta .
$$

\item[$(iii)$]
If the almost contact metric structure is of type $\mathcal{A} \oplus \mathcal{C} \oplus \mathcal{E}$ or   $\mathcal{B} \oplus \mathcal{D} \oplus \mathcal{E}$ or $\mathcal{C}_1 \oplus \mathcal{C} \oplus \mathcal{C}_9 \oplus \mathcal{E}$ or  $\mathcal{C}_3 \oplus \mathcal{C}_5 \oplus \mathcal{C}_6 \oplus \mathcal{E}$ or $\mathcal{A} \oplus \mathcal{B}\oplus   \mathcal{E}$,
 then 
$$
 	\hspace{1cm} \varphi \xi_{e_i} \xi_{\varphi e_i} \zeta =  0,         \qquad    \xi_{e_i} \xi_{ e_i} \zeta = \xi_\zeta \xi_\zeta \zeta =  \xi_{(11)\zeta} \xi_\zeta \zeta - \| \nabla_\zeta  \zeta \|^2 \zeta \quad \mbox{and} \quad
  \| \nabla  \zeta \| =   \| \nabla_\zeta  \zeta \|.  
$$
\end{enumerate}
\end{lemma}
\begin{proof}
For (i) (for (ii)),  by using  \eqref{inttorcar}, we have 
\begin{eqnarray*}
 \varphi \xi_{e_i} \xi_{\varphi e_i} \zeta & = & -   \xi_{e_i} \varphi \xi_{\varphi e_i} \zeta
+  \eta ( \xi_{\varphi e_i}\zeta ) \varphi  \xi_{e_i} \zeta +   \eta ( \xi_{e_i} \varphi  \xi_{\varphi e_i}\zeta )  \zeta. 
\end{eqnarray*}
Note that the second summand of the right side is equal to zero and the required  identity follows by considering the properties of the bilinear form $(\xi_\cdot  \eta)\cdot$ acting on $\zeta^\perp$, i.e. in the case  (i) (case (ii)),  it is a Hermitian (skew Hermitian) bilinear form on $\zeta^\perp$.
Finally, note that one has the identity 
$
 \xi_\zeta \xi_\zeta \zeta =  - \| \nabla_\zeta  \zeta \|^2 \zeta + \xi_{(11)\zeta} \xi_\zeta \zeta.
$
\vspace{1mm}

For the remaining particular cases, if the type is $\mathcal{A} \oplus \mathcal{C} \oplus \mathcal{E}$ or  $\mathcal{B} \oplus \mathcal{D} \oplus \mathcal{E}$,  we have $ \sum_{i=1}^{n}  \xi_{\varphi e_i} \xi_{\varphi e_i} \zeta = - \sum_{i=1}^{n}  \xi_{ e_i} \xi_{ e_i} \zeta$. Therefore,
 \begin{gather*}
 \textstyle \sum_{i=1}^{2n+1}  \xi_{e_i} \xi_{ e_i} \zeta =   \textstyle \sum_{i=1}^{n}  \xi_{e_i} \xi_{ e_i} \zeta +  \textstyle \sum_{i=1}^{n}  \xi_{\varphi e_i} \xi_{\varphi e_i} \zeta + \xi_\zeta \xi_\zeta \zeta = \xi_\zeta \xi_\zeta \zeta.
 \end{gather*}
 
For the type $\mathcal{A} \oplus \mathcal{B}\oplus   \mathcal{E}$,   as a consequence of $(i)$ and $(ii)$, it is followed  
$$
 \varphi \xi_{e_i} \xi_{\varphi e_i} \zeta = 0, \qquad  
   \xi_{e_i} \xi_{ e_i} \zeta = \xi_\zeta \xi_\zeta \zeta =  \xi_{(11)\zeta} \xi_\zeta \zeta - \| \nabla \zeta \|^2 \zeta , \quad 
 \mbox{and} \quad
  \| \nabla  \zeta \|^2 =   \| \nabla_\zeta  \zeta \|^2. 
$$
The remaining cases are derived by a similar way  using  properties of the intrinsic torsion. 
\end{proof}

Next it is pointed out  more  relative to notation.
\begin{remark}
  {\rm We will use the following standard notation:  $\lambda_0^{p,q}$ is a complex irreducible $\Lie{U}(n)$-module coming from the $(p,q)$-part of the  
complex exterior algebra, and that its corresponding dominant weight in standard
coordinates is given by $(1, \dots,1,0, \dots,0, -1, \dots , -1)$, where $1$ and $-1$ are repeated $p$ and $q$ times, respectively. By analogy with the exterior algebra, there
are also complex irreducible $\Lie{U}(n)$-modules $\sigma^{p,q}_0$, with dominant weights $(p,0, \dots ,0, -q)$ 
coming from the complex symmetric algebra. The notation $\lcf V \rcf$ stands for the real vector space underlying a complex vector space $V$, and $[W ]$ denotes a real vector space that admits $W$ as its complexification. Thus, being $A$ the irreducible $\Lie{U}(n)$-module with dominant weight $(2,1,0, \dots, 0)$,   for the $\Lie{U}(n)$-modules above mentioned one has
\begin{gather*}
\eta^{\perp} \cong \lcf \lambda^{1,0} \rcf, \quad  \lie{u}(n)  \cong  [\lambda^{1,1}], \quad  \lie{su}(n)_s \cong  \lie{su}(n)_a \cong  [\lambda^{1,1}_0] ,  \quad \lie{u}(n)^{\perp}_{|\zeta^{\perp}} \cong  \lcf \lambda^{2,0} \rcf.
\\
\mathcal C_1 \cong \lcf \lambda^{3,0} \rcf, \quad \mathcal C_2 \cong \lcf A \rcf, \quad
\mathcal C_3 \cong \lcf \lambda^{2,0}_0 \rcf, \quad \mathcal C_4 \cong \lcf \lambda^{1,0} \rcf.
\end{gather*}

The space of two forms  $\Lambda^2 \mathrm{T}^* M$ is decomposed into irreducible  $\Lie{U}(n)$-components as follows:
$$
\Lambda^2 \mathrm{T}^* M =   \mathbb{R} \, F+ [\lambda_0^{1,1}]  + \lcf \lambda^{2,0} \rcf +   \eta \wedge  \lcf \lambda^{1,0} \rcf . 
$$
The components of a two-form $\alpha$ are given by 
\begin{gather*}
2 \alpha_{[\lambda^{1,1}]} (X,Y)  =  \alpha (\varphi^2 X , \varphi^2 Y) +  \alpha (\varphi X , \varphi Y),  \quad
2 \alpha_{\lcf \lambda^{2,0}\rcf} (X,Y)  =  \alpha (\varphi^2 X , \varphi^2 Y) -  \alpha (\varphi X , \varphi Y),  \\
\alpha_{\eta \wedge \lcf \lambda^{1,0} \rcf  }  =  \eta  \wedge (\zeta \lrcorner \alpha), \qquad \alpha_{\mathbb{R}F} = \frac1{2n} \langle \alpha , F \rangle F,
\end{gather*}
where $\lrcorner$ denotes the interior product and it is used  the metric given by \eqref{extendedmetric}.  
 In the sequel, we will consider the orthonormal basis for  vectors   $\{e_1, \dots, e_{2n}, e_{2n+1} = \zeta\}$. Likewise, we will use the summation convention. The repeated indexes will mean that the sum is extended from $i=1$ to $i=2n+1$.  Otherwise, the sum will be explicitly  written. 

   }
  \end{remark}

\vspace{0mm}

\section{Geometric interrelations between components of the intrinsic torsion 
}{\indent} \label{foursection}
In this section we will display  several identities relating components of the intrinsic torsion of an almost contact streucture. Such   identities   were already obtained in \cite{FMC6} and  are consequences of the equalities $d^2 F=0$ and $d^2 \eta=0$. They are interesting  on their own and are applied in next section to derive conditions  for  the harmocinity of the structure. Likewise, the identities are also used  to claim the non-existence of certain types of almost contact metric structures. This was already initiated in  \cite{FMC6}. Here we will give further results in such a direction in the final section. The identity in next Lemma is  a consequence of $d^2F=0$. 
\begin{lemma} \label{lambdaunouno}
  For almost contact metric manifolds of dimension $2n+1$, $n>1$, the
following identity is satisfied
{\rm  \small
 \begin{align*}
   0   = &
    \tfrac{n-2}{n-1} \textstyle   \langle \nabla^{\Lie{U}(n)}_{\varphi^2  X} \xi_{(4)e_i} { e_i},    Y \rangle
   -  \tfrac{n-2}{n-1} \textstyle   \langle \nabla^{\Lie{U}(n)}_{\varphi^2 Y} \xi_{(4) e_i} { e_i},    X \rangle
  - \tfrac{2}{n-1} \textstyle  (\nabla^{\Lie{U}(n)}_{e_j}  (\xi_{(4)e_i} { e_i})^\flat)) (\varphi e_j)  F(X,Y)
  \\
  &
   - 2 \textstyle  \langle (\nabla^{\Lie{U}(n)}_{e_i} \xi_{(3)})_{ X} Y,  { e_i} \rangle
    + 2 \textstyle  \langle (\nabla^{\Lie{U}(n)}_{e_i} \xi_{(3)})_{ Y}   X, { e_i} \rangle
  - \tfrac{n-2}{n-1} \textstyle  \langle \nabla^{\Lie{U}(n)}_{\varphi  X} \xi_{(4)e_i} { e_i}, \varphi  Y \rangle
   \\
   &
   +  \tfrac{n-2}{n-1} \textstyle   \langle \nabla^{\Lie{U}(n)}_{\varphi Y} \xi_{(4)e_i} { e_i}, \varphi  X \rangle
- 3 \textstyle  \langle \xi_{(1)X} e_i, \xi_{(2)Y} e_i \rangle 
+ 3  \textstyle  \langle \xi_{(1)Y} e_i, \xi_{(2)X} e_i \rangle     
\\
&
- \tfrac{2}{n^2} d^*\eta d^*F(\zeta) F(X,Y) 
+ \tfrac{4}{n} d^*\eta (\xi_{(7)X} \eta)(Y) 
- \tfrac{4}{n} d^*F(\zeta)    (\xi_{(8)X}  \eta) (\varphi Y) 
\\
&
+ 4 \langle   \xi_{(7)X}  \zeta, \xi_{(8)Y}  \zeta \rangle -  4 \langle   \xi_{(7)Y}  \zeta, \xi_{(8)X}  \zeta \rangle
+ 4 \langle   \xi_{(11)\zeta }  X , \xi_{(10)Y}  \zeta \rangle -  4 \langle   \xi_{(11)\zeta }  Y, \xi_{(10)X}  \zeta \rangle.
\end{align*} }
 \end{lemma}
 
 In  previous Lemma, if we use the equality
\begin{gather*}
\textstyle    (\nabla^{\Lie{U}(n)}_X (\xi_{(4)e_i} { e_i})^\flat) (Y) - (\nabla^{\Lie{U}(n)}_Y (\xi_{(4)e_i} { e_i})^\flat) (X)  = 
\textstyle \tfrac{n-1}{2}  \left(d \theta(X,Y) +   (\xi_X \theta)(Y) -   (\xi_Y \theta)(X)\right).    
\end{gather*}
we will obtain the $[\lambda^{1,1}]$-component  of the exterior derivative of the Lee form $\theta$.
 \begin{proposition} \label{divergenciaunouno}
 For almost contact metric manifolds of dimension $2n+1$, $n>1$,  we have
 \begin{align*}
   \tfrac{n-2}{2} \textstyle  d \theta_{[\lambda^{1,1}]} ( X,  Y)    = &
     \tfrac{1}{4} \textstyle  \langle d \theta , F \rangle   F(X,Y) 
   - \textstyle  \langle (\nabla^{\Lie{U}(n)}_{e_i} \xi_{(3)})_{ X} Y,  { e_i} \rangle
    + \textstyle   \langle (\nabla^{\Lie{U}(n)}_{e_i} \xi_{(3)})_{ Y}   X, { e_i} \rangle
    \\
 &
  +  \tfrac{n-2}{2} \textstyle  \theta (\xi_{(3)  X} Y- \xi_{(3)  Y} X) 
  - \tfrac{3}{2} \textstyle  \langle \xi_{(1)X} e_i, \xi_{(2)Y} e_i \rangle 
  + \tfrac{3}{2} \textstyle  \langle \xi_{(1)Y} e_i, \xi_{(2)X} e_i \rangle     
\\
&
- \tfrac{1}{n^2} d^*\eta d^*F(\zeta) F(X,Y) 
+ \tfrac{2}{n} d^*\eta (\xi_{(7)X} \eta)(Y) 
- \tfrac{2}{n} d^*F(\zeta)    (\xi_{(8)Y}  \eta) (\varphi X) 
\\
&
+ 2 \langle   \xi_{(7)X}  \zeta, \xi_{(8)Y}  \zeta \rangle -  2 \langle   \xi_{(7)Y}  \zeta, \xi_{(8)X}  \zeta \rangle
\\
&
+ 2 \langle   \xi_{(11)\zeta }  X , \xi_{(10)Y}  \zeta \rangle -  2 \langle   \xi_{(11)\zeta }  Y, \xi_{(10)X}  \zeta \rangle
\end{align*}
and
$
 \textstyle   (d\theta)_{\mathbb{R}} ( X,  Y)  = \tfrac{1}{2n}  \textstyle   \langle d \theta , F \rangle F(X,Y), 
$
where
\begin{gather*}
 \textstyle
  \tfrac14 \langle d \theta , F \rangle 
   =
\tfrac{1}{2n}   d^*\eta d^*F(\zeta) -   \textstyle 
  \langle   \xi_{(7) \varphi e_i }  \zeta, \xi_{(8) e_i}  \zeta \rangle  
  -   \textstyle    \langle \xi_{(11)\zeta }  \varphi e_i  , \xi_{(10)e_i}  \zeta \rangle.
\end{gather*}
 \end{proposition}
 
   The identity in next Lemma is also a consequence of $d^2F=0$.  
\begin{lemma}[\cite{FMC6}] \label{firstident}
 For almost contact metric structures of dimension $2n+1$, $n>1$,  the
following identity is satisfied
\vspace{-1mm}
{\rm  \small
     \begin{align*}
     0= \;&
      3 \textstyle  \langle (\nabla^{\Lie{U}(n)}_{e_i} \xi_{(1)} )_{e_i} X_{}, Y_{} \rangle
   - \textstyle  \langle (\nabla^{\Lie{U}(n)}_{e_i} \xi_{(3)} )_{e_i} X_{}, Y_{} \rangle
   + (n-2) \textstyle   \langle (\nabla^{\Lie{U}(n)}_{e_i} \xi_{(4)} )_{e_i} X_{},Y_{}
   \rangle
  \\
&
  - \textstyle   \langle \xi_{{(3)}X_{}} e_i,  \xi_{{(1)} Y_{}} e_i \rangle
  + \textstyle   \langle  \xi_{{(3)}Y_{}} e_i,  \xi_{{(1)}  X_{} }e_i \rangle
 +  \frac12\textstyle   \langle  \xi_{{(3)}X_{}} e_i,  \xi_{{(2)} Y_{}} e_i \rangle
   - \textstyle   \frac12  \langle  \xi_{{(3)}Y_{}} e_i,  \xi_{{(2)}  X_{}} e_i \rangle
    \\
   &
 \textstyle   - \tfrac{n-5}{n-1}  \textstyle  \langle \xi_{{(1)} \xi_{{(4)}e_i} e_i }X_{}, Y_{} \rangle
    - \tfrac{n-2}{n-1}  \textstyle  \langle  \xi_{{(2)} \xi_{{(4)}e_i} e_i } X_{}, Y_{} \rangle
      +  \textstyle   \langle \xi_{{(3)} \xi_{(4)e_i} e_i } X_{}, Y_{} \rangle
     \\
     &
     + \textstyle  (\xi_{(6)e_i} \eta) (\varphi e_i)  \langle  \xi_ {(11)\zeta}  X_{},  \varphi Y_{} \rangle
     + (n-2) \textstyle  (\xi_{(5)e_i} \eta ) \wedge   (  \xi_{(10) e_i} \eta) 
 (  X_{},  Y_{})
       \\
     &
       + (n-2) \textstyle  (\xi_{(6)e_i} \eta ) \wedge   (  \xi_{(10) e_i} \eta) 
 (  X_{},  Y_{}) 
  - 2 \textstyle   (   \xi_{{(7)} \; e_i} \eta) \wedge (  \xi_{{(9)} \; e_i}
     \eta)(  X_{},  Y_{})
              \\
              &
   - 2 \textstyle  (   \xi_{{(7)} \; e_i} \eta) \wedge (  \xi_{{(10)} \; e_i}   \eta)(  X_{},  Y_{})
  - 2  (   \xi_{{(8)} \; e_i} \eta) \wedge (  \xi_{{(10)} \; e_i}    \eta)(  X_{},  Y_{})
      \\
 &
+ 2\textstyle  (\xi_{(7)X_{}} \eta)(\xi_{(11)\zeta} Y_{})   
- 2\textstyle  (\xi_{(7)Y_{}} \eta)(\xi_{(11)\zeta}X_{}). 
   \end{align*}
   }
  \end{lemma}
In   previous Lemma, if we use  the following identity proved in \cite[p. 450, Lemma 4.4]{GDMC} 
\begin{equation}\label{nablatheta}
\langle (\nabla^{\Lie{U}(n)}_{e_i} \xi_{(4)})_{ e_i} X_{}, Y_{}\rangle =\tfrac12 d \theta_{\lcf \lambda^{2,0}\rcf}(X,Y)  - \langle \xi_{(1) \theta^\sharp} X_{} , Y \rangle  + \tfrac12 \langle \xi_{(2) \theta^\sharp} X,Y \rangle, 
\end{equation}
we will obtain   the $\lcf \lambda^{2,0} \rcf$-component of the exterior derivative of the Lee form $\theta$.
\begin{proposition} \label{divergenciadoscero}
  For almost contact metric manifolds of dimension $2n+1$, $n>1$, the
following identity is satisfied
\begin{equation*}
{\rm  \small
     \begin{array}{rl}
           \tfrac{n-2}{2} \textstyle  d\theta_{\lcf \lambda^{2,0} \rcf} (X,Y) 
     = &
      - 3 \textstyle  \langle (\nabla^{\Lie{U}(n)}_{e_i} \xi_{(1)} )_{e_i} X_{}, Y_{} \rangle
   + \textstyle  \langle (\nabla^{\Lie{U}(n)}_{e_i} \xi_{(3)} )_{e_i} X_{}, Y_{} \rangle
  + \textstyle   \langle \xi_{{(3)}X_{}} e_i,  \xi_{{(1)} Y_{}} e_i \rangle
    \\[2mm]
&
  - \textstyle   \langle  \xi_{{(3)}Y_{}} e_i,  \xi_{{(1)}  X_{} }e_i \rangle
  -  \frac12\textstyle   \langle  \xi_{{(3)}X_{}} e_i,  \xi_{{(2)} Y_{}} e_i \rangle
   + \textstyle  \frac12    \langle  \xi_{{(3)}Y_{}} e_i,  \xi_{{(2)}  X_{}} e_i \rangle
    \\[2mm]
   &
 \textstyle  
  + \tfrac{3(n-3)}{2}  \textstyle  \langle \xi_{{(1)}
 \theta^\sharp }X_{}, Y_{} \rangle
        - \textstyle \frac{n-1}{2}  \langle \xi_{{(3)} 
      \theta^\sharp } X_{}, Y_{} \rangle
     - \textstyle d^*F(\zeta)  
      \langle  \xi_ {(11)\zeta}  X_{},  \varphi Y_{} \rangle
        \\[2mm]
     &
          - \tfrac{n-2}{2n} d^* \eta     (  \xi_{(10)  X_{}} \eta)  (   Y_{})
         + \tfrac{n-2}{n} d^* F (\zeta) \textstyle     (  \xi_{(10) \varphi X} \eta)  ( \varphi   Y_{}) 
   \\[2mm]
     &
  + 2 \textstyle   (   \xi_{{(7)} \; e_i} \eta) \wedge (  \xi_{{(9)} \; e_i}
     \eta)(  X_{},  Y_{})
   + 2 \textstyle   (   \xi_{{(7)} \; e_i} \eta) \wedge (  \xi_{{(10)} \; e_i}   \eta)(  X_{},  Y_{})
      \\[2mm]
     &
  + 2  (   \xi_{{(8)} \; e_i} \eta) \wedge (  \xi_{{(10)} \; e_i}    \eta)(  X_{},  Y_{})
- 2\textstyle  (\xi_{(7)X_{}} \eta)(\xi_{(11)\zeta} Y_{})   
+ 2\textstyle  (\xi_{(7)Y_{}} \eta)(\xi_{(11)\zeta}X_{}). 
   \end{array}}
  \end{equation*}
  \end{proposition}


  Next we give another   consequence of $d^2 F=0$.  
   \begin{lemma}[\cite{FMC6}] \label{etameanid1}
 For almost contact metric manifolds of dimension $2n+1$,  the following
identity is satisfied
  \begin{eqnarray*}
0 & = &
    - \textstyle \langle (\nabla^{\Lie{U}(n)}_{\zeta} \xi_{(4)})_{e_i} { e_i},  X \rangle
 - (n-1)   \textstyle   ((\nabla^{\Lie{U}(n)}_{e_i} \xi_{(5)})_{e_i} {\eta})  (X)
      +  \textstyle    ((\nabla^{\Lie{U}(n)}_{e_i} \xi_{(8)})_{e_i} {\eta}) (X)
      \\
      &&
        - \textstyle     ((\nabla^{\Lie{U}(n)}_{e_i} \xi_{(10)})_{ e_i} {\eta}) (X)
   + \textstyle    \langle (\nabla^{\Lie{U}(n)}_{e_i} \xi_{(11)})_{\zeta}   { e_i} , X \rangle
               -  \textstyle   (\xi_{(8)e_i} {\eta})(\xi_{(3)e_i} X)
               \\
               &&
              -  \textstyle    (\xi_{(7)e_i} {\eta})(\xi_{(3)e_i} X) 
       + \textstyle    (\xi_{(10)e_i} {\eta})(\xi_{(1)X} e_i)
              -	 \tfrac12  \textstyle   (\xi_{(10)e_i} {\eta})(\xi_{(2)X} e_i)
         \\
         &&
           +  \textstyle    \langle  \xi_{(11)\zeta}   {e_i} , \xi_{(1)X}  e_i \rangle 
 - \textstyle   \tfrac12    \langle  \xi_{(11)\zeta}   {e_i} , \xi_{(2)X}  e_i \rangle 
        +  \textstyle     (\xi_{(5)   \xi_{(4)e_i} { e_i}} \eta ) (X)
        \\
        &&
        - \tfrac{1}{n-1} \textstyle    (\xi_{(8)   \xi_{(4)e_i} { e_i}} \eta ) (X)
         + \textstyle    \langle  (\xi_ {(9)\xi_{(4)e_i} {e_i}}  {\eta}) (X) 
   - \textstyle    \langle  (\xi_ {(6)\xi_{(4)e_i} {e_i}}  {\eta}) (X) 
   \\
   &&
     - \textstyle  
      (\xi_{(7)   \xi_{(4)e_i} { e_i}} \eta ) (X)
- (n-1) \textstyle     (\xi_{(5)\xi_\zeta \zeta} \eta)(X)
- \textstyle   ( \xi_{(10) \xi_{\zeta} \zeta} \eta)(X)
     - \textstyle   \langle  \xi_{(11)\zeta}   X , \xi_\zeta \zeta  \rangle.
       \end{eqnarray*}      
\end{lemma}

Next by  noting  that  $ (\nabla^{\Lie{U}(n)}_X \xi_{(4)})_{e_i} e_i =  \nabla^{\Lie{U}(n)}_X \xi_{(4) e_i} e_i$ and using the identities   
\begin{gather*}
(\xi_{(5)X} \eta) (Y_{})  =  \tfrac{d^*\eta}{2n}( \langle X_{} , Y_{} \rangle - \eta(X) \eta(Y)), 
\quad (\xi_{(6)X} \eta) (Y)  =  -\tfrac{d^*F(\zeta)}{2n} F( X, Y),
\\
( (\nabla^{\Lie{U}(n)}_X  \xi_{(5)})_Y \eta) (Z_{}) = \tfrac{d(d^*\eta)(X)}{2n}( \langle Y_{} , Z_{} \rangle - \eta(Y) \eta(Z)), \\
\textstyle \langle (\nabla^{\Lie{U}(n)}_{\zeta} \xi_{(4)})_{ e_i} { e_i},  X \rangle = \tfrac{n-1}{2} \left( d \theta (\zeta , X) - \theta ( \xi_{(11)\zeta} X)  - \theta (\xi_{\varphi^2 X} \zeta)\right),
 \end{gather*}
  another version  of the  identity in  previous Lemma  is given  in next Proposition. Such a version relates the exterior derivatives of the Lee form   $\theta$ and the coderivative $d^*\eta$.
\begin{proposition}   \label{divergencia}
For almost contact metric manifolds of dimension $2n+1$,  we have 
  \begin{align*}
 		     \textstyle  \tfrac{n-1}{2} d \theta (\zeta , X)
		     		       = &
          \tfrac{n-1}{2n} d(d^* \eta) (\varphi^2 X) 
         +  \textstyle    ((\nabla^{\Lie{U}(n)}_{e_i} \xi_{(8)})_{e_i} {\eta}) (X)
        - \textstyle      ((\nabla^{\Lie{U}(n)}_{e_i} \xi_{(10)})_{ e_i} {\eta}) (X)
 \\
   &
   + \textstyle    \langle (\nabla^{\Lie{U}(n)}_{e_i} \xi_{(11)})_{\zeta}   { e_i} , X \rangle     
                  -  \textstyle   (\xi_{(7)e_i} {\eta})(\xi_{(3)e_i} X)   
               -  \textstyle   (\xi_{(8)e_i} {\eta})(\xi_{(3)e_i} X)
                                  \\
& 
       + \textstyle    (\xi_{(10)e_i} {\eta})(\xi_{(1)X} e_i)
              -	\textstyle  \tfrac12      (\xi_{(10)e_i} {\eta})(\xi_{(2)X} e_i)
           +  \textstyle    \langle  \xi_{(11)\zeta}   {e_i} , \xi_{(1)X}  e_i \rangle 
           \\
           &
 - \textstyle  \tfrac12    \langle  \xi_{(11)\zeta}   {e_i} , \xi_{(2)X}  e_i \rangle 
        - \tfrac{n}{2} \textstyle    (\xi_{(8) 
        \theta^\sharp } \eta ) (X) 
         + \textstyle    \frac{n-1}{2}\langle  (\xi_ {(10)
           \theta^\sharp }  {\eta}) (X) 
         \\
         &
          +  \textstyle  \tfrac{n-1}{2}  
          \theta( \xi_{(11)\zeta}   X) 
- \tfrac{n-1}{2n} d^*\eta (\xi_\zeta \eta)(X) 
- \textstyle   ( \xi_{(10) \xi_{\zeta} \zeta} \eta)(X)
     - \textstyle   \langle  \xi_{(11)\zeta}   X , \xi_\zeta \zeta  \rangle.
       \end{align*}      
   In particular, if the  structure is  of type $\mathcal{C}_1 \oplus \mathcal{C}_2 \oplus \mathcal{C}_3 \oplus    \mathcal{C}_5    \oplus \mathcal{C}_6 \oplus \mathcal{C}_9 \oplus  \mathcal{C}_{12} $ or $\mathcal{C}_1 \oplus \mathcal{C}_2 \oplus \mathcal{C}_5 \oplus    \mathcal{C}_6    \oplus \mathcal{C}_7 \oplus \mathcal{C}_9 \oplus  \mathcal{C}_{12} $  and  $n>1$,  then  $d (d^* \eta)$ is given  by 
   $d (d^* \eta) = - d^*\eta \,\xi_\zeta \eta + d (d^* \eta)(\zeta) \eta $. Likewise,  for the type 
   $\mathcal{C}_1 \oplus \mathcal{C}_2 \oplus \mathcal{C}_3 \oplus    \mathcal{C}_5     \oplus \mathcal{C}_9 \oplus  \mathcal{C}_{12} $ and $n > 1$,   the one-form $\mathrm{div } (\zeta) \; \eta = - d^*\eta \, \eta$ is closed.
\end{proposition}

If  we consider the identity $d^2 \eta=0$, we will  obtain an expression for $d(d^*F(\zeta))$.
\begin{lemma} \label{mainid} 
For almost contact metric  manifolds of dimension $2n+1$, the exterior derivative $d(d^*F(\zeta))$ is given by 
\begin{align*}
  \textstyle  \frac{n-1}{2n} d(d^* F(\zeta))(X_{\zeta^\perp})  = & 
     \textstyle    ((\nabla^{\Lie{U}(n)}_{e_i} \xi_{(7)})_{ e_i} \eta))(\varphi X)
   -   \textstyle     ((\nabla^{\Lie{U}(n)}_{e_i} \xi_{(10)})_{ e_i} \eta)(\varphi X)
 -  \textstyle       (\xi_{(7)e_i} \eta)(\xi_{(3)e_i} \varphi  X) 
\\
&
- 2\textstyle        (\xi_{(10) e_i} \eta)( \xi_{(1)\varphi X} e_i )  
-  \textstyle   \frac12     (\xi_{(10) e_i} \eta)( \xi_{(2)\varphi X} e_i ) 
  -\frac{n-1}{2n}     d^* F(\zeta)  \theta( X)   
  \\
  &
-   \textstyle  \tfrac{n-1}{2}  ( \xi_{(7) \theta^\sharp } \eta)(\varphi  X)
+ \textstyle     \frac{n-2}{2} ( \xi_{(10) \theta\sharp } \eta)(\varphi X ) 
+  \textstyle  \frac{n-1}{2n} d^* F(\zeta) ( \xi_{(12)\zeta  }  \eta ) ( X)
\\
&
-  (\xi_{(7)  \xi_{(12)\zeta  } \zeta} \eta ) ( \varphi X )
+  (\xi_{(10)  \xi_{(12)\zeta  } \zeta} \eta ) ( \varphi X ),
\end{align*}
\begin{align*}
 d(d^*F(\zeta))(\zeta) =&   
   \textstyle   \frac{1}{n} d^*\eta d^*F(\zeta)  - \langle d \xi_\zeta \eta , F \rangle      
       + \textstyle  2     (\xi_{(7) e_i} \eta)  (\varphi \xi_{(8)e_i} {\zeta}) 
       + 2      (\xi_{(10) e_i} \eta) (\varphi  \xi_{(11)\zeta} e_i).      
\end{align*}
\end{lemma}
For our purposes, it is interesting to note that
$
  (\nabla^{\Lie{U}(n)}_{Z} \xi_{(6)})_{X} Y  =  \tfrac{1}{2n} d(d^*F(\zeta))(Z) (F(X,Y) \zeta + \eta(Y) \varphi(X)).
$
From this, it is obtained
$
\textstyle  ( (\nabla^{\Lie{U}(n)}_{e_i} \xi_{(6)})_{e_i} \eta)(Z) = - \frac{1}{2n} d(d^*F(\zeta))(\varphi Z).
$
\vspace{1mm}

Also as a consequence of  $d^2 \eta =0$, one has the identities given   in next Lemma already proved   in \cite[Lemma 3.9]{FMC6}.  $(d \xi_\zeta \eta)_V$ will denote the projection of $d \xi_\zeta \eta$ on the $\Lie{U}(n)$-space $V$.
\begin{lemma} 
\label{dosdeeta}
The $\Lie{U}(n)$-components of $d \xi_\zeta \eta$ are given by: 

 \noindent $(d \xi_\zeta \eta)_{\mathbb{R} \, F} = \frac1{2n} \langle d \xi_\zeta \eta, F \rangle \; F$, where 
{\small \begin{align*}
\qquad \tfrac12 \langle d \xi_\zeta \eta, F \rangle 
 =& 
       -  d(d^*F(\zeta))(\zeta)  + \textstyle \frac{1}{n} d^*\eta \, d^*F(\zeta)
      + 2 \textstyle   ( \xi_{(7)  e_i} \eta)  (\varphi \xi_{(8)e_i} \zeta)   
        + 2 \textstyle   (\xi_{(10) e_i} \eta) (\varphi \xi_{(11)\zeta} e_i ),  
\end{align*}
\begin{align*}
 (d \xi_{\zeta}  \eta)_{[\lambda^{1,1}] }  (X, Y) =& 
      -  \textstyle  \tfrac{1}{n} d(d^*F(\zeta))(\zeta) F(X, Y)
     +  \textstyle  \tfrac{1}{n^2} d^*\eta d^*F(\zeta) F(X, Y)
     +2((\nabla^{\Lie{U}(n)}_{\zeta} \xi_{(7)})_{X}  \eta) (Y )
    \\
    & - \tfrac{2}{n} d^* \eta (\xi_{(7) X} \eta) (Y)
      +  \textstyle  \tfrac{2}{n} d^*F(\zeta)  (\xi_{(8) X} \eta) (\varphi Y)
      - 2   (\xi_{(7) X} {\eta})  (\xi_{(8) Y} \zeta )  
      \\
      &
       +  2    ( \xi_{(7) Y} {\eta})  (\xi_{(8) X} \zeta) 
        +  2  (\xi_{(9) X} \eta)  (\xi_{(10) Y} {\zeta}) 
     - 2  (\xi_{(9) Y} \eta)  (\xi_{(10) X} {\zeta})
      \\
     &
    + 2 (\xi_{(10)X} \eta) (\xi_{(11)\zeta} Y ) 
    - 2 (\xi_{(10)Y} \eta) (\xi_{(11)\zeta} X ),     
       \end{align*}       
      \begin{align*}
  \hspace{-11mm} 
  (d\xi_{\zeta}  \eta)_{\lcf \lambda^{2,0} \rcf } ( X, Y)  = & 
    2((\nabla^{\Lie{U}(n)}_{\zeta} \xi_{(10)})_{X}  \eta) (Y )
   -   \textstyle  \tfrac2{n} d^*F(\zeta) \langle \xi_{(11)\zeta} X , \varphi Y \rangle
      - 2   (\xi_{ (7) X} \eta)   (\xi_{ (9) Y} \zeta)
      \\
      &
       + 2   (\xi_{ (7) Y} \eta)  ( \xi_{ (9) X} \zeta)
    + 2  (\xi_{(7)X} \eta) (\xi_{(11)\zeta} Y ) 
    - 2  (\xi_{(7)Y} \eta) (\xi_{(11)\zeta} X  )
\\
& - \tfrac{2}{n} d^*\eta (\xi_{(10) X} \eta)(Y)
        + 2  (\xi_{ (8) X} \eta)   (\xi_{ (10) Y} \zeta)  
       - 2   ( \xi_{ (8) Y} \eta)   (\xi_{ (10) X} \zeta),   
   \end{align*} }
 \noindent $(d \xi_{\zeta}  \eta)_{\eta \wedge \lcf \lambda^{1,0} \rcf} = \eta \wedge \zeta \lrcorner d \xi_\zeta \eta$, where  
{\small \begin{align*}
  \zeta \lrcorner d \xi _\zeta \eta (X)  
  =&    ((\nabla^{\Lie{U}(n)}_{\zeta} \xi_{(12)})_\zeta  \eta) (X )  +  (\xi_{(12)\zeta} \eta)  ( \xi_{(11)\zeta} X)  
    -  \textstyle   \tfrac{1}{2n} d^*\eta   (\xi_{(12)\zeta} \eta )  (X) 
    \\
    & 
     - (\xi_{(8) \xi_{(12)\zeta} \zeta} \eta) (X) 
-  (\xi_{(9) \xi_{(12)\zeta} \zeta} \eta) (X)  
 -  \textstyle  \tfrac{1}{2n} d^*F(\zeta) (\xi_{(12)\zeta} \eta )  (\varphi X)  
 \\
 &
 + (\xi_{(7) \xi_{(12)\zeta} \zeta} \eta) (X) 
+  (\xi_{(10) \xi_{(12)\zeta} \zeta} \eta) (X). 
 \end{align*}}
\end{lemma}

\section{Harmonic almost contact structures}{\indent}
\setcounter{equation}{0} \label{harmonicalmostcontact}
  In this section we will show
conditions relating harmonicity, curvature  and types 
 of almost contact metric structures. Firstly, we recall the following characterization 
for  harmonic almost contact structures given  in \cite{GDMC4}.
\begin{theorem} \label{characharmalmcontact}
If  $(M,  \langle \cdot, \cdot \rangle,\varphi, \zeta )$ is  a
$2n+1$-dimensional  almost contact metric manifold, then the following conditions are equivalent:
 \begin{enumerate}
 \item[{\rm (i)}] The  structure   is harmonic.






\item[{\rm (ii)}] For all $X, Y \in \mathfrak X ( M)$, one has
$\langle  (\nabla^{\Lie{U}(n)}_{e_i} \xi)_{ e_i}  X_{\zeta^{\perp}} , Y_{\zeta^{\perp}} \rangle
  + \langle \xi_{\xi_{e_i}  e_i}  X_{\zeta^{\perp}} , Y_{\zeta^{\perp}}
  \rangle = 0\,$ and $\,(\nabla^{\Lie{U}(n)}_{e_i} \xi)_{e_i} \zeta +   \xi_{\xi_{e_i} e_i} \zeta
   =0$.
\end{enumerate}
In such case we have $\nabla^{*} \nabla \zeta = -
\xi_{e_i}\xi_{e_i} \zeta$.  In  particular, a structure of type
$\mathcal{C}_{5} \oplus \ldots \oplus \mathcal{C}_{10} \oplus
\mathcal{C}_{12}$ is harmonic if and only if $\nabla^* \nabla \zeta
= \| \nabla \zeta \|^2 \zeta$, that is, the Reeb  vector
field $\zeta$ is harmonic unit vector field {\rm (see
\cite{Wie1,GMS} for this notion)}.
 \end{theorem}

In next results, for certain types of almost contact metric
structures, we will deduce  conditions  characterising harmonic
 structures. Such conditions are mainly given in terms of   $\ast$-Ricci tensor and  components of the intrinsic torsion.
\begin{theorem} \label{harmclasscontact}
For a $2n+1$-dimensional almost contact metric manifold
$(M,\langle \cdot ,\cdot \rangle, \varphi , \zeta)$,  we have:
\begin{enumerate}
\item[{\rm (i)}] If $M$ is
  of type  $\mathcal C_1 \oplus \mathcal C_2 \oplus \mathcal C_4 \oplus  \mathcal C_5 \oplus \mathcal C_6 \oplus \mathcal C_7 \oplus \mathcal  C_8 \oplus \mathcal C_{11} \oplus \mathcal  C_{12} =  \mathcal A\oplus \mathcal C_4  \oplus \mathcal C \oplus  \mathcal E$,  then the structure is harmonic if and
only if 
\begin{equation} \label{acuatrocehor}
{\small \begin{array}{rl}
  \Ricac_{\mbox{\rm \footnotesize alt}}( X_{\zeta^{\perp}} , Y_{\zeta^{\perp}})  = \; 
 &
 d \theta_{\lcf \lambda^{2,0}\rcf} (  X , Y)
  +(n-3) \langle \xi_{(1)\theta }  X , Y
  \rangle 
   + n \langle \xi_{(2)\theta }  X , Y
  \rangle 
  \\[1mm]
  &
  - d^*\eta  \langle \xi_{(11)\zeta}  X , Y \rangle
    - d^*F(\zeta) \langle \xi_{(11) \zeta } \varphi X , Y
  \rangle 
  \\
  &
   + \langle  (\nabla^{\Lie{U}(n)}_{\zeta} \xi_{(11)})_{\zeta}  X , Y \rangle
    + \langle \xi_{(1)\xi_{(12)\zeta}  \zeta}  X , Y
  \rangle
  \\
   &  + \langle \xi_{(2)\xi_{(12)\zeta}  \zeta}  X , Y
  \rangle
   +\tfrac14   \xi_{\zeta}  \eta\wedge \theta  (X , Y)
  \\[0.5mm]
  &
  -\tfrac14   \varphi \xi_{\zeta}  \eta\wedge \varphi \theta  (X , Y)
  \rangle,\end{array}}
  \end{equation} 
for all $X,Y \in \mathfrak
X(M)$,  and 
\begin{align} \label{cinco}
\Ricac( \zeta) =
 &
  - (\zeta \lrcorner d \xi_\zeta \eta)^\sharp  
 -\tfrac{n-1}{4n}  d^* \eta  \theta^\sharp 
 -\tfrac{2n-1}{2}  \xi_{(8) \theta^\sharp} \zeta 
 -\tfrac{3(n-1)}{4n}  d^*F(\zeta) \varphi   \theta^\sharp 
 \\
 &
 -\tfrac{2n-3}{2}  \xi_{(7) \theta^\sharp} \zeta 
  +\tfrac{n-1}{n} d^*\eta \xi_{\zeta} \zeta 
- 2  \xi_{(8)\xi_{\zeta} \zeta } \zeta 
 + d^*F(\zeta) \varphi \xi_\zeta \zeta
 - \xi_{(11)\zeta} \xi_{(12)\zeta} \zeta.
    \nonumber
    \end{align}

\item[{\rm (ii)}] If $M$ is
  of type  $\mathcal C_1 \oplus \mathcal C_2 \oplus \mathcal C_4 \oplus \mathcal \mathcal C_9  \oplus  \mathcal C_{10}    \oplus  \mathcal C_{11} \oplus  \mathcal C_{12}   = \mathcal A \oplus \mathcal C_4 \oplus \mathcal D \oplus  \mathcal E$,  then the  structure is harmonic if and
only if 
\begin{align} \label{acuatrodehor}
  \Ricac_{\mbox{\rm \footnotesize alt}}( X_{\zeta^{\perp}} , Y_{\zeta^{\perp}})  = \; 
 &
 d \theta_{\lcf \lambda^{2,0}\rcf} (  X , Y)
  +(n-3) \langle \xi_{(1)\theta^\sharp  }  X , Y
  \rangle 
   + n \langle \xi_{(2)\theta^\sharp  }  X , Y
  \rangle 
   \\
  &
   + \langle  (\nabla^{\Lie{U}(n)}_{\zeta} \xi_{(11)})_{\zeta}  X , Y \rangle
     + \langle \xi_{(1)\xi_{(12)\zeta}  \zeta}  X , Y
  \rangle \nonumber
   + \langle \xi_{(2)\xi_{(12)\zeta}  \zeta}  X , Y
  \rangle
  \\
  &
   +\tfrac14   \xi_{\zeta}  \eta\wedge \theta  (X , Y)\nonumber
  -\tfrac14   \varphi \xi_{\zeta}  \eta\wedge \varphi \theta  (X , Y), \nonumber
\end{align} 
for all $X,Y \in \mathfrak
X(M)$,  and
\begin{equation} 
\begin{array}{rl}
\Ricac( \zeta) =
 &
   (\zeta \lrcorner d \xi_\zeta \eta)^\sharp  
   -  \xi_{(1)e_i} \xi_{(10) e_i} \zeta 
   -   \xi_{(2)e_i} \xi_{(\mathcal D) e_i} \zeta
      +(n-1)  \xi_{(9)\theta^\sharp} \zeta 
     \\
  &
  +(n-1)  \xi_{(10)\theta^\sharp} \zeta
 + 2 \xi_{(9)\xi_{\zeta} \zeta } \zeta  + \xi_{(11)\zeta} \xi_{(12)\zeta} \zeta.
    \end{array} \label{unocuatrode}
    \end{equation}

\item[{\rm (iii)}] 
If $M$ is
 of type $ \mathcal C_3  \oplus  \mathcal C_4  \oplus  \mathcal C_5  \oplus  \mathcal C_6  \oplus  \mathcal C_7  \oplus  \mathcal C_8  \oplus  \mathcal C_{11}  \oplus  \mathcal C_{12} =\mathcal B \oplus  \mathcal C \oplus \mathcal E$,
then the  structure is harmonic if and only
if
\begin{align*}
 \hspace{0.2cm}  \Ricac_{\mbox{\rm \footnotesize alt}}( X_{\zeta^{\perp}} , Y_{\zeta^{\perp}})  =
 &
       -  (n-1) \langle \xi_{(3) \theta^\sharp }  X , Y\rangle 
    -   \langle  (\nabla^{\Lie{U}(n)}_{\zeta} \xi_{(11)})_{\zeta}  X , Y \rangle
  -    \langle \xi_{(3) \xi_\zeta \zeta }  X , Y\rangle 
  \\
  &
  - \tfrac14 \xi_\zeta \eta \wedge \theta (X , Y)
 +    \tfrac14 \varphi \xi_\zeta \eta \wedge \varphi\theta (X , Y)
   + d^*\eta  \langle \xi_{(11)\zeta}  X , Y \rangle
   \\
   &
  - d^*F(\zeta)  \langle \xi_{(11)\zeta}  \varphi X ,  Y
\rangle,
\end{align*} 
for all $X,Y \in \mathfrak
X(M)$,  and 
\begin{align*}
\Ricac( \zeta) =
 &
  - (\zeta \lrcorner d \xi_\zeta \eta)^\sharp  
   - \langle \xi_{(8)} \zeta, \langle \xi_{(3) \cdot} e_j , \cdot \rangle \rangle e_j  
 - \langle \xi_{(7)} \zeta, \langle \xi_{(3) \cdot} e_j , \cdot \rangle \rangle e_j  
    \\
& 
-\tfrac{n-1}{4n} d^*\eta   \theta^\sharp
-\tfrac{2n-1}{2}  \xi_{(8)\theta^\sharp} \zeta
-\tfrac{3(n-1)}{4n} d^*F(\zeta)  \varphi  \theta^\sharp
  -\tfrac{2n-3}{2}  \xi_{(7)\theta^\sharp} \zeta
  \\
  &
  + \tfrac{n-1}{n} d^*\eta\xi_{\zeta} \zeta - 2\xi_{(8)\xi_{\zeta} \zeta } \zeta 
    + d^*F(\zeta) \varphi    \xi_{\zeta } \zeta - \xi_{(11)\zeta} \xi_{(12)\zeta} \zeta.
     \end{align*}
     
     Note that in this case of harmonic structure  one has
     \begin{align*}
     \langle  (\nabla^{\Lie{U}(n)}_{\zeta} \xi_{(11)})_{\zeta}  X , Y \rangle = &  -  \tfrac{n-1}2  d \theta_{\lcf \lambda^{2,0}\rcf} ( X , Y) -        (n-1) \langle \xi_{(3) \theta^\sharp }  X, Y\rangle
      -    \langle \xi_{(3) \xi_\zeta \zeta }  X , Y\rangle 
      \\
      &
       - \tfrac14 ( \xi_\zeta \eta \wedge \theta -
     \varphi \xi_\zeta \eta \wedge \varphi\theta) (X , Y)
     + d^*\eta  \langle \xi_{(11)\zeta}  X, Y \rangle
     \\
     &
     - d^*F(\zeta)  \langle \xi_{(11)\zeta}  \varphi X ,  Y \rangle 
     - 2 (\xi_{(7)X}\eta) ( \xi_{(11)\zeta } Y)
     \\
     &
       + 2 (\xi_{(7)X}\eta) (  \xi_{(11)\zeta}   Y).
     \end{align*}

     \item[{\rm (iv)}]
 If $M$ is
 of type  $\mathcal C_3 \oplus \mathcal C_4 \oplus \mathcal C_9 \oplus \mathcal C_{10} \oplus \mathcal C_{11} \oplus \mathcal C_{12}  = \mathcal B \oplus  \mathcal D \oplus \mathcal E$,
 then the  structure is harmonic if and only if
\begin{align*}
  \Ricac_{\mbox{\rm \footnotesize alt}}( X_{\zeta^{\perp}} , Y_{\zeta^{\perp}})  =
 &
       -  (n-1) \langle \xi_{(3) \theta^\sharp }  X , Y\rangle 
      - \langle  (\nabla^{\Lie{U}(n)}_{\zeta} \xi_{(11)})_{\zeta}  X , Y \rangle,
\end{align*} 
for all $X,Y \in \mathfrak X(M)$, and 
  \begin{align*}
\Ricac( \zeta) =
 &
 (\zeta \lrcorner d \xi_\zeta\eta)^\sharp 
 + (n-1)  \xi_{(9)\theta^\sharp} \zeta
  + (n-1)  \xi_{(10)\theta^\sharp} \zeta
 + 2 \xi_{(9)\xi_{\zeta} \zeta } \zeta 
 + \xi_{(11)\zeta}  \xi_\zeta \zeta. 
     \end{align*} 
     Note that in this case of harmonic structure  one has
     $$
     \langle  (\nabla^{\Lie{U}(n)}_{\zeta} \xi_{(11)})_{\zeta}  X , Y \rangle =  -  \tfrac{n-1}2  d \theta_{\lcf \lambda^{2,0}\rcf} ( X , Y) -        (n-1) \langle \xi_{(3) \theta^\sharp }  X , Y\rangle.
     $$    
\end{enumerate}
\end{theorem}

\begin{proof}
For (i),  the tensor $\xi$ for  structures of
type $ \mathcal A \oplus \mathcal C_4 \oplus \mathcal C \oplus
\mathcal E$ is such
that    $\xi_{(\mathcal A)\varphi X} \varphi Y = -
\xi_{(\mathcal A)X} Y$, 
$\xi_{(4)\varphi X} \varphi Y = 
\xi_{(4)X} Y$,
 and $(\xi_{\varphi X}
\eta )(\varphi Y) = (\xi_{X_{\zeta^{\perp}}}\eta)(Y_{\zeta^{\perp}})$ \cite{ChineaGonzalezDavila}.
Also in such a case, one has $\xi_{\xi_{(6)e_i} \varphi e_i} \zeta = - d^*F(\zeta) \xi_\zeta \zeta$ and $\xi_{\xi_{(4)e_i} \varphi e_i} \zeta = -\xi_{\varphi \xi_{(4)e_i}e_i} \zeta  = -\tfrac{n-1}{2}  \xi_{\varphi \theta^\sharp} \zeta$.
Taking all of this into account, if the structure is harmonic, using  Lemma \ref{astricciacm1},  Theorem
\ref{characharmalmcontact} and  \eqref{nablatheta}, 
it is followed the first condition  required in (i).  

The second condition in (i) is derived by making use of Lemma \ref{astricciacm1},  Theorem
\ref{characharmalmcontact}, 
Lemma \ref{xixizeta1}, Lemma \ref{dosdeeta} and the following identities
\begin{align}
\xi_{(\mathcal A) e_i } \xi_{(\mathcal C) e_i } \zeta  =\,&0,\label{aczero}\\
\xi_{(4) e_i } \xi_{(\mathcal C) e_i } \zeta  =\,& \tfrac{n-1}{4n} d^*\eta \theta^\sharp - \tfrac12 \xi_{(8) \theta^\sharp} \zeta - \tfrac{n-1}{4n} d^*F(\zeta) \varphi  \theta^\sharp  + \tfrac12 \xi_{(7) \theta^\sharp} \zeta, \label{idxicuatro} \\ 
\xi_{(5)X} Y =\,& - \tfrac{d^*\eta}{2n} (\langle X, Y \rangle \zeta - \eta(Y) X), \label{idxicinco} 
\\
\xi_{(6)X} Y =\,&  \tfrac{d^*F(\zeta)}{2n} (F (X, Y ) \zeta - \eta(Y) \varphi X). \label{idxiseis} 
\end{align}

Conversely, if both conditions are satisfied, using   Lemma \ref{astricciacm1}, Lemma \ref{dosdeeta},
Lemma \ref{xixizeta1} and the  identities \eqref{aczero}-\eqref{idxiseis},  we will  deduce the harmonicity conditions given by 
 Theorem
\ref{characharmalmcontact}. 
\vspace{1mm}

The proof for (ii) is similar. In this case
 one has    $(\xi_{\varphi X} \eta
)(\varphi Y)  = - (\xi_{X_{\zeta^{\perp}}}\eta)(Y_{\zeta^{\perp}})$ and 
$\xi_{\xi_{e_i} \varphi
e_i} \eta = \xi_{\xi_{(4)e_i} \varphi e_i} \zeta = -\xi_{\varphi \xi_{(4)e_i}e_i} \zeta  = -\tfrac{n-1}{2}  \xi_{\varphi \theta^\sharp} \zeta$.
\vspace{1mm}

For (iii),  the intrinsic torsion in this case   is
such that   $\xi_{(\mathcal{B}) \varphi X} \varphi Y =
\xi_{(\mathcal{B})X_{\zeta^{\perp}}} Y_{\zeta^{\perp}}$ and  $(\xi_{\varphi X}
\eta) (\varphi Y) = (\xi_{\zeta^{\perp}} \eta)(Y_{\zeta^{\perp}})$ (see
\cite{ChineaGonzalezDavila}). Therefore, as in the previous cases, the  required identities
in (iii) are  consequences of Lemma \ref{astricciacm1},
 Lemma \ref{dosdeeta}, Theorem \ref{characharmalmcontact},  Lemma \ref{xixizeta1} and identities \eqref{nablatheta}, \eqref{idxicuatro}, \eqref{idxicinco}, \eqref{idxiseis}. Moreover, to deduce an expression for $ \langle  (\nabla^{\Lie{U}(n)}_{\zeta} \xi_{(11)})_{\zeta}  X , Y \rangle$, Lemma \ref{firstident} is used.  The converse is straightforward.
 
 For (iv),   $\xi$  in this case   is
such that     $(\xi_{\varphi X}
\eta) (\varphi Y) = -(\xi_{X_{\zeta^{\perp}}} \eta)(Y_{\zeta^{\perp}})$. Therefore, as in the previous cases, the  required identities
in (iv) are  consequences of Lemma \ref{astricciacm1},
 Lemma \ref{dosdeeta}, Theorem \ref{characharmalmcontact} and  Lemma \ref{xixizeta1}.  Moreover, to deduce an expression for $ \langle  (\nabla^{\Lie{U}(n)}_{\zeta} \xi_{(11)})_{\zeta}  X , Y\rangle$, Lemma \ref{firstident} is used. The converse easily follows by using results above mentioned. 
\end{proof}

Next proposition contains some  particular cases, most of them  of previous Theorem. For proving them,   identities of Section \ref{foursection} are used.
\begin{proposition} $\,$ \label{harmclasscontact2}
For a $2n+1$-dimensional almost contact metric manifold
$(M,\langle \cdot ,\cdot \rangle, \varphi , \zeta)$,  we have:
\begin{enumerate}
\item[{\rm (i)}] If $M$ is
  of type  $\mathcal C_1 \oplus \mathcal C_2 \oplus \mathcal C_5 \oplus \mathcal C_6 \oplus \mathcal C_7 \oplus \mathcal  C_8 \oplus \mathcal C_{11} \oplus \mathcal  C_{12}$,  then the  structure is harmonic if and
only if 
\begin{align*}
\hspace{1cm} \Ricac_{\rm alt}( X_{\zeta^{\perp}} , Y_{\zeta^{\perp}} )
= &
\langle (\nabla^{\Lie{U}(n)}_\zeta \xi_{(11)})_\zeta X ,  Y \rangle 
+ \langle  \xi_{(1)\xi_\zeta \zeta}  X,  Y \rangle
 + \langle  \xi_{(2)\xi_\zeta \zeta}  X ,  Y \rangle
 \\
&
-  d^*\eta (\zeta)\langle   \xi_{(11)\zeta}  X ,  Y \rangle
- d^*F(\zeta)\langle  \xi_{(11)\zeta}  X ,  \varphi  Y \rangle, 
\end{align*}
for all $X,Y \in \mathfrak
X(M)$,  and 
\begin{align*}
\Ricac( \zeta) =
 &
   - (\zeta \lrcorner d \xi_\zeta \eta)^\sharp  
  +\tfrac{n-1}{n} d^*\eta \xi_{\zeta} \zeta 
- 2  \xi_{(8)\xi_{\zeta} \zeta } \zeta 
 + d^*F(\zeta) \varphi \xi_\zeta \zeta
 - \xi_{(11)\zeta} \xi_{(12)\zeta} \zeta.
    \end{align*}
  
\item[{\rm (ii)}] If $M$ is
  of type  $ \mathcal C_1 \oplus \mathcal C_2 \oplus \mathcal C_9 \oplus \mathcal C_{10} \oplus \mathcal C_{11} \oplus \mathcal  C_{12}$,  then the  structure is harmonic if and
only if 
\begin{align*}
\hspace{1cm}  \Ricac_{\mbox{\rm \footnotesize alt}}( X_{\zeta^{\perp}} , Y_{\zeta^{\perp}})  = \; 
 &
    \langle  (\nabla^{\Lie{U}(n)}_{\zeta} \xi_{(11)})_{\zeta}  X , Y \rangle
      + \langle \xi_{(1)\xi_{(12)\zeta}  \zeta}  X , Y
  \rangle
   + \langle \xi_{(2)\xi_{(12)\zeta}  \zeta}  X , Y
  \rangle, 
\end{align*} 
for all $X,Y \in \mathfrak
X(M)$,  and 
\begin{align*}
\Ricac( \zeta) =
 &
   (\zeta \lrcorner d \xi_\zeta \eta)^\sharp  
   -  \xi_{(1)e_i} \xi_{(10) e_i} \zeta 
   -   \xi_{(2)e_i} \xi_{(\mathcal D) e_i} \zeta 
 + 2 \xi_{(9)\xi_{\zeta} \zeta } \zeta  + \xi_{(11)\zeta} \xi_{(12)\zeta} \zeta.
    \end{align*}

 \item[{\rm (iii)}]
  If $M$ is
 of type  $\mathcal C_1 \oplus  \mathcal C_4 \oplus  \mathcal C_5 \oplus  \mathcal C_6 \oplus  \mathcal C_7 \oplus  \mathcal C_8 \oplus  \mathcal C_{11} \oplus  \mathcal C_{12}$, then
the structure is harmonic for $n \neq 5$ if and only if
\begin{align*}
\tfrac{n-5}{n+1}  \Ricac_{\mbox{\rm \footnotesize alt}}( X_{\zeta^{\perp}}, Y_{\zeta^{\perp}})  =
 &
   \textstyle   \langle (\nabla^{\Lie{U}(n)}_{\zeta} \xi_{(11)} )_{\zeta} X, Y \rangle
 \textstyle   
  + (n-3)  \textstyle  \langle \xi_{{(1)} \theta^\sharp }X , Y\rangle
  \\
  &
 +  \textstyle  \langle \xi_{(1) \xi_\zeta \zeta} X ,  Y \rangle         
  +\tfrac{1}4 (\xi_\zeta \eta \wedge \theta - \varphi  \xi_\zeta \eta \wedge \varphi \theta)(X, Y)
  \\
 &
  - \textstyle  d^*\eta  \langle  \xi_ {(11)\zeta}  X,  \varphi Y \rangle  
    - \tfrac{n-3}{n+1} d^*F(\zeta)  \langle  \xi_ {(11)\zeta}  \varphi X,  Y \rangle,
 \end{align*}
 for all $X,Y \in \mathfrak X(M)$, and  it is satisfied the identity \eqref{cinco}.

For $n=5$, the above mentioned type of structure is
harmonic if and only if identities \eqref{acuatrocehor}  and \eqref{cinco} are satisfied for this particular case.

 \item[{\rm (iv)}]
  If $M$ is
 of type  $\mathcal C_1 \oplus  \mathcal C_4 \oplus \mathcal C_{9} \oplus \mathcal C_{10} 
\oplus \mathcal C_{11} \oplus \mathcal C_{12}$, then the  structure is
harmonic for $n\neq 5$ if and only if
\begin{equation*}
{\small \begin{array}{rl}
\tfrac{n-5}{n+1}  \Ricac_{\mbox{\rm \footnotesize alt}}( X_{\zeta^{\perp}}, Y_{\zeta^{\perp}})  =
 &
    \textstyle   \langle (\nabla^{\Lie{U}(n)}_{\zeta} \xi_{(11)} )_{\zeta} X,Y \rangle
 \textstyle  
 +  (n-3) \textstyle  \langle \xi_{{(1)} \theta^\sharp }X, Y \rangle 
 \\
 &    
  + \textstyle  \langle \xi_{(1) \xi_\zeta \zeta} X ,  Y \rangle    
  + \tfrac{1}4 (\xi_\zeta \eta \wedge \theta - \varphi  \xi_\zeta \eta \wedge \varphi \theta)(X, Y),
     \end{array}}
     \end{equation*}
 for all $X,Y \in \mathfrak X(M)$, and it is satisfied the identity \eqref{unocuatrode}.
 
 For $n=5$, the above mentioned type of structure is
harmonic if and only if identities \eqref{acuatrodehor}  and \eqref{unocuatrode} are satisfied for this particular case.

   \item[{\rm (v)}] If $M$ is
 of type  $\mathcal C_2 \oplus \mathcal C_4 \oplus \mathcal C_5 \oplus \mathcal C_6 \oplus\mathcal C_7 \oplus\mathcal C_8 \oplus\mathcal C_{11} \oplus\mathcal C_{12} = \mathcal C_2 \oplus  \mathcal C_4 \oplus \mathcal C
\oplus \mathcal E$,
then the structure is
harmonic for $n\neq 2$ if and only if
\begin{align*}
  \Ricac_{\mbox{\rm \footnotesize alt}}( X_{\zeta^{\perp}} , Y_{\zeta^{\perp}})  =
 &
   \langle  (\nabla^{\Lie{U}(n)}_{\zeta} \xi_{(11)})_{\zeta}  X , Y \rangle
 + n  \textstyle  \langle  \xi_{{(2)} \theta^\sharp } X, Y \rangle 
  -  d^* \eta  \langle \xi_{(11) \zeta}  X , Y \rangle
  \\
  &
    -\tfrac{n}{n-2}\,\textstyle d^* F(\zeta) \langle  \xi_ {(11)\zeta}  X,  \varphi Y \rangle 
  -\tfrac{4}{n-2}\textstyle  (\xi_{(7)X} \eta)(\xi_{(11)\zeta} Y)   
  \\
  &
+ \tfrac{4}{n-2}\textstyle  (\xi_{(7)Y} \eta)(\xi_{(11)\zeta}X) 
  +  \langle \xi_{(2)\xi_{\zeta} \zeta}  X , Y \rangle
  \\
  &
 + \tfrac{1}{4} (\xi_\zeta \eta \wedge \theta - \varphi \xi_\zeta \eta \wedge \varphi\theta)( X , Y), 
 \end{align*}
 for all $X,Y \in \mathfrak X(M)$, and it is satisfied the identity \eqref{cinco}.

 For $n=2$, the above mentioned type of  structure is
harmonic if and only if identities \eqref{acuatrocehor}  and \eqref{cinco} are satisfied for this particular case.

 \item[{\rm (vi)}] If $M$ is
 of type  $\mathcal C_2 \oplus  \mathcal C_4 \oplus  \mathcal C_9 \oplus  \mathcal C_{10} \oplus  \mathcal C_{11} \oplus  \mathcal C_{12}    = \mathcal C_2 \oplus  \mathcal C_4 \oplus \mathcal D
\oplus \mathcal E$, then the  structure
is harmonic  for $n \neq 2$ if and only if
 \begin{align*}
\hspace{1cm} \Ricac_{\mbox{\rm \footnotesize alt}}( X_{\zeta^{\perp}} , Y_{\zeta^{\perp}})  =
 &
    \langle  (\nabla^{\Lie{U}(n)}_{\zeta} \xi_{(11)})_{\zeta}  X , Y \rangle 
 + n   \langle \xi_{(2)\theta^\sharp}  X , Y \rangle 
  +  \langle \xi_{(2)\xi_{\zeta} \zeta}  X , Y \rangle
 \\
 &
 +\tfrac14   (\xi_{\zeta}  \eta\wedge \theta 
  -   \varphi \xi_{\zeta}  \eta\wedge \varphi \theta)  (X , Y),
 \end{align*}
 for all $X,Y \in \mathfrak X(M)$, and it is satisfied the identity \eqref{unocuatrode}.

 For $n=2$, the above mentioned type of structure is
harmonic if and only if identities \eqref{acuatrodehor}  and \eqref{unocuatrode} are satisfied for this particular case.

    \item[{\rm (vii)}]
    If $M$ is of type  $\mathcal C_1 \oplus \mathcal C_5 \oplus\mathcal C_9 \oplus\mathcal C_{11} \oplus\mathcal C_{12}$,
then the  structure is harmonic if and only if
\begin{align*}
  0 =
 &
 \langle  (\nabla^{\Lie{U}(n)}_{\zeta} \xi_{(11)})_{\zeta}  X , Y \rangle
+ \langle \xi_{(1)\xi_\zeta \zeta} X , Y \rangle
- d^* \eta \langle \xi_{(11)\zeta} X , Y\rangle,
\end{align*}
 for all $X,Y \in \mathfrak X(M)$, and it is satisfied  the identity 
\begin{align*}
\Ricac( \zeta) =
 &
  \tfrac{1}{n} (d d^*\eta)^\sharp_{\zeta^\perp}
   + (\zeta \lrcorner d \xi_\zeta \eta)^\sharp 
  - \tfrac{n-1}{n} d^*\eta \xi_{\zeta} \zeta 
+ 2  \xi_{(9)\xi_{\zeta} \zeta } \zeta 
 + \xi_{(11)\zeta} \xi_{(12)\zeta} \zeta.
    \end{align*}
    In particular, if the structure is of type   $\mathcal C_1 \oplus \mathcal C_5 \oplus\mathcal C_9$, the structure is harmonic if and only if $\Ricac (\zeta)=0$. Note that, for such a type, we always have $\Ricac_{\rm alt} ( X_{\zeta^\perp} , Y_{\zeta^\perp}) =0$, for all $X, Y \in \mathfrak X (M)$.
    \end{enumerate}
\end{proposition}
\begin{proof}
Parts  (i) and (ii) are  merely  particular cases  of Theorem \ref{harmclasscontact} (i) and (ii), respectively.
For part (iii) and $n\neq5$, from properties of the involved components of $\xi$,    Proposition  \ref{divergenciadoscero} and  Theorem \ref{characharmalmcontact},  we deduce that in this case
{\small \begin{align*} 
   \tfrac{n-5}2  d \theta_{\lcf \lambda^{2,0}\rcf} ( X , Y)=&
     3 \langle  (\nabla^{\Lie{U}(n)}_{\zeta} \xi_{(11)})_{\zeta}  X , Y \rangle
  + 3(n-3) \langle \xi_{(1) \theta^\sharp }  X , Y\rangle
  +3 \langle \xi_{(1)\xi_{(12)\zeta}  \zeta}  X , Y
  \rangle
  \\
  &
  +\tfrac34  \xi_{(12)\zeta}   \eta \wedge \theta   (  X , Y)
-\tfrac34   \varphi \xi_{(12)\zeta}   \eta \wedge \varphi \theta (  X , Y)
  - 3 d^*\eta  \langle \xi_{(11)\zeta}  X , Y \rangle
  \\
  &
    - d^*F(\zeta)  \langle \xi_{(11)\zeta} \varphi X , Y \rangle
   - 2 (\xi_{(7)X} \eta) (\xi_{(11)\zeta} Y)
 + 2 (\xi_{(7)Y} \eta) (\xi_{(11)\zeta} X),
  \end{align*}}
  \vspace{-8mm}
  
  {\small
  \begin{align*}
   {\tiny(n-5)} \langle  (\nabla^{\Lie{U}(n)}_{e_i} \xi_{(1)})_{ e_i}  X , Y \rangle = & 
    -(n-2) \langle  (\nabla^{\Lie{U}(n)}_{\zeta} \xi_{(11)})_{\zeta}  X , Y \rangle
  - \tfrac{(n-3)(n+1)}{2} \langle \xi_{(1) \theta^\sharp }  X , Y\rangle 
 \\
 &
  -(n-2) \langle \xi_{(1)\xi_{(12)\zeta}  \zeta}  X , Y
  \rangle
   -\tfrac{n-2}4   \xi_{(12)\zeta}   \eta \wedge \theta  (  X, Y)
   \\
   &
   +\tfrac{n-2}4   \varphi \xi_{(12)\zeta}   \eta \wedge \varphi \theta (  X , Y)
  + (n-2) d^*\eta  \langle \xi_{(11)\zeta}  X , Y \rangle
\\
&
    + d^*F(\zeta)  \langle \xi_{(11)\zeta} \varphi X , Y\rangle
    - 2 (\xi_{(7)X} \eta) (\xi_{(11)\zeta} Y)
 + 2 (\xi_{(7)Y} \eta) (\xi_{(11)\zeta} X).
  \end{align*}}
Then we take the first identity given in Lemma \ref{astricciacm1} and use properties of the involved components of $\xi$. In the resulting identity we replace  $d \theta_{\lcf \lambda^{2,0}\rcf}$ and $\langle  (\nabla^{\Lie{U}(n)}_{e_i} \xi_{(1)})_{ e_i}  X , Y \rangle$ by the expressions given above. The remaining second  required   identity and the particular case $n=5$ immediately follow from Theorem \ref{harmclasscontact} (i).

For part (iv) and $n\neq5$, from properties of the involved components of $\xi$,    Proposition  \ref{divergenciadoscero}  and  Theorem \ref{characharmalmcontact},  we deduce that in this case
{\small  \begin{align*} 
   \tfrac{n-5}6  d \theta_{\lcf \lambda^{2,0}\rcf} ( X , Y)=&
      \langle  (\nabla^{\Lie{U}(n)}_{\zeta} \xi_{(11)})_{\zeta}  X , Y\rangle
  + (n-3) \langle \xi_{(1) \theta^\sharp }  X , Y\rangle
  + \langle \xi_{(1)\xi_{(12)\zeta}  \zeta}  X , Y
  \rangle
  \\
  &
  +\tfrac14  \xi_{(12)\zeta}   \eta \wedge \theta   (  X , Y)
-\tfrac14   \varphi \xi_{(12)\zeta}   \eta \wedge \varphi \theta (  X , Y),
  \end{align*}}
  \vspace{-9mm}
  
  {\small
  \begin{align*}
   {\tiny(n-5)} \langle  (\nabla^{\Lie{U}(n)}_{e_i} \xi_{(1)})_{ e_i}  X, Y\rangle = & 
    -(n-2) \langle  (\nabla^{\Lie{U}(n)}_{\zeta} \xi_{(11)})_{\zeta}  X , Y\rangle
  - \tfrac{(n-3)(n+1)}{2} \langle \xi_{(1) \theta^\sharp }  X, Y\rangle 
 \\
 &
  -(n-2) \langle \xi_{(1)\xi_{(12)\zeta}  \zeta}  X, Y
  \rangle
   -\tfrac{n-2}4   \xi_{(12)\zeta}   \eta \wedge \theta  (  X , Y)
   \\
   &
   +\tfrac{n-2}4   \varphi \xi_{(12)\zeta}   \eta \wedge \varphi \theta (  X , Y).
   \end{align*}}
As before, we take the first identity given in Lemma \ref{astricciacm1} and use properties of the involved components of $\xi$. Then  we replace  $d \theta_{\lcf \lambda^{2,0}\rcf}$ and $\langle  (\nabla^{\Lie{U}(n)}_{e_i} \xi_{(1)})_{ e_i}  X , Y \rangle$ by the expressions given above. The remaining second required  identity and the particular case $n=5$ immediately follow from Theorem \ref{harmclasscontact} (ii).

For part (v) and $n\neq2$, from properties of the involved components of $\xi$,    Proposition  \ref{divergenciadoscero} and  Theorem \ref{characharmalmcontact},  we deduce that in this case
 {\small \begin{align*} 
   \tfrac{n-2}2  d \theta_{\lcf \lambda^{2,0}\rcf} ( X , Y)=&
       - d^*F(\zeta)  \langle \xi_{(11)\zeta} \varphi X , Y \rangle
   - 2 (\xi_{(7)X} \eta) (\xi_{(11)\zeta} Y)
 + 2 (\xi_{(7)Y} \eta) (\xi_{(11)\zeta} X),
 \end{align*}}
 \vspace{-8mm}
 
 {\small \begin{align*}
    \langle  (\nabla^{\Lie{U}(n)}_{e_i} \xi_{(2)})_{ e_i}  X , Y \rangle = & 
    - \langle  (\nabla^{\Lie{U}(n)}_{\zeta} \xi_{(11)})_{\zeta}  X , Y \rangle
  - \tfrac{n}{2} \langle \xi_{(2) \theta^\sharp }  X , Y\rangle 
  - \langle \xi_{(2)\xi_{(12)\zeta}  \zeta}  X , Y
  \rangle
  \\
  &
    -\tfrac{1}4   \xi_{(12)\zeta}   \eta \wedge \theta  (  X , Y)
   +\tfrac{1}4   \varphi \xi_{(12)\zeta}   \eta \wedge \varphi \theta (  X , Y)
  +  d^*\eta \langle \xi_{(11)\zeta}  X , Y \rangle
  \\
  &
  + \tfrac{1}{n-2} d^*F(\zeta) \langle \xi_{(11)\zeta} \varphi X , Y \rangle
     + \tfrac{2}{n-2} (\xi_{(7)X} \eta) (\xi_{(11)\zeta} Y)
 - \tfrac{2}{n-2} (\xi_{(7)Y} \eta) (\xi_{(11)\zeta} X).
   \end{align*}}
As in the previous cases, we take the first identity given in Lemma \ref{astricciacm1} and use properties of the involved components of $\xi$. Then,  in the resulting identity,  we replace  $d \theta_{\lcf \lambda^{2,0}\rcf}$ and $\langle  (\nabla^{\Lie{U}(n)}_{e_i} \xi_{(2)})_{ e_i}  X , Y \rangle$ by the expressions given above. The remaining second required identity  and the particular case $n=2$ immediately follow from Theorem \ref{harmclasscontact} (i).

For part (vi), from properties of the involved components of $\xi$,    Proposition  \ref{divergenciadoscero}  and  Theorem \ref{characharmalmcontact},  we deduce that in this case
$
   \tfrac{n-2}2  d \theta_{\lcf \lambda^{2,0}\rcf} ( X , Y)= 0,
$ and 
 {\small   \begin{align*}
    \langle  (\nabla^{\Lie{U}(n)}_{e_i} \xi_{(2)})_{ e_i}  X , Y \rangle = & 
    - \langle  (\nabla^{\Lie{U}(n)}_{\zeta} \xi_{(11)})_{\zeta}  X , Y \rangle
  - \tfrac{n-}{2} \langle \xi_{(2) \theta^\sharp }  X , Y\rangle 
  - \langle \xi_{(2)\xi_{(12)\zeta}  \zeta}  X , Y
  \rangle
  \\
  &
    -\tfrac{1}4   \xi_{(12)\zeta}   \eta \wedge \theta  (  X , Y)
   +\tfrac{1}4   \varphi \xi_{(12)\zeta}   \eta \wedge \varphi \theta (  X , Y).
   \end{align*}}
As before, we take the first identity given in Lemma \ref{astricciacm1} and use properties of the involved components of $\xi$.   Then  we replace  $d \theta_{\lcf \lambda^{2,0}\rcf}$ and $\langle  (\nabla^{\Lie{U}(n)}_{e_i} \xi_{(2)})_{ e_i}  X , Y \rangle$ by the expressions given above. The remaining second required  identity  and the particular case $n=2$ immediately follow from Theorem \ref{harmclasscontact} (ii).

For part (vii) and $n\neq2$, from properties of the involved components of $\xi$,   and  Theorem \ref{characharmalmcontact},  we deduce 
 $
    \langle  (\nabla^{\Lie{U}(n)}_{e_i} \xi_{(1)})_{ e_i}  X , Y \rangle =  0 
  $ and the first required identity.  The remaining second required  identity follows from  the second  identity given in Lemma \ref{astricciacm1}. Thus,  due to the properties of the components of $\xi$, we firstly obtain 
\begin{align*}
\Ricac( \zeta) =
 &
 (\nabla^{\Lie{U}(n)}_{e_i} \xi_{(5)})_{ e_i} \zeta -  (\nabla^{\Lie{U}(n)}_{e_i} \xi_{(9)})_{ e_i} \zeta - \xi_{(1) e_i } \xi_{(9) e_i} \zeta.
     \end{align*}  
Since,  by one hand, one has 
$$
 (\nabla^{\Lie{U}(n)}_{e_i} \xi_{(5)})_{ e_i} \zeta = \tfrac1{2n} (dd^*\eta)^\sharp - \tfrac1{2n} d d^*\eta(\zeta) \zeta,
$$
 and, by the other hand, it is satisfied  
\begin{align*}
(\nabla^{\Lie{U}(n)}_{e_i} \xi_{(9)})_{ e_i} \zeta  =
 &
 - (\nabla^{\Lie{U}(n)}_{e_i} \xi_{(5)})_{ e_i} \zeta -  (\nabla^{\Lie{U}(n)}_{\zeta} \xi_{(12)})_{ \zeta} \zeta + \tfrac{2n-1}{2n} d^*\eta \xi_{\zeta} \zeta  - \xi_{(9)\xi_{\zeta} \zeta} \zeta, 
     \end{align*}  
      by  Theorem \ref{characharmalmcontact}. 
  Finally, using Lemma \ref{dosdeeta}, we will obtain the second required identity.
  \end{proof}

Next we focus attention on some particular cases where the harmonicity of the structure will  require  only one condition in all dimensions   except for the dimension $5$, i.e. $n=2$. This is possible because it is  used some of the   identities obtained in Section \ref{foursection}. This is one of the applications of such identities. We write such cases as a corollary because they can be considered as consequences of previous Theorem and Proposition.
\begin{corollary} \label{onecondition} For a $2n+1$-dimensional almost contact metric manifold
$(M,\langle \cdot ,\cdot \rangle, \varphi , \zeta)$,  we have:
\begin{enumerate}
 \item[{\rm (i)}] If $M$ is of type 
   $\mathcal C_1 \oplus  \mathcal C_5 \oplus \mathcal C_6 \oplus \mathcal C_7 \oplus \mathcal C_8 \oplus \mathcal C_{12}$, then the structure is harmonic if and only if 
   \begin{align*}
\Ricac( \zeta) =&- (\zeta \lrcorner d \xi_\zeta \eta)^\sharp + \tfrac{n-1}{n} d^*\eta \xi_\zeta \zeta - 2  \xi_{(8)  \xi_\zeta \zeta  } \zeta +  d^*F(\zeta)  \varphi \xi_\zeta \zeta.    
     \end{align*}  
Note that, for such a type,  
$
\Ricac_{\mathrm{alt}}  (X_{\zeta^\perp} , Y_{\zeta^\perp}  )  =   \langle  \xi_{(1)  \xi_\zeta \zeta  } X, Y \rangle$ and $ (\nabla^{\Lie{U}(n)}_{e_i} \xi_{(1)})_{ e_i}  =0$.

 \item[{\rm (ii)}] If $M$ is of type 
   $\mathcal C_1 \oplus  \mathcal C_9 \oplus \mathcal C_{10} \oplus \mathcal C_{12}$, then the structure is harmonic if and only if 
   \begin{align*}
\Ricac( \zeta) =& (\zeta \lrcorner d \xi_\zeta \eta)^\sharp -  \xi_{(1) e_i}  \xi_{(10) e_i} \zeta + 2 \xi_{(9)\xi_\zeta \zeta} \zeta.  
     \end{align*}  
Note that, for such a type,  
$
\Ricac_{\mathrm{alt}}  (X_{\zeta^\perp}  , Y_{\zeta^\perp}  )  =   \langle  \xi_{(1)  \xi_\zeta \zeta  } X , Y \rangle$ and $ (\nabla^{\Lie{U}(n)}_{e_i} \xi_{(1)})_{ e_i}   =0$.

\item[{\rm (iii)}] If $M$ is of type 
   $\mathcal C_3 \oplus  \mathcal C_5 \oplus \mathcal C_6 \oplus \mathcal C_7 \oplus \mathcal C_8 \oplus \mathcal C_{12}$, then the structure is harmonic if and only if 
   \begin{align*}
\Ricac( \zeta) =&- (\zeta \lrcorner d \xi_\zeta \eta)^\sharp   - \langle \xi_{(8)} \zeta, \langle \xi_{(3) \cdot} e_j , \cdot \rangle \rangle e_j  
 - \langle \xi_{(7)} \zeta, \langle \xi_{(3) \cdot} e_j , \cdot \rangle \rangle e_j\\
 &
 + \tfrac{n-1}{n} d^*\eta \xi_\zeta \zeta - 2  \xi_{(8)  \xi_\zeta \zeta  } \zeta +  d^*F(\zeta)  \varphi \xi_\zeta \zeta.    
     \end{align*}  
Note that, for such a type, one  has always
$
\Ricac_{\mathrm{alt}}  (X_{\zeta^{\perp}} , Y_{\zeta^{\perp}} )  =    (\nabla^{\Lie{U}(n)}_{e_i} \xi_{(3)})_{ e_i}  X , Y \rangle  =0$.

 \item[{\rm (iv)}] If $M$ is of type 
   $\mathcal C_3 \oplus  \mathcal C_9 \oplus \mathcal C_{10} \oplus \mathcal C_{12}$, then the structure is harmonic if and only if 
   \begin{align*}
\Ricac( \zeta) =& (\zeta \lrcorner d \xi_\zeta \eta)^\sharp + 2 \xi_{(9)\xi_\zeta \zeta} \zeta.  
     \end{align*}  
Note that, for such a type,  one  has always
$
\Ricac_{\mathrm{alt}}  (X_{\zeta^{\perp}} , Y_{\zeta^{\perp}} )  = \langle  (\nabla^{\Lie{U}(n)}_{e_i} \xi_{(3)})_{ e_i}  X , Y \rangle  =0$.

 \item[{\rm (v)}] For $n \neq 2$, 
     if $M$ is of type 
   $\mathcal C_4 \oplus  \mathcal C_5 \oplus \mathcal C_6 \oplus \mathcal C_7 \oplus \mathcal C_8 \oplus \mathcal C_{12}$, then the structure is harmonic if and only if 
   \begin{align}\label{case456}
\Ricac( \zeta) =&- (\zeta \lrcorner d \xi_\zeta \eta)^\sharp - \tfrac{n-1}{4n} d^*\eta \theta^\sharp - \tfrac{2n-1}{2} \xi_{(8)  \theta^\sharp} \zeta - \tfrac{3(n-1)}{4n} d^*F(\zeta) \varphi  \theta^\sharp \\
& - \tfrac{2n-3}{2} \xi_{(7)  \theta^\sharp} \zeta + \tfrac{n-1}{n} d^*\eta \xi_\zeta \zeta - 2  \xi_{(8)  \xi_\zeta \zeta  } \zeta +  d^*F(\zeta)  \varphi \xi_\zeta \zeta. \nonumber
     \end{align}  
Note that, for such a type and dimension, 
$
\Ricac_{\mathrm{alt}}  (X_{\zeta^{\perp}} , Y_{\zeta^{\perp}} )  =    \langle  (\nabla^{\Lie{U}(n)}_{e_i} \xi_{(4)})_{ e_i}  X , Y \rangle =  d \theta_{\lcf \lambda^{2,0} \rcf } (X,Y) =0.
$
If $n=2$,  the structure is harmonic if and only if it is satisfied \eqref{case456} and 
$d \theta_{\lcf \lambda^{2,0} \rcf } =0$.

 \item[{\rm (vi)}] For $n \neq 2$, if $M$ is of type 
   $\mathcal C_4 \oplus  \mathcal C_9 \oplus \mathcal C_{10}  \oplus \mathcal C_{12}$, then the structure is harmonic if and only if 
   \begin{align}\label{case4910}
\Ricac( \zeta) =& (\zeta \lrcorner d \xi_\zeta \eta)^\sharp  + (n-1) \xi_{(9)  \theta^\sharp} \zeta  + (n-1)  \xi_{(10)  \theta^\sharp} \zeta +   2  \xi_{(9)  \xi_\zeta \zeta  } \zeta. 
     \end{align}  
Note that for such a type and dimension 
$
\Ricac_{\mathrm{alt}}  (X_{\zeta^{\perp}} , Y_{\zeta^{\perp}} )  =    \langle  (\nabla^{\Lie{U}(n)}_{e_i} \xi_{(4)})_{ e_i}  X , Y \rangle =  d \theta_{\lcf \lambda^{2,0} \rcf } (X,Y) =0.
$
If $n=2$,  the structure is harmonic if and only if it is satisfied \eqref{case4910} and 
$d \theta_{\lcf \lambda^{2,0} \rcf } =0$.

 \item[{\rm (vii)}]   If $M$ is of type
   $\mathcal C_1 \oplus  \mathcal C_5 \oplus \mathcal C_6 \oplus  \mathcal C_8 \oplus \mathcal C_9 \oplus    \mathcal C_{12}$,
then the  structure is harmonic if and only if
 \begin{align} 
\Ricac( \zeta) = 
  &
  - 2  (\nabla^{\Lie{U}(n)}_{e_i} \xi_{(9)})_{ e_i} \zeta 
   -  (\zeta \lrcorner d \xi _\zeta \eta)^\sharp     
   \label{case1568912}
     + \tfrac{n-1}{n} d^*\eta \xi_{\zeta}  \zeta   
     -  2    \xi_{(8)\xi_{\zeta}  \zeta} \zeta 
\\
&
+ d^*F(\zeta)  \varphi   \xi_{(12) \zeta } \zeta
  -   2   \xi_{(9)\xi_{\zeta}  \zeta} \zeta.  \nonumber
      \end{align}
Note that,  for such a  type, 
$
 (\nabla^{\Lie{U}(n)}_{e_i} \xi_{(1)})_{ e_i}  =0.
$

 \item[{\rm (viii)}]   If $M$ is of type
   $\mathcal C_3 \oplus  \mathcal C_5 \oplus \mathcal C_6 \oplus  \mathcal C_8 \oplus \mathcal C_9 \oplus    \mathcal C_{12}$,
then the  structure is harmonic if and only if it is satisfied \eqref{case1568912}.

Note that,  for such a  type, 
$
 (\nabla^{\Lie{U}(n)}_{e_i} \xi_{(3)})_{ e_i}  =0.
$

   \item[{\rm (ix)}] For $n \neq 2$, 
     if $M$ is of type
   $\mathcal C_4 \oplus  \mathcal C_5 \oplus \mathcal C_6 \oplus  \mathcal C_8 \oplus \mathcal C_9 \oplus    \mathcal C_{12}$,
then the  structure is harmonic if and only if
 \begin{align} 
\Ricac( \zeta) =  \nonumber 
  &
  - 2  (\nabla^{\Lie{U}(n)}_{e_i} \xi_{(9)})_{ e_i} \zeta 
   -  (\zeta \lrcorner d \xi _\zeta \eta)^\sharp     
   - \tfrac{n-1}{4n}    d^*\eta \theta^\sharp  
   - \tfrac{2n-1}{2}      \xi_{(8)\theta^\sharp} \zeta  
   \\
   &\label{case4568912}
   - \tfrac{3(n-1)}{4n}    d^*F(\zeta) \varphi\theta^\sharp   
     + \tfrac{n-1}{n} d^*\eta \xi_{\zeta}  \zeta   -  2    \xi_{(8)\xi_{\zeta}  \zeta} \zeta 
\\
&
+ d^*F(\zeta)  \varphi   \xi_{(12) \zeta } \zeta
  -   2   \xi_{(9)\xi_{\zeta}  \zeta} \zeta.  \nonumber
      \end{align}

Note that,  for such a type and  dimension, 
$
  \langle  (\nabla^{\Lie{U}(n)}_{e_i} \xi_{(4)})_{ e_i}  X , Y \rangle =  d \theta_{\lcf \lambda^{2,0} \rcf } (X,Y) =0.
$
If $n=2$,  the structure is harmonic if and only if it is satisfied \eqref{case4568912} and 
$d \theta_{\lcf \lambda^{2,0} \rcf }  =0$.
\end{enumerate}
\end{corollary}
\begin{proof}
For (i), (ii), (iii) and  (iv),  the first required identity  by Theorem \ref{characharmalmcontact} is reduced to   $ (\nabla^{\Lie{U}(n)}_{e_i} \xi_{(1)})_{ e_i}  =0$  in case of  (i) and (ii)  or  $ (\nabla^{\Lie{U}(n)}_{e_i} \xi_{(3)})_{ e_i}  =0$ in case of  (iii) and (iv),  which  follows from Proposition \ref{divergenciadoscero}. The respective required condition is just the second condition of Theorem \ref{harmclasscontact} (i), (ii), (iii) and  (iv), respectively,  applied to these particular cases.

For  (v) and  (vi), the first required  identity  by Theorem \ref{characharmalmcontact} is reduced to   $d \theta_{\lcf \lambda^{2,0}\rcf} =0$,  which  follows from Proposition \ref{divergenciadoscero} when  $n\neq 2$. The respective required condition is just the second condition of Theorem \ref{harmclasscontact} (iii) and  (iv), respectively,  applied to these cases.

The respective proofs for (vii), (viii) and (ix) are similar as for  the previous cases but now we have to take also $\xi_{(9)}$ into account. The respective required identity it follows from the second identities in Theorem \ref{characharmalmcontact} and in Lemma \ref{astricciacm1}.
\end{proof}

In the following corollaries some very special particular cases are displayed.
\begin{corollary}  \label{harmonicalways} For a $2n+1$-dimensional almost contact metric manifold
$(M,\langle \cdot ,\cdot \rangle, \varphi , \zeta)$,  we have:
 \begin{enumerate}
  \item[{\rm (i)}]
  If $M$ is of type 
 $\mathcal C_1 \oplus  \mathcal C_5 \oplus
  \mathcal C_6 \oplus  \mathcal C_7 \oplus  \mathcal C_8 $ or $\mathcal C_3 \oplus  \mathcal C_5 \oplus
  \mathcal C_6  $, the structure is harmonic if and only if 
  $ (d d^*\eta)_{|\zeta^\perp} = d (d^*F(\zeta)) \circ \varphi$. In particular,  structures of types $\mathcal{C}_6$,  $\mathcal C_1 \oplus  \mathcal C_5 $, $\mathcal C_1 \oplus  \mathcal C_7 
   \oplus  \mathcal C_8 $ or $\mathcal C_3 \oplus  \mathcal C_5$ are harmonic.
   
     \item[{\rm (ii)}]
  If $M$ is of type 
 $\mathcal C_1 \oplus  \mathcal C_9 \oplus
  \mathcal C_{10} $ or  $\mathcal C_3 \oplus  \mathcal C_9 \oplus
  \mathcal C_{10} $, the structure is harmonic if and only if $ (\nabla^{\Lie{U}(n)}_{e_i} \xi_{(9)})_{ e_i} \zeta =0$. In particular,  structures of type   $\mathcal C_1 \oplus  \mathcal C_{10}$ or  $\mathcal C_3 \oplus  \mathcal C_{10}$   are harmonic.

   \item[{\rm (iii)}]
  For $n \neq 2$, if $M$ is of type $\mathcal C_4$, then  the structure is harmonic.   
   \item[{\rm (iv)}]
  If $M$ is of type 
 $
  \mathcal C_{12} $, the structure is harmonic if and only if the one-form  $\xi_\zeta \eta$ is closed.
   \end{enumerate}
\end{corollary}
\begin{proof} 
For (i), the first identity of Theorem \ref{characharmalmcontact} is satisfied in this case because  
Proposition \ref{divergenciadoscero}. The second required identity in such a Theorem is equivalent 
to    $ (d d^*\eta)_{|\zeta^\perp} - d (d^*F(\zeta)) \circ \varphi =0$. This is  due to the identities 
$ (\nabla^{\Lie{U}(n)}_{e_i} \xi_{(5)})_{ e_i}  \eta = \frac{1}{2n} (d d^*\eta)_{|\zeta^\perp}$, 
$ (\nabla^{\Lie{U}(n)}_{e_i} \xi_{(6)})_{ e_i}  \eta = - \frac{1}{2n} (d (d^*F(\zeta)) \circ \varphi$,
Proposition \ref{divergencia}
and Lemma \ref{mainid}. The claims saying to be harmonic are deduced from the fact that $ (d d^*\eta)_{|\zeta^\perp} = d (d^*F(\zeta)) \circ \varphi =0$ in such cases.

For (ii), the first identity of Theorem \ref{characharmalmcontact} is satisfied  by  
Proposition \ref{divergenciadoscero}. The second required identity  is equivalent 
to  $ (\nabla^{\Lie{U}(n)}_{e_i} \xi_{(9)})_{ e_i} \zeta =0$. This is deduced from Proposition \ref{divergencia}.

Parts (iii) and (iv)  follow from  Proposition \ref{divergenciadoscero} and Lemma \ref{dosdeeta}, respectively.
\end{proof}

\begin{corollary}  \label{corharmonic} For a $2n+1$-dimensional almost contact metric manifold
$(M,\langle \cdot ,\cdot \rangle, \varphi , \zeta)$,  we have:
\begin{enumerate}
\item[{\rm (i)}]If the  structure is of type $  \mathcal C_5$,  
$C_{6}$,
 $ \mathcal C_7 \oplus \mathcal C_{8}$,
 or 
 $  \mathcal C_{10}$, then the Reeb vector
field $\zeta$ is  harmonic unit vector field.
 \item[{\rm (ii)}] For a conformally flat
manifold $(M,\langle \cdot ,\cdot \rangle)$ of dimension $2n+1$ with
$n>1$,  if an almost contact  structure compatible with
$\langle \cdot ,\cdot \rangle$ is of type
$ \mathcal C_1 \oplus \mathcal C_2 \oplus \mathcal C_5
\oplus \mathcal C_{6}\oplus \mathcal C_7 \oplus \mathcal C_{8}$, $
 \mathcal
C_1 \oplus \mathcal C_2 \oplus  \mathcal C_9 \oplus  \mathcal
C_{10}$,   $
 \mathcal
C_1 \oplus \mathcal C_5 \oplus  \mathcal C_9$ or $\mathcal C_3 \oplus  \mathcal C_5$,
 then it is harmonic.
 \end{enumerate}
\end{corollary}
\begin{proof}
Part (i) follows from Corollary \ref{corharmonic} and Theorem \ref{characharmalmcontact}.
Part (ii) is a consequence of the fact that, for those types, the harmonicity is equivalent to $\Ricac_{\mathrm{alt}} = 0$ and   
$\Ricac(\zeta)=0$. Such conditions are satisfied due to  the Weyl curvature vanishes (see Lemma \ref{astricciacm1}). 
\end{proof}


\subsection{Harmonicity  of almost
contact metric  structures  as a map} \label{subsfiveone}
Now, we focus our attention on studying harmonicity as a map  of almost
contact metric  structures,  $\sigma:M \to \mathcal{SO}(M) /
(\Lie{U}(n) \times 1)$. Results in that direction were already
obtained by Vergara-Díaz and Wood \cite{VergaraWood} and in \cite{GDMC4} for the type $\mathcal{C}_1 \oplus \dots \oplus \mathcal{C}_{10}$. We will
complete such results  for the  general type    $\mathcal{C}_1 \oplus \dots \oplus \mathcal{C}_{12}$. In next Lemma,
$s^{*} = \Ricac (e_i, e_i)$ will denote the {\it $*$-scalar  curvature}. If $\Ricac (X,Y) = \frac{1}{2n}
s^{*} (\langle X,Y \rangle - \eta (X) \eta (Y))$, the almost
contact metric manifold is said to be {\it weakly-ac-Einstein}. If
$s^{*}$ is constant, a weakly-ac-Einstein manifold is called {\it
ac-Einstein}.

In Riemannian geometry, it is satisfied $2 d^* \Ric + d s =0$,
 where $s$ is the scalar curvature. The $*$-analogue in almost
 contact metric geometry does not hold in general.
\begin{lemma} \label{id:genera}
 For almost contact metric manifolds of type
$\mathcal C_1 \oplus \ldots \oplus \mathcal C_{11} \oplus \mathcal C_{12}$, we have
\begin{align*}
 d^* (\Ricac)^t(X) + \tfrac12 ds^{*} (X)   = &
         \langle R_{e_i , X} ,    \xi_{\varphi e_i} \varphi\rangle
        -   (n-1) \Ricac ( X,   \theta^\sharp)   + 2 \langle \Ricac, \xi^\flat_X \rangle
        \\
        &
        -  d^* F (\zeta) \Ricac (\zeta , \varphi X) - \Ricac(X, \xi_\zeta \zeta),
\end{align*}
 where $(\Ricac)^t(X,Y) = \Ricac(Y,X)$ and $\xi_X^{\flat} (Y,Z) =
\langle \xi_X Y, Z \rangle$. In particular, if the manifold is
weakly-ac-Einstein, then
\begin{align*}
(n - 1) d s^*  =  \; & 2n  \langle R_{e_i , \cdot \, } ,    \xi_{\varphi e_i} \varphi\rangle
        -  (n-1) s^* \theta
        +  \left( \tfrac{n-1}{n} s^* d^* \eta \ - 
 2  \langle R_{e_i , \zeta} ,    \xi_{\varphi e_i} \varphi\rangle\right)  \eta.
\end{align*}
\end{lemma}

\begin{proof}
Note that $ (\Ricac)^t(X,Y) = \tfrac12   \langle R_{e_i,\varphi
e_i} Y,  \varphi X\rangle$.  Then, we get
\begin{align*}
d^* (\Ricac)^t(X)  = & - \textstyle \sum_{j=1}^{2n} (\nabla_{e_j} (\Ricac)^t) ( e_j ,X)  - (\nabla_{\zeta} (\Ricac)^t) ( \zeta ,X)
\\
=& \textstyle \sum_{j=1}^{2n} (-\tfrac12  e_j \langle R_{e_i, \varphi e_i} X ,\varphi e_j
\rangle
   + \tfrac12 \langle R_{e_i, \varphi e_i}  \nabla_{e_j} X , \varphi e_j \rangle
   + \tfrac12 \langle R_{e_i,\varphi e_i}   X,   \varphi \nabla_{e_j} e_j \rangle)
   \\
   &
   + \tfrac12 \langle R_{e_i,\varphi e_i}   X,   \varphi \nabla_{\zeta} \zeta \rangle
   \end{align*}
   \begin{align*}
   \hspace{1cm} =& \textstyle \sum_{j=1}^{2n} (- \tfrac12  \langle ( \nabla_{e_j} R)_{e_i,\varphi  e_i} X ,\varphi e_j \rangle
      -  \langle R_{\nabla_{e_j} e_i, \varphi e_i}   X , \varphi e_j \rangle
        - \tfrac12 \langle R_{e_i,\varphi e_i}   X,   (\nabla_{e_j}\varphi) e_j
        \rangle)
        \\
        &
        - \Ricac(X, \xi_\zeta \zeta).
\end{align*}
Since  $\textstyle \sum_{j=1}^{2n} (\nabla_{e_j} \varphi)(e_j) = 2 \varphi \xi_{(4) e_j} e_j
+ d^* F (\zeta) \zeta = (n-1) \varphi \theta^\sharp +  d^* F (\zeta) \zeta$, by symmetry properties of $R$, it follows
\begin{align*}
d^* (\Ricac)^t(X)  = & - \tfrac12  \langle ( \nabla_{e_j} R)_{X ,
\varphi e_j}  e_i, \varphi e_i\rangle
      +  \langle R_{X , e_j} e_i ,    \nabla_{\varphi e_j} \varphi e_i) \rangle
        - \tfrac{ n-1}{2} \langle R_{e_i,\varphi e_i}   X,  \varphi \theta^\sharp\rangle
        \\
        & - \tfrac12 d^* F (\zeta) \langle R_{e_i,\varphi e_i}   X,  \zeta
        \rangle  - \Ricac(X, \xi_\zeta \zeta).
\end{align*}

Using second Bianchi's identity and taking into account
$$
 \langle R_{X , e_j} e_i ,    \nabla_{\varphi e_j} \varphi e_i) \rangle =
   \langle R_{X , e_j} e_i ,    \nabla^{\Lie{U}(n)}_{\varphi e_j} \varphi e_i) \rangle
 -  \langle R_{X , e_j} e_i ,    \xi_{\varphi e_j} \varphi e_i)
 \rangle,
$$
we get
\begin{align*}
d^* (\Ricac)^t (X)  = & - \tfrac14  \langle ( \nabla_{X} R)_{e_j,
\varphi e_j }  e_i,\varphi e_i\rangle
      -  \langle R_{X , e_j} ,    \xi_{\varphi e_j} \varphi
      \rangle
        - (n-1) \Ricac ( X,   \theta^\sharp )
        \\
        &
        - d^* F (\zeta) \Ricac (\zeta , \varphi X) 
        - \Ricac(X, \xi_\zeta \zeta).
\end{align*}
Note that
$
 \langle R_{X , e_j} e_i ,    \nabla^{\Lie{U}(n)}_{\varphi e_j} \varphi e_i) \rangle =
  \langle R_{X , e_j} e_i , e_k \rangle
     \langle \nabla^{\Lie{U}(n)}_{\varphi e_j} \varphi e_i , e_k \rangle =0,
$
because it is a scalar product of a skew-symmetric matrix by a
symmetric matrix.
Finally, it is obtained
\begin{eqnarray} \label{uno}
\qquad 2d^* (\Ricac)^t(X) & = & - \tfrac12  \langle ( \nabla_{X}
R)_{e_j, \varphi e_j }  e_i,\varphi e_i\rangle
      - 2 \langle R_{X , e_j} ,    \xi_{\varphi e_j} \varphi
      \rangle \\
      && \nonumber
        -  2(n-1) \Ricac ( X,   \theta^\sharp ) - 2 d^* F (\zeta) \Ricac (\zeta , \varphi X)-2 \Ricac(X, \xi_\zeta \zeta)
.
\end{eqnarray}

Furthermore, ones obtains  $ ds^{*} (X) = \tfrac12
X \langle R_{e_i,\varphi e_i} e_j, \varphi e_j \rangle$. Hence
\begin{eqnarray*}
 ds^{*} (X) & = & \tfrac12   \langle (\nabla_X R)_{e_i,\varphi e_i} e_j, \varphi e_j \rangle
  + 2 \langle R_{e_i, \varphi e_i} e_j, \nabla_X \varphi e_j \rangle.
\end{eqnarray*}
But we also have  that
\begin{eqnarray*}
 \langle R_{e_i, \varphi e_i} e_j, \nabla_X \varphi e_j \rangle & = &  \langle R_{e_i, \varphi e_i} e_j, e_k \rangle \langle
 \nabla^{\Lie{U}(n)}_X \varphi e_j , e_k \rangle - \langle R_{e_i,\varphi e_i} e_j, e_k \rangle \langle
 \xi_X \varphi e_j , e_k \rangle \\
 & = &\textstyle \sum_{j=1}^{2n} (\langle R_{e_i,\varphi e_i} e_j, \varphi \xi_X e_j \rangle + \eta ( \xi_X \varphi e_j )
 \langle R_{e_i,\varphi e_i} \zeta , e_j \rangle)
\\
&=&  
 \textstyle \sum_{j=1}^{2n} (2 \Ricac (e_j , \xi_X e_j) - 2 \eta ( \xi_X \varphi e_j )
\Ricac  (\zeta , \varphi e_j))
  \\
&=&
 2 \textstyle \sum_{j=1}^{2n}  \Ricac  (e_j , \xi_X e_j) + 2    \Ricac  (\zeta , \xi_X \zeta )
\\
&=& 2 \langle \Ricac   , \xi_X  \rangle.
\end{eqnarray*}
Thus,
 \begin{equation} \label{dos}
  ds^{*}  (X)  =  \tfrac12   \langle (\nabla_X
R)_{e_i,\varphi e_i} e_j, \varphi e_j \rangle + 4 \langle \Ricac ,
\xi^\flat_X \rangle.
 \end{equation}
 From \eqref{uno} and \eqref{dos}, the required identity is
 obtained. 
 
 The claim relative to the  case of weakly ac-Einstein follows by a straightforward way. Note that
$
2n \; (d^* \Ricac)^t  = 2n \; d^* \Ricac = - d s^* - s^* \xi_\zeta \eta  + (d s^* (\zeta) 
- s^* d^* \eta) \eta$  in such a case.
\end{proof}
\begin{remark}{\rm It is interesting to compare Lemma \ref{id:genera} with the analogous  result for almost Hermitian geometry given in \cite{GDMC}. Thus, for an almost Hermitian manifold $(M, J, \langle \cdot, \cdot \rangle)$ of dimension $2n$, one has
$$
d^* (\Ricac)^t(X) +\tfrac12  d s^*(X) =  \langle R_{e_i, X} , \xi_{J e_i} J \rangle - (n-1) \Ricac(X, \theta^\sharp)+ 2 \langle  \Ricac , \xi^\flat_X \rangle,
$$
for all $X \in \mathfrak X (M)$. In particular, if the manifold is weakly $*$Einstein, then
$$
(n-1) d s^*  = 2n  \langle R_{e_i, \cdot \,} , \xi_{J e_i} J \rangle - (n-1) \theta.
$$
}
\end{remark}

Next we focus
our attention on  conditions relative to harmonicity as a map of  almost contact metric structures.
\begin{theorem} \label{mapharmcontact}
 For an $2n+1$-dimensional  almost contact metric  manifold $(M,\langle \cdot
,\cdot \rangle, \varphi, \zeta)$,  we have:
\begin{enumerate}
\item[{\rm (i)}] If $M$ is of type  $\mathcal{C}_1 \oplus
\mathcal{C}_2 \oplus \mathcal{C}_4 \oplus \mathcal{C}_5 \oplus \mathcal{C}_6\oplus
\mathcal{C}_7\oplus \mathcal{C}_8 \oplus
\mathcal{C}_{11} \oplus \mathcal{C}_{12}$, then the structure is a
harmonic map  if and only if it is a harmonic structure and, for all $X\in \mathfrak X (M)$,
 \begin{align*}
\quad \;    d^* (\Ricac)^t (X)+ \tfrac12  ds^{*}(X)  = &  \Ric ( X,   \theta^\sharp)  - n \Ricac ( X,   \theta^\sharp)       
            + 2 \langle \Ricac, \xi^\flat_X \rangle
        \\
       &
      -  d^* F (\zeta) \Ricac (\zeta , \varphi X)          - \Ricac(X, \xi_\zeta \zeta)
   -  \langle R_{\zeta,X}\zeta, \theta^\sharp\rangle 
         \\
        &
         + 3  \textstyle \sum_{i=1}^{2n}  \langle R_{e_i,X} \zeta ,  \xi_{e_i} \zeta \rangle
         +   \langle R_{\zeta,X}, \xi_\zeta \rangle. 
\end{align*}

\item[{\rm (ii)}] If $M$ is of type  $\mathcal{C}_1 \oplus
\mathcal{C}_2 \oplus \mathcal{C}_4 \oplus \mathcal{C}_9 \oplus \mathcal{C}_{10} \oplus
\mathcal{C}_{11} \oplus \mathcal{C}_{12}$, then the  structure is a
harmonic map  if and only if it is a harmonic structure and, for all $X\in \mathfrak X (M)$, 
\begin{align*}
\hspace{1cm}    d^* (\Ricac)^t(X) +\tfrac12  ds^{*}(X)  = &  \Ric ( X,   \theta^\sharp)  - n \Ricac ( X,   \theta^\sharp)  
     + 2 \langle \Ricac, \xi^\flat_X \rangle         - \Ricac(X, \xi_\zeta \zeta)
     \\
     &
     -  \langle R_{\zeta,X}\zeta, \theta^\sharp\rangle
              + \textstyle \sum_{i=1}^{2n}      \langle R_{e_i,X} \zeta ,  \xi_{e_i} \zeta \rangle
       + \langle R_{\zeta , X}  , \xi_{\zeta}
     \rangle.                
\end{align*}

 \item[{\rm (iii)}] If $M$ is of type  $\mathcal{C}_1 \oplus
\mathcal{C}_2 \oplus \mathcal{C}_5\oplus \mathcal{C}_6\oplus
\mathcal{C}_7\oplus \mathcal{C}_8$, then the  structure is a
harmonic map  if and only if  $\Ricac$ is symmetric and $(d^*
\Ricac +\tfrac12  ds^{*})(X) =   3   \langle R_{e_i,X} \zeta, \xi_{e_i} \zeta \rangle$,  for all $X\in \mathfrak X (M)$. Furthermore, one has the particular cases:
\begin{enumerate}
\item[$(\mathrm{a})$] If the manifold is nearly-K-cosymplectic $(\mathcal{C}_1)$, then  structure is a harmonic map. 
\item[$(\mathrm{b})$]
 If  the  manifold is
weakly-$ac$-Einstein,  the  structure is a
harmonic map  if and only if  it is  satisfied
 $
(n-1) d s^{*} = 4n \langle R_{e_i, \cdot \, } \zeta , \xi_{e_i} \zeta\rangle + \left(\tfrac{n-1}{n} s^* d^*\eta +  6 \langle R_{e_i, \zeta\, }\zeta,  \xi_{e_i} \zeta \rangle  \right)\eta. 
 $
  \end{enumerate}
 
 \item[{\rm (iv)}] If $M$ is of type  $\mathcal{C}_1 \oplus
\mathcal{C}_2  \oplus \mathcal{C}_9 \oplus \mathcal{C}_{10} $, then the  structure is a
harmonic map  if and only if it is a harmonic structure and,  for all $X\in \mathfrak X (M)$, 
\begin{align*}
    d^* (\Ricac)^t (X) + \tfrac12  ds^{*}(X)  = &        
             2 \langle \Ricac (\zeta), \xi_X \zeta\rangle   +   \langle R_{e_i,X} \zeta ,  \xi_{e_i} \zeta \rangle.
     \end{align*}
 In particular,  
 if  the  manifold is
weakly-$ac$-Einstein, then the  structure is a
harmonic map  if and only if  it  is a harmonic structure and   satisfied
 $
(n-1) d s^{*} = 2n \langle R_{e_i, \cdot \, } \zeta , \xi_{e_i} \zeta\rangle   -2  \langle R_{e_i, \zeta \, } \zeta , \xi_{e_i} \zeta\rangle  \, \eta. 
 $


\item[{\rm (v)}] If $M$  is of type 
 $\mathcal{C}_3 \oplus \mathcal{C}_4
\oplus \mathcal{C}_5 \oplus \mathcal{C}_{6} \oplus \mathcal{C}_7 \oplus \mathcal{C}_{8}  \oplus \mathcal{C}_{11} \oplus \mathcal{C}_{12}$,
  then the  structure  is a harmonic map
 if and only if  it is a harmonic structure and, for all $X \in \mathfrak X (M)$,  satisfied
 \begin{align*}
  \;\;\; \;\; \;  \;\;\; \;\; \;  d^* (\Ricac)^t (X) + \tfrac12 ds^{*}(X)  = & -   (n-1)\Ricac(X,\theta^\sharp) 
 + 2 \langle \Ricac , \xi_X\rangle  - d^* F(\zeta) \Ricac (\zeta , \varphi X)
  \\
  &
        - \Ricac(X, \xi_\zeta \zeta)
  - \textstyle \sum_{i=1}^{2n}    \langle R_{e_i,X} \zeta, \xi_{e_i} \zeta \rangle
   -             \langle R_{\zeta,X}, \xi_\zeta \rangle.
  \end{align*}

\item[{\rm (vi)}] If $M$  is of type 
 $\mathcal{C}_3 \oplus \mathcal{C}_4
\oplus \mathcal{C}_5 \oplus \mathcal{C}_{6} \oplus \mathcal{C}_7 \oplus \mathcal{C}_{8}  $,
  then the   structure  is a harmonic map
 if and only it is a harmonic  structure  and, for all $X \in \mathfrak X (M)$, it is 
 satisfied
 \begin{align*}
 d^* (\Ricac)^t (X) + \tfrac12 ds^{*}(X)  = & -   (n-1)\Ricac(X,\theta^\sharp) 
 + 2 \langle \Ricac , \xi_X\rangle  
  \\
  & - d^* F(\zeta) \Ricac (\zeta , \varphi X)  -  \langle R_{e_i,X} \zeta, \xi_{e_i} \zeta \rangle.
  \end{align*}
 In particular,  
 if  the  manifold is
weakly-$ac$-Einstein, then the structure is a
harmonic map  if and only if  the structure is harmonic and  it is  satisfied
 $$
(n-1) d s^{*} =   -  (n-1) \theta - 2n \langle R_{e_i, \cdot \, } \zeta, \xi_{e_i} \zeta\rangle  + \left( \tfrac{n-1}{n} s^* d^*\eta  + 2 \langle R_{e_i, \zeta} \zeta, \xi_{e_i} \zeta\rangle \right) \,  \eta. 
 $$

   \item[{\rm (vii)}] If $M$  is of type 
 $\mathcal{C}_3 \oplus \mathcal{C}_4
\oplus \mathcal{C}_9 \oplus \mathcal{C}_{10} \oplus \mathcal{C}_{11} \oplus \mathcal{C}_{12}$,
  then the   structure  is a harmonic map
 if and only if the structure  is  harmonic  and, for all $X \in \mathfrak X (M)$,
 \begin{align*}
 \;\; \; \;d^* (\Ricac)^t (X) + \tfrac12 ds^{*}(X)  = & -   (n-1)\Ricac(X,\theta^\sharp) + 2 \langle \Ricac , \xi_X\rangle         - \Ricac(X, \xi_\zeta \zeta)
 \\
 &
  - 3 \textstyle \sum_{i=1}^{2n}      \langle R_{e_i,X} \zeta, \xi_{e_i} \zeta  \rangle
   -\langle R_{\zeta, X}, \xi_\zeta \rangle.
  \end{align*}

  \item[{\rm (viii)}] If $M$  is of type 
 $\mathcal{C}_3 \oplus \mathcal{C}_4
\oplus \mathcal{C}_9 \oplus \mathcal{C}_{10}$,
  then the   structure  is a harmonic map
 if and only if  it is a harmonic structure and, for all $X \in \mathfrak X (M)$,
 $$
 \; \; \;d^* (\Ricac)^t (X) + \tfrac12 ds^{*}(X)  = -   (n-1)\Ricac(X,\theta^\sharp) + 2 \langle \Ricac , \xi_X\rangle  -3     \langle R_{e_i,X} \zeta , \xi_{e_i} \zeta  \rangle.
$$

 In particular:
  \begin{enumerate}
  \item[{\rm (a)$^*$}] If $\Ricac$ is symmetric,  then the  structure is a harmonic map if and only if   $\xi_{\theta^\sharp}=0$  and
   $ d^{*} \Ricac + \tfrac12 ds^{*} = -  (n-1) \theta^\sharp   \lrcorner 
  \Ricac - 3    \langle R_{e_i,\cdot \;}\zeta , \xi_{e_i} \zeta   \rangle.$  Furthermore, if the manifold is weakly-$ac$-Einstein, then the  structure
  is a harmonic map if and only if $\xi_{\theta^\sharp}=0$ and
   $
    (n-1)ds^{*} =  - (n-1) s^{*}  \theta - 6n \langle R_{e_i, \,\cdot \,} \zeta , \xi_{e_i} \zeta   \rangle
    + 6 \langle R_{e_i, \,\zeta \, } \zeta , \xi_{e_i} \zeta  \rangle  \eta   .
   $
      
   \item[{\rm (b)$^*$}] If the  structure is of type $\mathcal{C}_3 \oplus \mathcal{C}_9 \oplus \mathcal C_{10}$, then the structure is a
   harmonic map if and only if $\Ricac$ is symmetric and 
   $
    d^* \Ricac + \tfrac12 ds^{*}= - 3    \langle R_{e_i,X}\zeta , \xi_{e_i} \zeta \rangle.
   $
   Furthermore, if the manifold is also weakly-$ac$-Einstein, then
   the  structure is a
   harmonic map if and only if 
   $
   (n-1)ds^{*} =   - 6n \langle R_{e_i, \,\cdot \,} \zeta , \xi_{e_i} \zeta   \rangle
    + 3 \langle R_{e_i, \,\zeta \, }\zeta , \xi_{e_i} \zeta \rangle  \eta.
   $
   
  \end{enumerate}
\end{enumerate}
\end{theorem}
\begin{proof} 
For (i),   if the structure  of type $\mathcal{C}_1 \oplus
\mathcal{C}_2 \oplus \mathcal{C}_4 \oplus   \mathcal{C}_5\oplus \mathcal{C}_6\oplus
\mathcal{C}_7\oplus \mathcal{C}_8 \oplus \mathcal{C}_{11} \oplus \mathcal{C}_{12}$, then by properties of the components of $\xi$ the identity in Lemma \ref{id:genera} can be expressed as
\begin{align*}
    d^* (\Ricac)^t(X) +\tfrac12  ds^{*}(X)  = &-\textstyle \sum_{i=1}^{2n+1} \langle R_{e_i , X} , \xi_{e_i} \rangle  +  \Ric ( X,   \theta^\sharp)  - n \Ricac ( X,   \theta^\sharp)  
    \\
    &
    + 3\textstyle  \sum_{i=1}^{2n}   \langle R_{e_i , X} \zeta , \xi_{e_i}\zeta  \rangle   
       + \langle R_{\zeta , X}  , \xi_{\zeta}
     \rangle     -  \langle R_{\zeta,X}\zeta, \theta^\sharp\rangle
            + 2 \langle \Ricac, \xi^\flat_X \rangle
            \\
            &
        -  d^* F (\zeta) \Ricac (\zeta , \varphi X)         - \Ricac(X, \xi_\zeta \zeta)
,
\end{align*}
for all $X \in  \mathfrak X(M)$. From this identity and taking Proposition  \ref{harmmap2} into account, part (i) follows.

For (ii), in this case, if the structure  of type $\mathcal{C}_1 \oplus
\mathcal{C}_2 \oplus \mathcal{C}_4 \oplus \mathcal{C}_9\oplus
\mathcal{C}_{10} \oplus \mathcal{C}_{11} \oplus \mathcal{C}_{12}$,   the identity in Lemma \ref{id:genera} can be expressed as
\begin{align*}
    d^* (\Ricac)^t (X) +\tfrac12  ds^{*}(X)  = &- \textstyle \sum_{i=1}^{2n+1} \langle R_{e_i , X} , \xi_{e_i} \rangle  +  \Ric ( X,   \theta^\sharp)  - n \Ricac ( X,   \theta^\sharp)  
    \\
   &         - \Ricac(X, \xi_\zeta \zeta)
    + \textstyle \sum_{i=1}^{2n}   \langle R_{e_i , X} \zeta , \xi_{e_i}\zeta 
     \rangle   
     \\
     &
       + \langle R_{\zeta , X}  , \xi_{\zeta}
     \rangle     -  \langle R_{\zeta,X}\zeta, \theta^\sharp\rangle
            + 2 \langle \Ricac, \xi^\flat_X \rangle,
\end{align*}
for all $X \in  \mathfrak X(M)$. From this identity,  using Proposition  \ref{harmmap2}, part (ii) follows.
\vspace{1mm}

For (iii), it follows as particular case of (i) noting that some summands vanish.
The particular case of structure of type $\mathcal{C}_1$ is derived from Corollary \ref{harmonicalways} (i) and Proposition \ref{pro:skew}. The another mentioned particular case,  of weakly ac-Einstein,  follows from the fact that $\Ricac$ is symmetric in such a case and by using the second identity given in Lemma \ref{id:genera}.  

For (iv), we firstly note that if we have a harmonic structure of type  $\mathcal{C}_1 \oplus
\mathcal{C}_2  \oplus \mathcal{C}_9\oplus
\mathcal{C}_{10}$, then $\Ricac_{\mathrm{alt}} (X_{\zeta^\perp}, Y_{\zeta^\perp}) =0$. This implies 
$\langle \Ricac , \xi_{X}
 \rangle = \langle \Ricac(\zeta) , \xi_{X} \zeta
 \rangle$. Then applying  the identity in (ii) to this particular case, it is deduced the required equality in (iv).
 The situation for   weakly ac-Einstein manifolds  is deduced by using the second identity given in Lemma \ref{id:genera}.  

The remaining parts (v), (vi),(vii) and (viii) are particular cases of the first four ones. They   follow by using  similar arguments as the previous ones.
 \end{proof}

\begin{remark} {\rm
See  the analogous result in \cite[Theorem 4.11, page 455]{GDMC}.}
 \end{remark}

\subsection{Harmonicity of the Reeb vector field as unit vector field}

Now we analyze conditions for the harmonicity of the Reeb vector field as unit vector field in case of harmonic structure. From  Theorem \ref{characharmalmcontact} and   \eqref{lapstaten}, the condition  $\nabla^* \nabla \zeta = - \xi_{e_i} \xi_{e_i} \zeta$ is equivalent to the condition $\,(\nabla^{\Lie{U}(n)}_{e_i} \xi)_{e_i} \zeta +   \xi_{\xi_{e_i} e_i} \zeta
   =0$ which is the second condition to characterize the harmonicity of the structure. Thus, if  the almost contact metric structure is of type $\mathcal{C}_5 \oplus \dots \oplus \mathcal{C}_{12}$, then the structure is harmonic if and only if the Reeb vector field is harmonic  unit vector field. In general, this equivalence is not true. However,  in some cases,  harmonicity of the structure implies harmonicity of the Reeb vector field. Results in this regard are the following ones.
\begin{proposition} \label{reebharm}
If we have a manifold with harmonic  almost contact metric structure  of type   $\mathcal{C}_1 \oplus \mathcal{C}_2 \oplus \mathcal{C}_5 \oplus \mathcal{C}_6 \oplus \mathcal{C}_7 \oplus \mathcal{C}_8 \oplus \mathcal{C}_x$ or 
 $\mathcal{C}_3 \oplus \mathcal{C}_4 \oplus \mathcal{C}_9 \oplus \mathcal{C}_{10} \oplus \mathcal{C}_x$ or  
  $\mathcal{C}_1 \oplus \mathcal{C}_5 \oplus \mathcal{C}_6 \oplus \mathcal{C}_7 \oplus \mathcal{C}_8 \oplus \mathcal{C}_9 \oplus \mathcal{C}_x$ or  $\mathcal{C}_1 \oplus \mathcal{C}_2 \oplus \mathcal{C}_3 \oplus \mathcal{C}_4 \oplus  \mathcal{C}_x$ or $\mathcal{C}_3 \oplus \mathcal{C}_5 \oplus \mathcal{C}_6  \oplus \mathcal{C}_{x}$, where $x = 11$ or  $12$, then the Reeb vector field $\zeta$ is harmonic as unit vector field. 
\end{proposition}
\begin{proof} It is an immediate consequence of Theorem \ref{characharmalmcontact}. 
  and Lemma \ref{xixizeta1}.
\end{proof}

\begin{example}[\bf The hyperbolic space]
{\rm The following example has been already considered in \cite{ChineaGonzalezDavila,GDMC4,FMC6}.  Let $\mathcal{H} = \{ (x_1, \dots , x_{2n+1}) \in \mathbb{R}^{2n+1} \; | \; x_1 >0\}$ be the $(2n+1)$-dimensional hyperbolic space  with the Riemannian metric
  $$
\textstyle  \langle \cdot , \cdot  \rangle = \frac1{c^2x_1^2} \left( dx_1 \otimes dx_1 + \ldots +   dx_{2n+1} \otimes dx_{2n+1} \right).
  $$
With respect to this metric,   $\{ E_1, \ldots , E_{2n+1}\}$ is an orthonormal frame field, where $
E_i = c x_1 \frac{\partial}{\partial x_i}$,  $i=1, \ldots , 2n+1$.  For the Lie brackets, one has $[E_1, E_j] = cE_j$, $j=2 , \ldots , 2n+1$. The remaining Lie brackets relative to this frame are zero.   
The corresponding metrically equivalent coframe is  $\{ e_1, \ldots , e_{2n+1}\}$, where  $e_i = \tfrac{1}{ cx_1} dx_i$. Note that 
$ de_i =  - ce_1  \wedge e_i$,  $i=1, \ldots , 2n+1$.

The almost contact metric structure $(\varphi = \sum_{i,j=1}^{2n+1} \varphi^i_j e_j \otimes E_i, \zeta , \eta, \langle \cdot , \cdot \rangle)$ is considered in \cite{ChineaGonzalezDavila}. The functions $\varphi^i_j$ are constant, $n\geq 2$ and  
$\zeta = \sum_{i=1}^{2n+1}   x_1 k_i \frac{\partial}{\partial x_i} =  \sum_{i=1}^{2n+1} \frac{k_i}{c} E_i$, being $k_i=$ constant  and $k_1^2 + \ldots + k_{2n+1} ^2 = c^2$.
The one-form $\eta$ and the  fundamental form $F$ 
are given by 
$
\textstyle \eta = \sum_{i=1}^{2n+1}   \frac{k_i}{c} e_i,  \; \; F = \sum_{i,j=1}^{2n+1} \varphi^i_j e_i \wedge e_j. 
$
Their exterior derivatives are expressed as
$ 
d \eta = - c e_1 \wedge \eta$,  $dF = - 2 c e_1 \wedge F.
$
Hence $d \eta = \xi_\zeta \eta \wedge \eta \in \lcf \lambda^{1,0} \rcf \cong \mathcal{C}_{12}$, where $\xi_\zeta \eta = k_1 \eta - c e_1$, and $dF = 2 (k_1 \eta - c e_1) \wedge F - 2 k_1 \eta  \wedge F \in \lcf \lambda^{1,0} \rcf \wedge F + \mathbb{R} \eta \wedge F \cong \mathcal{C}_{4} \oplus \mathcal{C}_{5}$. Since $dF_{\lcf \lambda^{1,0} \rcf \wedge F} =  
\theta  \wedge F$ and $dF_{\mathbb{R} \eta  \wedge F} = - \frac{d^*\eta}{n} \eta \wedge F$, one has 
$\langle \cdot \lrcorner dF , F \rangle = 2 (n-1) \theta - 2 d^*\eta \, \eta$. Taking this into account we obtain 
$$
\textstyle 
(n-1) \theta  =    
2 (n-1) \xi_\zeta \eta, \quad d^* \eta = 2n k_1. 
$$ 

Now, by using the Lie brackets described above, one can check that $N_\varphi (E_i , E_j) = 0$. From all of this, for $n>1$,  the structure is of type $\mathcal{C}_4 \oplus \mathcal{C}_5 \oplus \mathcal{C}_{12}$ and, for $n=1$, it  is of type $\mathcal{C}_5 \oplus \mathcal{C}_{12}$.

Note that $ (n-1)  d \theta 
= (n-1) k_1 
\theta \wedge \eta$, 
 $d\xi_\zeta \eta = k_1 \xi_\zeta \eta \wedge \eta$ and  $d(d^*\eta) = 0$.  From  $d\xi_\zeta \eta = k_1 \xi_\zeta \eta \wedge \eta$,  using Lemma  \ref{dosdeeta}, we deduce  $(\nabla^{\Lie{U}(n)}_\zeta \xi)_\zeta \zeta  =0$. 
  
 Particular cases are:
 
\noindent (i) $k_1 =0$ and $n>1$. The structure is of strict  type $\mathcal{C}_4 \oplus \mathcal{C}_{12}$. The one-forms 
$\theta$ and $\xi_\zeta \eta$ are   closed. In fact, $\xi_\zeta \eta = - d (\ln x_1)$. If we do the conformal change of metric $x_1^2 \langle \cdot , \cdot \rangle$, we obtain the flat cosymplectic structure on $\mathcal{H}$ as an open  set of $\mathbb{R}^{2n+1}$ with the Euclidean metric.  

\noindent (ii) $k_1 \neq 0$, $k_1 \neq c$ and $n=1$. The structure is of strict  type $  \mathcal{C}_{5}  \oplus \mathcal{C}_{12}$.

\noindent (iii) $k_1 =0$ and $n=1$. The structure is of strict  type $ \mathcal{C}_{12}$ and $\xi_\zeta \eta$ is    closed.

\noindent (iv) $k_1 =c$.  The structure is of strict  type $\mathcal{C}_5$ with $d^* \eta = 2nc$. 

Considering again the general case, since $\mathcal{H}$ has constant sectional curvarture, one has that it is ac-Einstein with $s^*=-2n c^2$. Therefore, $\Ricac_{\mathrm{alt}} =0$ and $\Ricac(\zeta)=0$ (note also that $d\theta_{\lcf \lambda^{2,0}\rcf} =0$). For having a harmonic structure, by Corollary \ref{onecondition} (v), we need the condition
$
0 = n \, k_1 \, \xi_\zeta \eta.
$
 Therefore, we have:
 \begin{enumerate}
  \item[-] The structure of type $\mathcal{C}_4 \oplus \mathcal C_5 \oplus \mathcal{C}_{12}$, for $k_1\neq 0$, $k_1\neq c$  and $n>1$,  is not  harmonic. Applying in this case the identity  \eqref{lapstaten}, we obtain $\nabla^* \nabla \zeta  = -(2n-1) k_1 \, \xi_\zeta \zeta + ((2n-1) k_1^2 + c^2) \zeta$. Hence the Reeb vector field $\zeta$ is not harmonic unit vector field.
 \item[-] The structure of type $\mathcal{C}_4 \oplus \mathcal{C}_{12}$, for $k_1=0$ and $n>1$, is harmonic. However, when one tries to check the condition in Theorem \ref{mapharmcontact} (v), one obtains   
 $c^2 \xi_\zeta \eta \neq 2n c^2 \xi_\zeta \eta$. Hence the structure is not  harmonic map.   The Reeb vector field $\zeta$ is harmonic unit vector field. In fact, $\nabla^* \nabla \zeta = - \xi_\zeta \xi_\zeta \zeta = \| \xi_\zeta \zeta\|^2 \zeta = c^2 \zeta$, in according with Proposition \ref{reebharm}.

  \item[-] The structure of type $\mathcal{C}_5 \oplus \mathcal{C}_{12}$,  for  $n =1$, $k_1 \neq 0$ and $k_1 \neq c$, is not  harmonic.  Since $\nabla^*\nabla \zeta =  -  k_1 \, \xi_\zeta \zeta + (k_1^2 + c^2) \zeta$, the Reeb vector field $\zeta$ is not harmonic unit vector field.

  \item[-]   The structure of type $\mathcal{C}_{12}$,  for $k_1=0$ and $n=1$, is harmonic. This was expected because  the one-form $\xi_\zeta \eta$ is closed. Since
  $\sum_{i=1}^{3} \langle R_{e_i,\cdot}, \xi_{e_i} \rangle  = \langle R_{\zeta,\cdot}, \xi_\zeta \rangle = - 2 c^2\xi_\zeta \eta \neq 0
$, the structure is not  harmonic map. The Reeb vector field $\zeta$ is a harmonic unit  vector field,             i.e. $\nabla^* \nabla \xi = - \xi_\zeta \xi_\zeta \zeta = \| \xi_\zeta \zeta \|^2 \zeta = c^2 \zeta$. This is agree with  Proposition \ref{reebharm}.

\item[-]   The structure of type $\mathcal{C}_{5}$,  for $k_1=c$, is harmonic. This was   expected. For the condition in Theorem \ref{mapharmcontact} (iii) (b), we have $(n-1) d s^* =0$ and the right side is equal to $4n(n+4) c^3 \,\eta \neq 0$. Hence the structure is not harmonic  map. Of course, we already know $\zeta$ must  be    harmonic unit vector field. In fact, $\nabla^*\nabla \zeta = - \xi_{e_i} \xi_{e_i} \zeta =  2nc^2 \,\zeta$.
 \end{enumerate}
 
 }
\end{example}

\section{Non-existence of certain types of almost contact metric structures
}
In this  section we point out another application of the identities given in Section \ref{foursection} as tools to  prove the non-existence in a proper 
way of certain types of almost contact metric structures. Results in this direction  have been derived in \cite{FMC6}. Here we do a further and more complete analysis.
\begin{theorem} \label{dimensionsuperior}
For a connected almost contact metric manifold of dimension $2n+1$ with $n>2$, we have:
\begin{enumerate}
\item[$(i)$] If the structure is of type $\mathcal C_1 \oplus \mathcal C_2 \oplus \mathcal C_3 \oplus\mathcal C_5 \oplus\mathcal C_6 \oplus\mathcal C_7 \oplus\mathcal C_9 \oplus\mathcal C_{10} \oplus\mathcal C_{12}$ with $(\xi_{(5)}, \xi_{(6)}) \neq (0,0)$, then it is of type $\mathcal C_1 \oplus \mathcal C_2 \oplus \mathcal C_3 \oplus\mathcal C_5  \oplus\mathcal C_7 \oplus\mathcal C_9 \oplus\mathcal C_{10} \oplus\mathcal C_{12}$ or $\mathcal C_1 \oplus \mathcal C_2 \oplus \mathcal C_3  \oplus\mathcal C_6 \oplus\mathcal C_7 \oplus\mathcal C_9 \oplus\mathcal C_{10} \oplus\mathcal C_{12}$.

\item[$(ii)$]
If the structure is of type $\mathcal C_1 \oplus \mathcal C_2 \oplus \mathcal C_3 \oplus\mathcal C_5 \oplus\mathcal C_6 \oplus\mathcal C_8 \oplus\mathcal C_9 \oplus\mathcal C_{10} \oplus\mathcal C_{12}$ with $(\xi_{(5)}, \xi_{(6)}) \neq (0,0)$, then it is of type $\mathcal C_1 \oplus \mathcal C_2 \oplus \mathcal C_3 \oplus\mathcal C_5  \oplus\mathcal C_8 \oplus\mathcal C_9 \oplus\mathcal C_{10} \oplus\mathcal C_{12}$ or $\mathcal C_1 \oplus \mathcal C_2 \oplus \mathcal C_3  \oplus\mathcal C_6 \oplus\mathcal C_8 \oplus\mathcal C_9 \oplus\mathcal C_{10} \oplus\mathcal C_{12}$.

\item[$(iii)$]
If the structure is of type $\mathcal C_1 \oplus \mathcal C_2 \oplus \mathcal C_3 \oplus\mathcal C_5 \oplus\mathcal C_6 \oplus\mathcal C_7 \oplus\mathcal C_9 \oplus\mathcal C_{11} \oplus\mathcal C_{12}$ with $(\xi_{(5)}, \xi_{(6)}) \neq (0,0)$, then it is of type $\mathcal C_1 \oplus \mathcal C_2 \oplus \mathcal C_3 \oplus\mathcal C_5  \oplus\mathcal C_7 \oplus\mathcal C_9 \oplus\mathcal C_{11} \oplus\mathcal C_{12}$ or $\mathcal C_1 \oplus \mathcal C_2 \oplus \mathcal C_3  \oplus\mathcal C_6 \oplus\mathcal C_7 \oplus\mathcal C_9 \oplus\mathcal C_{11} \oplus\mathcal C_{12}$.
\item[$(iv)$]
If the structure is of type $\mathcal C_1 \oplus \mathcal C_2 \oplus \mathcal C_3 \oplus\mathcal C_5 \oplus\mathcal C_6 \oplus\mathcal C_8 \oplus\mathcal C_9 \oplus\mathcal C_{11} \oplus\mathcal C_{12}$ with $(\xi_{(5)}, \xi_{(6)}) \neq (0,0)$, then it is of type $\mathcal C_1 \oplus \mathcal C_2 \oplus \mathcal C_3 \oplus\mathcal C_5  \oplus\mathcal C_8 \oplus\mathcal C_9 \oplus\mathcal C_{11} \oplus\mathcal C_{12}$ or $\mathcal C_1 \oplus \mathcal C_2 \oplus \mathcal C_3  \oplus\mathcal C_6 \oplus\mathcal C_8 \oplus\mathcal C_9 \oplus\mathcal C_{11} \oplus\mathcal C_{12}$.
\item[$(v)$]
If the structure is of type $\mathcal C_1  \oplus\mathcal C_5 \oplus\mathcal C_7   \oplus\mathcal C_9 \oplus\mathcal C_{10} \oplus\mathcal C_{12}$ with $(\xi_{(5)}, \xi_{(7)}) \neq (0,0)$, then it is of type $\mathcal C_1 \oplus \mathcal C_5  \oplus\mathcal C_9  \oplus\mathcal C_{10}  \oplus\mathcal C_{12}$ or $\mathcal C_1 \oplus \mathcal C_7   \oplus\mathcal C_9 \oplus\mathcal C_{10}  \oplus\mathcal C_{12}$.
\item[$(vi)$]
If the structure is of type $\mathcal C_2  \oplus\mathcal C_5 \oplus\mathcal C_7   \oplus\mathcal C_9 \oplus\mathcal C_{10} \oplus\mathcal C_{12}$ with $(\xi_{(5)}, \xi_{(7)}) \neq (0,0)$, then it is of type $\mathcal C_2 \oplus \mathcal C_5  \oplus\mathcal C_9  \oplus\mathcal C_{10}  \oplus\mathcal C_{12}$ or $\mathcal C_2 \oplus \mathcal C_7   \oplus\mathcal C_9 \oplus\mathcal C_{10}  \oplus\mathcal C_{12}$.
\item[$(vii)$]
If the structure is of type $\mathcal C_1  \oplus\mathcal C_5 \oplus\mathcal C_7   \oplus\mathcal C_9 \oplus\mathcal C_{11} \oplus\mathcal C_{12}$ with $(\xi_{(5)}, \xi_{(7)}) \neq (0,0)$, then it is of type $\mathcal C_1 \oplus \mathcal C_5  \oplus\mathcal C_9  \oplus\mathcal C_{11}  \oplus\mathcal C_{12}$ or $\mathcal C_1 \oplus \mathcal C_7   \oplus\mathcal C_9 \oplus\mathcal C_{11}  \oplus\mathcal C_{12}$.
\item[$(viii)$]
If the structure is of type $\mathcal C_2  \oplus\mathcal C_5 \oplus\mathcal C_7   \oplus\mathcal C_9 \oplus\mathcal C_{11} \oplus\mathcal C_{12}$ with $(\xi_{(5)}, \xi_{(7)}) \neq (0,0)$, then it is of type $\mathcal C_2 \oplus \mathcal C_5  \oplus\mathcal C_9  \oplus\mathcal C_{11}  \oplus\mathcal C_{12}$ or $\mathcal C_2 \oplus \mathcal C_7   \oplus\mathcal C_9 \oplus\mathcal C_{11}  \oplus\mathcal C_{12}$.
\item[$(ix)$]
If the structure is of type $\mathcal C_1  \oplus\mathcal C_6 \oplus\mathcal C_8   \oplus\mathcal C_9 \oplus\mathcal C_{10} \oplus\mathcal C_{12}$ with $(\xi_{(6)}, \xi_{(8)}) \neq (0,0)$, then it is of type $\mathcal C_1 \oplus \mathcal C_6  \oplus\mathcal C_9  \oplus\mathcal C_{10}  \oplus\mathcal C_{12}$ or $\mathcal C_1 \oplus \mathcal C_8   \oplus\mathcal C_9 \oplus\mathcal C_{10}  \oplus\mathcal C_{12}$.
\item[$(x)$]
If the structure is of type $\mathcal C_2  \oplus\mathcal C_6 \oplus\mathcal C_8   \oplus\mathcal C_9 \oplus\mathcal C_{10} \oplus\mathcal C_{12}$ with $(\xi_{(6)}, \xi_{(8)}) \neq (0,0)$, then it is of type $\mathcal C_2 \oplus \mathcal C_6  \oplus\mathcal C_9   \oplus\mathcal C_{12}$ or $\mathcal C_2 \oplus \mathcal C_8   \oplus\mathcal C_9 \oplus\mathcal C_{10}  \oplus\mathcal C_{12}$.
\item[$(xi)$]
If the structure is of type $\mathcal C_1  \oplus\mathcal C_6 \oplus\mathcal C_8   \oplus\mathcal C_9 \oplus\mathcal C_{11} \oplus\mathcal C_{12}$ with $(\xi_{(6)}, \xi_{(8)}) \neq (0,0)$, then it is of type $\mathcal C_1 \oplus \mathcal C_6  \oplus\mathcal C_9  \oplus\mathcal C_{11}  \oplus\mathcal C_{12}$ or $\mathcal C_1 \oplus \mathcal C_8   \oplus\mathcal C_9 \oplus\mathcal C_{11}  \oplus\mathcal C_{12}$.
\item[$(xii)$]
If the structure is of type $\mathcal C_2  \oplus\mathcal C_6 \oplus\mathcal C_8   \oplus\mathcal C_9 \oplus\mathcal C_{11} \oplus\mathcal C_{12}$ with $(\xi_{(6)}, \xi_{(8)}) \neq (0,0)$, then it is of type $\mathcal C_2 \oplus \mathcal C_6  \oplus\mathcal C_9    \oplus\mathcal C_{12}$ or $\mathcal C_2 \oplus \mathcal C_8   \oplus\mathcal C_9 \oplus\mathcal C_{11}  \oplus\mathcal C_{12}$.
\item[$(xiii)$]
If the structure is of type $\mathcal C_2  \oplus \mathcal C_5  \oplus \mathcal C_6 \oplus\mathcal C_8   \oplus\mathcal C_9 \oplus\mathcal C_{11} \oplus\mathcal C_{12}$ with $(\xi_{(6)}, \xi_{(11)}) \neq (0,0)$, then it is of type $\mathcal C_2   \oplus\mathcal C_6 
  \oplus\mathcal C_{9}  \oplus\mathcal C_{12}$ or $\mathcal C_2 \oplus \mathcal C_5 \oplus\mathcal C_{8}   \oplus\mathcal C_9 \oplus\mathcal C_{11}  \oplus\mathcal C_{12}$.

\item[$(xiv)$]
If the structure is of type $\mathcal C_2   \oplus \mathcal C_6 \oplus\mathcal C_9   \oplus\mathcal C_{10} \oplus\mathcal C_{11} \oplus\mathcal C_{12}$ with $(\xi_{(6)}, \xi_{(11)}) \neq (0,0)$, then it is of type $\mathcal C_2 \oplus\mathcal C_6 
 \oplus\mathcal C_{9}   \oplus\mathcal C_{12}$ or $\mathcal C_2 \oplus \mathcal C_9 \oplus\mathcal C_{10}    \oplus\mathcal C_{11}  \oplus\mathcal C_{12}$.

\item[$(xv)$]
If the structure is of type $\mathcal C_2  \oplus \mathcal C_5  \oplus \mathcal C_9 \oplus\mathcal C_{10}   \oplus\mathcal C_{11} \oplus \mathcal C_{12}$ with $(\xi_{(5)}, \xi_{(10)}) \neq (0,0)$, then it is of type $\mathcal C_2 \oplus \mathcal C_5  \oplus\mathcal C_9 \oplus
 \mathcal C_{11}  \oplus\mathcal C_{12}$ or $\mathcal C_2 \oplus \mathcal C_9 \oplus\mathcal C_{10}   \oplus\mathcal C_{11}  \oplus\mathcal C_{12}$.
\item[$(xvi)$]
If the structure is of type $\mathcal C_2  \oplus \mathcal C_5  \oplus \mathcal C_7 \oplus\mathcal C_{10}    \oplus\mathcal C_{12}$ with $(\xi_{(5)}, \xi_{(10)}) \neq (0,0)$, then it is of type $\mathcal C_2 \oplus \mathcal C_5  
   \oplus\mathcal C_{12}$ or $\mathcal C_2 \oplus \mathcal C_7 \oplus\mathcal C_{10}     \oplus\mathcal C_{12}$.
\item[$(xvii)$]
If the structure is of type $\mathcal C_2  \oplus \mathcal C_6  \oplus \mathcal C_9 \oplus\mathcal C_{10}    \oplus\mathcal C_{12}$ with $(\xi_{(6)}, \xi_{(10)}) \neq (0,0)$, then it is of type $\mathcal C_2 \oplus \mathcal C_6  \oplus\mathcal C_9
   \oplus\mathcal C_{12}$ or $\mathcal C_2 \oplus \mathcal C_9 \oplus\mathcal C_{10}     \oplus\mathcal C_{12}$.
\item[$(xviii)$]
If the structure is of type $\mathcal C_1  \oplus \mathcal C_2  \oplus \mathcal C_3 \oplus\mathcal C_{5}    \oplus\mathcal C_{6} \oplus\mathcal C_{12}$ with $(\xi_{(5)}, \xi_{(12)}) \neq (0,0)$ and $dd^* \eta$ is proportional  to $\eta$, then it is of type $\mathcal C_1 \oplus \mathcal C_2  \oplus\mathcal C_3
   \oplus\mathcal C_{5} $ or $\mathcal C_1 \oplus \mathcal C_2 \oplus\mathcal C_{3} \oplus\mathcal C_{6}     \oplus\mathcal C_{12}$.
\item[$(xix)$]
If the structure is of type $\mathcal C_1  \oplus \mathcal C_2  \oplus \mathcal C_5 \oplus\mathcal C_{6}    \oplus\mathcal C_{7} \oplus\mathcal C_{12}$ with $(\xi_{(5)}, \xi_{(12)}) \neq (0,0)$ and $dd^* \eta$  is proportional to $\eta$, then it is of type $\mathcal C_1 \oplus \mathcal C_2  \oplus\mathcal C_5 \oplus\mathcal C_{7}$ or $\mathcal C_1 \oplus \mathcal C_2 \oplus\mathcal C_{6} \oplus\mathcal C_{7}     \oplus\mathcal C_{12}$.
\item[$(xx)$]
If the structure is of type $\mathcal C_1  \oplus \mathcal C_2  \oplus \mathcal C_3 \oplus\mathcal C_{4}    \oplus\mathcal C_{5} \oplus \mathcal C_{6} \oplus\mathcal C_{8}\oplus\mathcal C_{9}\oplus\mathcal C_{11}$ with $(\xi_{(4)}, \xi_{(6)}) \neq (0,0)$ and $d(d^*F(\zeta))$  is proportional to $\eta$, then it is of type $\mathcal C_1 \oplus \mathcal C_2  \oplus\mathcal C_3
   \oplus\mathcal C_{4} \oplus\mathcal C_{5} \oplus\mathcal C_{8} \oplus \mathcal C_{9} \oplus\mathcal C_{11} $ or $\mathcal C_1 \oplus \mathcal C_2 \oplus\mathcal C_{3}  \oplus\mathcal C_{6}     \oplus\mathcal C_{7} \oplus\mathcal C_{9} \oplus\mathcal C_{11}$.
\item[$(xxi)$]
If the structure is of type $\mathcal C_1  \oplus \mathcal C_2  \oplus \mathcal C_3 \oplus\mathcal C_{5}    \oplus\mathcal C_{6} \oplus \mathcal C_{8} \oplus\mathcal C_{9}\oplus\mathcal C_{11}\oplus\mathcal C_{12}$ with $(\xi_{(6)}, \xi_{(12)}) \neq (0,0)$ and $d(d^*F(\zeta))$  is proportional  to $\eta$, then it is of type $ \mathcal C_2  
   \oplus\mathcal C_{6}  \oplus \mathcal C_{9} $ or $\mathcal C_1 \oplus \mathcal C_2 \oplus\mathcal C_{3} \oplus\mathcal C_{5} \oplus\mathcal C_{8}     \oplus\mathcal C_{9} \oplus\mathcal C_{11} \oplus\mathcal C_{12}$.
\item[$(xxii)$]
If the structure is of type $\mathcal C_1  \oplus \mathcal C_2  \oplus \mathcal C_3 \oplus\mathcal C_{5}    \oplus\mathcal C_{6} \oplus \mathcal C_{8} \oplus\mathcal C_{9}\oplus\mathcal C_{11}\oplus\mathcal C_{12}$ with $(\xi_{(5)}, \xi_{(6)}) \neq (0,0)$ and $\langle d \xi_\zeta \eta , F \rangle =0$, then it is of type $ \mathcal C_2  
   \oplus\mathcal C_{6}  \oplus \mathcal C_{9} \oplus\mathcal C_{12} $ or $\mathcal C_1 \oplus \mathcal C_2 \oplus\mathcal C_{3} \oplus\mathcal C_{5} \oplus\mathcal C_{8}     \oplus\mathcal C_{9} \oplus\mathcal C_{11} \oplus\mathcal C_{12}$.
\end{enumerate}
\end{theorem}
\begin{proof}
For $(i)$, $(ii)$, $(iii)$ and $(iv)$,  making use of Proposition \ref{divergenciaunouno}, we obtain $d^*\eta d^*F(\zeta)=0$. Since $(d^*\eta, d^*F(\zeta)) \neq (0,0)$, it follows that set $A$ of those points such that $d^*F(\zeta) =0$ coincides with set of points such that $d^*\eta \neq 0$. Hence $A$ is open. Likewise the set $B$ consisting of points such that $d^*\eta=0$ coincides with the set of points such that  $d^*F(\zeta) \neq 0$. Thus $B$ is open. Since the manifold is connected, $A$ or $B$ must be empty. Therefore,  $d^*F(\zeta) =0$  on the whole manifold, or $d^*\eta =0$  on the whole manifold.

For $(v)$, $(vi)$, $(vii)$ and $(viii)$, also  making use of Proposition \ref{divergenciaunouno}, we obtain  $d^*\eta (\xi_{(7)X} \eta)(Y)=0$. Since $(d^*\eta, \xi_{(7)}) \neq (0,0)$,  the proof is derived by a similar  argument as before. 

For $(ix)$, $(x)$, $(xi)$ and $(xii)$, again   making use of Proposition \ref{divergenciaunouno}, $d^*F(\zeta) (\xi_{(8)X} \eta)(Y)=0$ is obtained. Since $(d^*F(\zeta), \xi_{(8)}) \neq (0,0)$,  the proof is similar  as before. Additionally, part $(xvii)$ is used for  $(x)$  and part $(xiv)$ is used for  $(xii)$.

For  $(xiii)$ and $(xix)$, now we    make use of Proposition \ref{divergenciadoscero} and  obtain $d^*F(\zeta) \langle \xi_{(11) \zeta} X , Y \rangle=0$. Since $(d^*F(\zeta), \xi_{(11)}) \neq (0,0)$,  the proof is similar  as before. Additionally, parts $(ii)$ and   $(x)$ are used for  $(xiii)$  and part $(xvii)$ is used for  $(xiv)$.

For  $(xv)$ and $(xvi)$,  by  Proposition \ref{divergenciadoscero}, we obtain $d^*F(\zeta) \langle \xi_{(11) \zeta} X , Y \rangle=0$. Since $(d^*F(\zeta), \xi_{(11)}) \neq (0,0)$,  the proof is similar  as before. Additionally, part $(vi)$ is used for  $(xvi)$.

For   $(xvii)$,  we also    make use of Proposition \ref{divergenciadoscero} and  obtain $d^*F(\zeta) (\xi_{(10)  X} \eta )  ( Y) =0$. Since $(d^*F(\zeta), \xi_{(10)}) \neq (0,0)$,  the proof is similar  as in previous cases.

For   $(xviii)$ and $(xix)$,  we    make use of Proposition \ref{divergencia} and  obtain $d^*\eta  (\xi_{(12) \zeta } \eta )  ( X) =0$. Since $(d^*\eta, \xi_{(12)}) \neq (0,0)$,  the proof is similar  as in previous cases together with the use of $(i)$.

For    $(xxi)$,  we    make use of Lemma \ref{mainid} and  obtain $d^*F(\zeta)  \theta  ( X) =0$. Since $(d^*F(\zeta), \theta) \neq (0,0)$,  the proof is similar  as in previous cases together with the use of $(i)$.

For    $(xxi)$,  we    make use of Lemma \ref{mainid} and  obtain $d^*F(\zeta)  (\xi_{(12) \zeta } \eta )  ( X) =0$. The proof is similar  as in previous cases using  $(d^*F(\zeta), \xi_{(12)})   \neq (0,0)$, $(i)$  and $(xxii)$.

For    $(xxii)$,  this case is essentially a further conclusion that the one obtained in part $(iv)$ and it was already proved in detail in  \cite{FMC6}. 
\end{proof}

\begin{remark}
{\rm For  connected almost contact metric manifolds of dimensión $2n+1$, $n>2$, by Theorem \ref{dimensionsuperior}, we have:
\begin{enumerate}
\item[-]   By part $(i)$, the non-existence  of structures of types $\mathcal{C}_1 \oplus \mathcal{C}_2 \oplus 
\mathcal{C}_3 \oplus \mathcal{C}_5 \oplus \mathcal{C}_6 \oplus \mathcal{C}_7 \oplus \mathcal{C}_9 \oplus \mathcal{C}_{10} \oplus \mathcal{C}_{12}$ with $\xi_{(5)} \neq 0$ and  $\xi_{(6)} \neq 0$ implies that $2^7$ types do not exist.

\item[-]  By part $(ii)$, the non-existence  of structures of types $\mathcal{C}_1 \oplus \mathcal{C}_2 \oplus \mathcal{C}_3 \oplus \mathcal{C}_5 \oplus \mathcal{C}_6 \oplus \mathcal{C}_8 \oplus \mathcal{C}_9 \oplus \mathcal{C}_{10} \oplus \mathcal{C}_{12}$ with $\xi_{(5)} \neq 0$ and  $\xi_{(6)} \neq 0$ implies that another $2^7 - 2^6 =2^6$ types do not exist. 

\item[-]  Part $(iii)$  implies that another $2^7 - 2^6 =2^6$ types do not exist. 

\item[-]  Part $(iv)$  implies that another $2^5$ types do not exist.

\item[-]  Part $(v)$  implies that another $2^4$ types do not exist.

\item[-]  Part $(vi)$  implies that another $2^3$ types do not exist. 

\item[-]  Part $(vii)$  implies that another $2^3$ types do not exist. 

\item[-]  Part $(viii)$  implies that another $2^2$ types do not exist. 

\item[-]  Part $(ix)$  implies that another $2^4$ types do not exist. 

\item[-]  Part $(x)$  implies that another $2^3+2^3$ types do not exist. 

\item[-]  Part $(xi)$  implies that another $2^3$ types do not exist. 

\item[-]  Part $(xii)$  implies that another $2^3+2^3$ types do not exist. 

\item[-]  The non-existence  of  types due to  part $(xiii)$ has already considered in the previous cases. 

\item[-]  Part $(xiv)$  implies that another $2^3$ types do not exist. They are the ones  such that $\xi_{(6)}\neq 0$, $\xi_{(10)}\neq 0$ and $\xi_{(11)}\neq 0$.

\item[-]  Part $(xv)$  implies that another $2^4$ types do not exist. 

\item[-]  Part $(xvi)$  implies that another $2^2$ types do not exist. 

\item[-]  Part $(xvii)$  implies the non-existence of types already considered  in part $(x)$. 

\item[-]  Part $(xxii)$ in  case  $\xi_{(12)}=0$ implies the non-existence of another $2^2$ types. 

\end{enumerate}
All together implies the non-existence of $412$ types on a connected almost contact metric manifold of dimension  $2n+1$ with $n>2$. Therefore, the possible types for  higher dimensions   are $4096 - 412 = 3684$, where the number $4096=2^{12}$ arises at the beginning  by considering the  possible types from algebraic point of view. Later,  due to  geometry, some of these types can not exist on a connected manifold.
  }
\end{remark}

\begin{theorem} \label{dimensioncinco}
For a connected almost contact metric manifold of dimension $5$, we have:
\begin{enumerate}
\item[$(i)$] If the structure is of type $ \mathcal C_2  \oplus\mathcal C_5 \oplus\mathcal C_6 \oplus\mathcal C_7 \oplus\mathcal C_9 \oplus\mathcal C_{10} \oplus\mathcal C_{12}$ with $(\xi_{(5)}, \xi_{(6)}) \neq (0,0)$, then it is of type $  \mathcal C_2  \oplus\mathcal C_5   \oplus\mathcal C_9 \oplus\mathcal C_{10} \oplus\mathcal C_{12}$ or $\mathcal C_2   \oplus\mathcal C_6 \oplus\mathcal C_7 \oplus\mathcal C_9 \oplus\mathcal C_{10} \oplus\mathcal C_{12}$.

\item[$(ii)$]
If the structure is of type $ \mathcal C_2  \oplus \mathcal C_5 \oplus\mathcal C_6 \oplus\mathcal C_8 \oplus\mathcal C_9 \oplus\mathcal C_{10} \oplus\mathcal C_{12}$ with $(\xi_{(5)}, \xi_{(6)}) \neq (0,0)$, then it is of type $ \mathcal C_2 \oplus\mathcal C_5  \oplus\mathcal C_8 \oplus\mathcal C_9 \oplus\mathcal C_{10} \oplus\mathcal C_{12}$ or $ \mathcal C_2   \oplus\mathcal C_6  \oplus\mathcal C_9 \oplus\mathcal C_{10} \oplus\mathcal C_{12}$.

\item[$(iii)$]
If the structure is of type $ \mathcal C_2 \oplus\mathcal C_5 \oplus\mathcal C_6 \oplus\mathcal C_7 \oplus\mathcal C_9 \oplus\mathcal C_{11} \oplus\mathcal C_{12}$ with $(\xi_{(5)}, \xi_{(6)}) \neq (0,0)$, then it is of type  $\mathcal C_2 \oplus\mathcal C_5  \oplus\mathcal C_9 \oplus\mathcal C_{11} \oplus\mathcal C_{12}$ or $ \mathcal C_2   \oplus\mathcal C_6 \oplus\mathcal C_7 \oplus\mathcal C_9 \oplus\mathcal C_{11} \oplus\mathcal C_{12}$.
\item[$(iv)$]
If the structure is of type $ \mathcal C_2 \oplus\mathcal C_5 \oplus\mathcal C_6 \oplus\mathcal C_8 \oplus\mathcal C_9 \oplus\mathcal C_{11} \oplus\mathcal C_{12}$ with $(\xi_{(5)}, \xi_{(6)}) \neq (0,0)$, then it is of type $ \mathcal C_2  \oplus\mathcal C_5  \oplus\mathcal C_8 \oplus\mathcal C_9 \oplus\mathcal C_{11} \oplus\mathcal C_{12}$ or $ \mathcal C_2  \oplus\mathcal C_6  \oplus\mathcal C_9  \oplus\mathcal C_{12}$.

\item[$(v)$]
If the structure is of type $\mathcal C_2   \oplus \mathcal C_6 \oplus\mathcal C_9   \oplus\mathcal C_{10} \oplus\mathcal C_{11} \oplus\mathcal C_{12}$ with $(\xi_{(6)}, \xi_{(11)}) \neq (0,0)$, then it is of type $\mathcal C_2 \oplus\mathcal C_6 
 \oplus\mathcal C_{9} \oplus\mathcal C_{10}  \oplus\mathcal C_{12}$ or $\mathcal C_2 \oplus \mathcal C_9 \oplus\mathcal C_{10}    \oplus\mathcal C_{11}  \oplus\mathcal C_{12}$.

\item[$(vi)$]
If the structure is of type $ \mathcal C_2  \oplus \mathcal C_5 \oplus\mathcal C_{6}    \oplus\mathcal C_{7} \oplus\mathcal C_{12}$ with $(\xi_{(5)}, \xi_{(12)}) \neq (0,0)$ and $dd^* \eta$  is proportional to $\eta$, then it is of type $ \mathcal C_2  \oplus\mathcal C_5$ or $ \mathcal C_2 \oplus\mathcal C_{6} \oplus\mathcal C_{7}     \oplus\mathcal C_{12}$.

\item[$(vii)$]
If the structure is of type $ \mathcal C_2   \oplus\mathcal C_{4}    \oplus\mathcal C_{5} \oplus \mathcal C_{6} \oplus\mathcal C_{8}\oplus\mathcal C_{9}\oplus\mathcal C_{11}$ with $(\xi_{(4)}, \xi_{(6)}) \neq (0,0)$ and $d(d^*F(\zeta))$  is proportional to $\eta$, then it is of type $ \mathcal C_2 
   \oplus\mathcal C_{4} \oplus\mathcal C_{5} \oplus\mathcal C_{8} \oplus \mathcal C_{9} \oplus\mathcal C_{11} $ or $ \mathcal C_2   \oplus\mathcal C_{6}     \oplus\mathcal C_{9} $.
\item[$(viii)$]
If the structure is of type $ \mathcal C_2  \oplus\mathcal C_{5}    \oplus\mathcal C_{6} \oplus \mathcal C_{8} \oplus\mathcal C_{9}\oplus\mathcal C_{11}\oplus\mathcal C_{12}$ with $(\xi_{(6)}, \xi_{(12)}) \neq (0,0)$ and $d(d^*F(\zeta))$  is proportional to $\eta$, then it is of type $ \mathcal C_2  
   \oplus\mathcal C_{6}  \oplus \mathcal C_{9} $ or $ \mathcal C_2  \oplus\mathcal C_{5} \oplus\mathcal C_{8}     \oplus\mathcal C_{9} \oplus\mathcal C_{11} \oplus\mathcal C_{12}$.
\end{enumerate}
\end{theorem}
\begin{proof}
It follows in a similar way as the proof for Theorem \ref{dimensionsuperior} in the particular case of $n=2$. Note that for $n=2$ some cases of such a theorem do not appear or are redundant.
\end{proof}

\begin{remark}
{\rm On a connected almost contact metric manifold of dimensión $5$ by using Theorem \ref{dimensioncinco}, we have:
\begin{enumerate}
\item[-]   By part $(i)$, the non-existence  of structures of types $\mathcal{C}_2 \oplus \mathcal{C}_5 \oplus \mathcal{C}_6 \oplus \mathcal{C}_7 \oplus \mathcal{C}_9 \oplus \mathcal{C}_{10} \oplus \mathcal{C}_{12}$ with $\xi_{(5)} \neq 0$ and  $\xi_{(6)} \neq 0$ implies that $2^5$ types do not exist. Also,   if $\xi_{(5)}\neq 0$,   then the type is    $\mathcal{C}_2  \oplus \mathcal{C}_5  \oplus \mathcal{C}_9 \oplus \mathcal{C}_{10} \oplus \mathcal{C}_{12}$. This implies that  another $2^4$ types do no exist.

\item[-] By part $(ii)$, the non-existence  of structures of types $\mathcal{C}_2 \oplus \mathcal{C}_5 \oplus \mathcal{C}_6 \oplus \mathcal{C}_8 \oplus \mathcal{C}_9 \oplus \mathcal{C}_{10} \oplus \mathcal{C}_{12}$ with $\xi_{(5)} \neq 0$ and  $\xi_{(6)} \neq 0$ implies that another $2^5 - 2^4 =2^4$ types do not exist. Also,  because if $\xi_{(6)}\neq 0$ then the type is    $\mathcal{C}_2  \oplus \mathcal{C}_6  \oplus \mathcal{C}_9 \oplus \mathcal{C}_{10} \oplus \mathcal{C}_{12}$,  $2^4$ types do not exist. 
\item[-] Part $(iii)$  implies that another $2^5 - 2^4 =2^4$ types do not exist.  Also, in this case,  if $\xi_{(5)}\neq 0$,   then the type is    $\mathcal{C}_2  \oplus \mathcal{C}_5  \oplus \mathcal{C}_9 \oplus \mathcal{C}_{11} \oplus \mathcal{C}_{12}$. This implies that  another $2^3$ types do no exist which have not been considered yet.  
\item[-] Part $(iv)$  implies that another $2^4 - 2^3 =2^3$ types do not exist. If  
   $\xi_{(6)}\neq 0$, then the type is    $\mathcal{C}_2  \oplus \mathcal{C}_6  \oplus \mathcal{C}_9 \oplus \mathcal{C}_{12}$. This implies that  another $3.2^3=24$ types do not exist.
 \item[-] Part $(v)$  implies the non-existence of  $2^3$ types which have not been considered yet.
\end{enumerate}
All together implies the non-existence of $144$ types of connected almost contact metric manifold of dimensión $5$. Therefore, the possible types in dimension $5$  are $1024- 144 = 880$, where the number $1024=2^{8}$ arises by considering the possible types from algebraic point of view.

Note that there is not any restriction for dimension $3$. The possible types in this case are as it is expected from algebraic point of view, i.e.  $2^4=16$.
  }
\end{remark}


\end{document}